\documentclass[11pt]{amsart}
\usepackage[all, import, 2cell]{xy} \UseAllTwocells 
\usepackage{graphicx} 
\usepackage{amsmath}
\usepackage{mathrsfs}
\usepackage{amsfonts}
\usepackage{amsthm} 
\usepackage[usenames]{color} 
\usepackage{hyperref}
\usepackage{tikz}
\newtheorem{theorem}{Theorem}[subsection]
\newtheorem{lemma}[theorem]{Lemma}
\newtheorem{proposition}[theorem]{Proposition}
\newtheorem{corollary}[theorem]{Corollary}
\newtheorem{conjecture}[theorem]{Conjecture}

\theoremstyle{definition}
\newtheorem{definition}[theorem]{Definition}
\newtheorem{notation}[theorem]{Notation}
\newtheorem{example}[theorem]{Example}

\newtheorem{remark}[theorem]{Remark}
\newtheorem{construction}[theorem]{Construction}
\newtheorem{image}[theorem]{Figure}

\newcommand\nc{\newcommand}

\nc{\on}{\operatorname}

\nc{\eqn}{\begin{equation}}
\nc{\eqnd}{\end{equation}}

\nc{\sSet}{\on{sSet}}
\nc{\symp}{\on{Symp}}
\nc{\Sets}{\on{Sets}}
\nc{\SimpCat}{\on{SimpCat}}
\nc{\Cat}{\on{Cat}}
\nc{\Mod}{\on{Mod}}
\nc{\Lag}{\mathsf{Lag}}
\nc{\lag}{\mathsf{Lag}}
\nc{\Open}{Open}
\nc{\StCat}{StCat}
\nc{\Spectra}{\on{Spectra}}
\nc{\Micro}{\on{Micro}}
\nc{\module}{\on{-mod}}
\nc{\perf}{\on{-perf}}
\nc{\Shv}{\on{Shv}}
\nc{\Perf}{\on{Perf}}
\nc\SpLag{\on{SpLag}}
\nc\WFuk{\mathsf{WFuk}}
\nc{\zerolag}{\sL_\emptyset}
\nc{\prism}{\mathcal Pr}
\nc\es{\emptyset}
\nc{\Ob}{\on{Ob}}
\nc{\del}{\partial}
\nc\oo{\infty}
\nc{\hocolim}{\on{hocolim}}
\nc{\colim}{\on{colim}}
\nc{\dist}{\on{dist}}
\nc{\Hom}{\on{Hom}}
\nc{\End}{\on{End}}
\nc{\Sing}{\on{Sing}}
\nc{\id}{\on{id}}
\nc{\tensor}{\otimes}
\nc{\ol}{\overline}
\nc{\totHom}{{\mathcal Hom}}
\nc{\totEnd}{{\mathcal End}}
\nc{\Maps}{\on{Maps}}
\nc{\into}{\hookrightarrow}
\nc{\Vect}{\on{Vect}}
\nc{\Crit}{\on{Crit}}
\nc{\Diff}{\on{Diff}}
\nc{\righttriangle}{\rhd}
\nc{\cone}{\rhd}
\def\Aut{\on{Aut}}
\nc\dd{\diamond}
\nc{\br}{\mathsf{Br}}
\nc{\kernelfunctor}{ {F}}
\nc{\hiro}{\textcolor{blue}}
\nc{\david}{\textcolor{red}}
\nc{\Z}{\mathbb{Z}} \nc{\N}{\mathbb{N}} \nc{\Q}{\mathbb{Q}} \nc{\R}{\mathbb{R}} \nc{\C}{\mathbb{C}} 
\nc{\BN}{\mathbb N} \nc{\BQ}{\mathbb Q}
\nc{\QQ}{\mathbb Q}
\def\CC{\mathbb C}\def\RR{\mathbb R}\def\ZZ{\mathbb Z}
\nc{\cA}{\mathcal A} \nc{\cC}{\mathcal{C}} \nc{\cE}{\mathcal E} \nc{\cN}{\mathcal{N}} \nc{\cP}{\mathcal{P}} \nc{\cQ}{\mathcal{Q}} \nc{\cS}{\mathcal{S}}
\nc{\fC}{\mathfrak{C}} \nc{\fL}{\mathfrak{L}} \nc{\fP}{\mathfrak{P}} \nc{\fX}{\mathfrak{X}}
\nc{\sh}{\mathsf{h}} \nc{\sk}{\mathsf{k}} \nc{\sD}{\mathsf{D}} \nc{\sE}{\mathsf{E}} \nc{\sF}{\mathsf{F}} \nc{\sG}{\mathsf{G}} \nc{\sH}{\mathsf{H}} \nc{\sI}{\mathsf{I}} \nc{\sK}{\mathsf{K}} \nc{\sL}{\mathsf{L}} \def\sN{\mathsf N} \nc{\sP}{\mathsf{P}} \nc{\sQ}{\mathsf{Q}} \def\sR{\mathsf R} \def\sS{{\mathsf S}} \def\sT{\mathsf T} \def\sV{\mathsf{V}} \def\sW{{\mathsf W}} \def\sX{{\mathsf X}} \def\sY{{\mathsf Y}} 

\begin{document}

\title{A stable $\oo$-category of Lagrangian cobordisms}
\author{David Nadler and Hiro Lee Tanaka}

\begin{abstract}
Given an exact symplectic manifold $M$ and a support Lagrangian $\Lambda\subset M$, we construct an $\infty$-category 
$\Lag_\Lambda(M)$ which we
conjecture to be equivalent (after specialization of the coefficients) to the partially wrapped Fukaya category of $M$ relative to $\Lambda$.
Roughly speaking, the objects of $\Lag_\Lambda(M)$ are Lagrangian branes inside of $M\times T^*\R^n$, for large $n$, and the morphisms are Lagrangian cobordisms  that are non-characteristic with respect to $\Lambda$. The main theorem of this paper is that $\Lag_\Lambda(M)$ is a stable $\oo$-category, and in particular its homotopy category is triangulated, with mapping cones given by an elementary construction. The shift functor is equivalent to the familiar shift of grading for Lagrangian branes. 
\end{abstract}


\maketitle
\tableofcontents

\section{Introduction}\label{sect. intro}
This paper introduces a stable  theory of Lagrangian cobordisms in exact symplectic manifolds. 
Fukaya categories~\cite{FOOO, Seidel} provide a geometric (or analytic) approach to
quantum symplectic geometry,
while deformation quantization~\cite{DeWildeLecomte, Fedosov, Kontsevich} provides an algebraic approach. We hope  (see Section~\ref{})
that Lagrangian cobordisms will provide
a missing topological (or homotopical) viewpoint.
By their  nature, Lagrangian cobordisms  involve the intuitive geometry
of Lagrangian submanifolds and the strength of
Hamiltonian isotopies. They also
bypass  unnatural choices of coefficients and the challenging partial differential equations of pseudo-holomorphic disks.
Our inspirations include Floer's original vision~\cite{floer, floer1, floer2}
 of Morse theory of path spaces, the notion of non-characteristic propagation~\cite{KS, GM, N} and
the homotopical setting of stable $\oo$-categories~\cite{Joyal, LurieTopos, LurieStab}.

Given an exact symplectic target $M$ with a collared boundary, equipped with a (not necessarily smooth or compact) support Lagrangian $\Lambda \subset M$,
we construct an $\oo$-category $\Lag_\Lambda(M)$ 
whose objects are Lagrangian branes inside of $M\times T^*\R^n$, for large $n$, and whose morphisms are Lagrangian cobordisms that are non-characteristic with respect to $\Lambda$.
Our main result states that $\Lag_\Lambda(M)$ is a stable $\oo$-category 
(Theorem~\ref{thm. stable})  with concretely realized mapping cones (Theorem~\ref{thm. cones}). 
In the rest of this introduction, we outline the construction of $\Lag_\Lambda(M)$
and the proof of our main result.
We also discuss a sequence of motivating conjectures centered around the following topics: 
\begin{enumerate}
\item Comparing $\Lag_{\Lambda}(M)$ to the partially wrapped Fukaya category of $M$ with respect to $\Lambda$ (see Conjecture~\ref{conj: fukaya}).
\item Comparing $\Lag_X(T^*X)$ to modules over chains on the based loop space of $X$, and more generally,
 for conic $\Lambda\subset T^*X$, relation of $\Lag_{\Lambda}(T^*X)$ to constructible sheaves on $X$ (Conjectures of Section~\ref{sect. cotangent bundles}).
\item Changing geometric structures imposed on $\lag_\Lambda(M)$ to change the coefficient ring spectra of the theory. (Conjectures of Section~\ref{sect. conj point}).
\item Gluing together $\Lag_\Lambda(M)$ from the assignment 
$U \mapsto \Lag_{U \cap \Lambda}(U)$ 
for appropriate open sets $U\subset M$ (Conjecture~\ref{conj: cosheaf}); that is, studying $\lag$ as a cosheaf of $\infty$-categories on $\Lambda$.
\end{enumerate}

\begin{remark}
Since the first draft of the present work, it has been shown that $\lag_\Lambda(M)$ indeed pairs with the wrapped Fukaya category of $M$ for appropriate choices of $\Lambda$~\cite{tanaka-pairing}. In particular, the homotopy groups of the morphism spaces of $\lag_\Lambda(M)$ map to wrapped Floer cohomology groups. 

Moreover, this pairing respects the stability of the present work (Theorem~\ref{thm. stable}), and the usual stability of (modules over) the wrapped Fukaya category~\cite{tanaka-exact}. As a result, the group maps lift to maps of spectra (in the sense of stable homotopy theory). All told, for any pair $L_0,L_1$ of appropriate branes, the spectrum of non-characteristic Lagrangian cobordisms maps to the spectrum of Floer cochains from $L_0$ to $L_1$.

Finally, applying the results above to the object $L_0 = L_1 = \RR^0$ inside $M = \RR^0=\Lambda$, one obtains a map of ring spectra from endomorphisms of a point, to the coefficient ring for Floer theory. (One could take this coefficient ring to be $\ZZ/2\ZZ$, $\ZZ$, $\QQ$, or a periodic version of these, depending on the structures one imposes on branes). These results are consistent with Conjectures~\ref{conj: point endos} and~\ref{conj: fukaya} below. 
\end{remark}

\subsection{Main constructions and definitions} 
It is possible that Lagrangian cobordisms could  be profitably considered in a diverse range of settings. 
There is a long history of the subject~\cite{Arnold80, Audin85, Audin87, ArnoldNovikov, Chekanov, Eli84} with many recent advances~\cite{Oh, Cornea,biran-cornea,biran-cornea-2}. In this paper we decorate every idea in sight with the standard structures needed to define a $\ZZ$-graded Fukaya category over $\ZZ$, as we keep in mind a particular application to the Fukaya category. But the proof of the main theorem requires none of these auxiliary structures, and in particular one can define a version of the category of cobordisms for various structure groups. (See Section~\ref{section.structures}.)

\subsubsection{The symplectic manifold, the objects, and morphisms}
For further details on the assumptions below, we refer the reader to Section~\ref{section.setup}.

Throughout we will assume that $M$ is a symplectic manifold satisfying standard convexity conditions. For example $M$ is exact, and its Liouville flow collars the boundary. We will also fix a trivialization of $\det^2(TM) \cong M \times S^1$, assuming a metric (or almost complex structure) which is compatible with the collaring near $\del M$. 

By Lagrangian {\em branes}, we will mean closed (but not necessarily compact) exact Lagrangian submanifolds $L\subset M$ equipped with three decorations:
\begin{enumerate}
	\item a function $f{\colon\thinspace}L\to \R$ realizing the exactness in the sense that $\theta|_L = df$,
	\item a function $\alpha{\colon\thinspace} L \to \RR$ called the grading, lifting the phase map $L \to S^1$, and
	\item a relative Pin structure $q$.
\end{enumerate}
To control the possibly noncompact nature of $L$, we will always assume taming hypotheses as 
discussed in  Section~\ref{section.setup}. We sometimes abuse notation and denote the data of a brane $(L,f,\alpha,q)$ by $L$ alone.

Finally, our morphisms are Lagrangian cobordisms in the exact symplectic manifold $M\times T^*\R_t$. 
(Here we write $\R_t$ both to denote the real line with coordinate $t$, and to convey that we think of it as the timeline of evolution.) 
Specifically, a Lagrangian cobordism $P$ from $L_0$ to $L_1$ is a Lagrangian brane $P \subset M \times T^*\RR_t$ such that
$$
P|_{(-\infty,t_0]} = L_0 \times (-\infty,t_0]
\qquad \text{and} \qquad
P|_{[t_1,\infty)} = L_1 \times [t_1, \infty)
$$
for some choice of $t_0 \leq t_1 \in \RR_t$. Since $P$ may be non-compact, we will assume 
taming hypotheses as 
discussed in  Section~\ref{section.setup}. We require that the brane structure of $P$ agrees with that of the $L_i$ in the collared regions. For example, $P$ comes equipped with a primitive function $f_P{\colon\thinspace} P \to \RR, df_P = \theta|_P$, such that 
$$
f_P(x_0,t) = f_{L_0}(x_0)
\qquad \text{and} \qquad
f_P(x_1,t) = f_{L_1}(x_1)
$$ 
along the collared regions. This in particular ensures that a composition of exact cobordisms is still exact.

\subsubsection{Space of cobordisms}
To define an $(\infty,1)$-category, one must define a space of morphisms---in our case, a space of cobordisms. One familiar approach is to consider the space where paths between two cobordisms are given by isotopies of cobordisms. We will take a much more flexible approach---a path between two cobordisms is given by a higher cobordism between them, and so on for higher morphisms. One might attempt to organize such data
 into an ``$(\oo, \oo)$-category" (and in the process formulate what such a structure would entail), but our definition naturally yields the structure of an $(\infty,1)$-category. The intuition for this fact is as follows: While there are two natural notions of ``invertible'' morphism in an $(\infty,\infty)$-category, the existence of higher cobordisms implies a dualizability for our morphisms---so choosing a notion of invertibility that amounts to full dualizability, one sees that the invertibility of higher cobordisms is forced upon us, hence higher cobordisms act as higher homotopies. 
 
As it turns out, without restricting the types of cobordism we allow, the space of Lagrangian cobordisms would be contractible due to the non-compact nature of our cobordisms (even equipped with tameness conditions). This is one motivation for requiring all our cobordisms to be {\em non-characteristic}.

\subsubsection{Non-characteristic cobordisms}
An important notion in symplectic geometry is that of  {\em support Lagrangian}. It is the fundamental invariant in quantum algebra and in the microlocal study of sheaves and PDE. (The theory of polarizations 
and their metaplectic symmetries studies the implications of a choice of support Lagrangian. The theory of moduli of stability conditions can also be phrased in terms of moduli of support Lagrangians.)  
We will incorporate the notion of support Lagrangian into our study of Lagrangian cobordisms in the following way. 
  
Fix once and for all a closed subset $\Lambda \subset M$. ($\Lambda$ need not be compact, and need not be smooth; see Remark~\ref{remark: any subset can be Lambda}.)
To control the possibly noncompact nature of $\Lambda$, we will always assume taming hypotheses as discussed in  Section~\ref{section.setup}.

\begin{example} 
\begin{enumerate}
	\item {\em The base case.} One can take $M$ to be the zero-dimensional connected symplectic manifold, i.e., a point, and $\Lambda = M$ to be the unique non-empty Lagrangian.
	\item {\em Conical Lagrangians in cotangent bundles}: take $\Lambda \subset M = T^*X$ to be any closed conical Lagrangian subvariety. In particular, we could take $\Lambda$ to be the zero section alone. 
	\item{\em Conical Lagrangians in Weinstein manifolds}: let $\Lambda^{sk}\subset M$ be the ``Lagrangian skeleton" of stable manifolds associated to the exhausting Morse function $h{\colon\thinspace}M\to \R$ (for details, see~\cite{Biran}).
We can then take $\Lambda$ to be the union of $\Lambda^{sk}$ together with the ``cone" with respect to dilation over any Legendrian subvariety in $\del M = M^\oo$.
\end{enumerate}
\end{example}

We say that a cobordism $P \subset M \times T^*\RR_t$ is {\em non-characteristic} with respect to $\Lambda$ if $P$ avoids $\Lambda$ near $-\infty dt$. Precisely, let 
	\[
	\pi_{\RR^\vee}{\colon\thinspace}
	M \times T^*\RR_t \to T^*_{\{0\}} \RR_t \cong \RR^\vee
	\]
be the projection map to a vertical fiber. (The splitting $T^*\RR_t \cong \RR_t \times \RR^\vee$ is induced by the usual symplectic and complex structures of $T^*\RR_t \cong \CC$.) We demand the existence of a number $T \in \RR^\vee$ such that 
	\[
	\inf_{\substack{x \in P \\ \pi_{\RR^\vee}(x) < T}} \dist_M(\pi_M(x),\Lambda) > 0.
	\]
In other words, if $x$ is a point on $P$ with very negative $\RR^\vee$-coordinate, its projection to the $M$ component is guaranteed to be outside some fixed tubular neighborhood of $\Lambda$.
An analogous definition makes sense for higher cobordisms.

\subsubsection{The $\infty$-category $\lag^0_\Lambda(M)$}
Given a fixed support Lagrangian  $ \Lambda\subset M$, 
we define an $\oo$-category $\Lag^{\dd 0}_\Lambda(M)$ with objects Lagrangian branes $L\subset M$, and $1$-morphisms non-characteristic Lagrangian cobordisms $P\subset M \times T^*\R_t$. See Section~\ref{sect. lag cobs} for a precise version. For further details on $\infty$-categories, we refer the reader to the Appendix.

The following result is the first hint of the role of the support Lagrangian  $\Lambda\subset M$ in our story
(and lends some justification to its name).
Recall that a {\em zero object} in an $\oo$-category is an object which is both initial and terminal, and any two zero objects are canonically equivalent. By the zero category, we will mean the category with a single object $*$ with the single morphism given by the identity of $*$. The following is proven in Section~\ref{sect: non-characteristic}.

\begin{proposition}\label{prop: zero}
The empty Lagrangian $L_\emptyset\subset M$ is a zero object in the $\oo$-category $\Lag^{\dd 0}_{\Lambda}(M)$. 
More generally, let $N_\epsilon(\Lambda)$ be a tubular neighborhood of $\Lambda$.
If $L \subset M$ is a brane satisfying $L \cap N_\epsilon(\Lambda) = \emptyset$, then $L$ is a zero object in $\Lag^{\dd 0}_{\Lambda}(M).$
\end{proposition}
In particular, when the support Lagrangian $\Lambda$ is empty, the $\oo$-category $\Lag^{\dd 0}_\Lambda(M)$ is equivalent to the zero category.
Note this proposition is compatible with the discussion regarding the partially wrapped Fukaya category---see Conjecture~\ref{conj: fukaya}.

\subsubsection{Stabilization}
The following is an $\infty$-categorical enhancement of triangulated categories. In particular, when we mention commutativity, (co)kernels, or other categorical notions, we mean the higher-categorical notions of homotopy commutativity, homotopy (co)kernels, and so on.
\begin{definition}[\cite{LurieHigherAlgebra}]
 An $\oo$-category $\cC$ is said to be {\em stable} if it satisfies the following:
\begin{enumerate}
	\item There exists a zero object $0\in \cC$.
	\item Every morphism in $\cC$ admits a kernel and a cokernel.
	\item Any commutative diagram in $\cC$ of the form
	$$
	\xymatrix{
	a\ar[d] \ar[r] & b \ar[d]\\
	0 \ar[r] & c
	}
	$$
	is Cartesian if and only if it is coCartesian.
\end{enumerate}
\end{definition}
The third condition states that the diagram presents $a$ as a kernel of $b\to c$
if and only if it presents $c$ as a cokernel of $a\to b$. In particular, such diagrams represent exact triangles in the homotopy category of $\cC$, which inherits the structure of a triangulated category. (See Theorem 3.11 of~\cite{LurieStab}.)

\begin{remark}
Any stable $\oo$-category $\cC$ is linear over the sphere spectrum, and hence may be localized with respect to any $E_\oo$-ring spectrum. In particular, we can forget all positive characteristic information and consider the rational stable $\oo$-category $\cC\otimes \BQ$. The theories of rational stable $\oo$-categories,  pre-triangulated rational $A_\oo$-categories, and pre-triangulated rational differential graded categories are all equivalent. (For various notions of enhanced category, see for example the surveys of~\cite{Bergner} and \cite{Keller}.)
\end{remark}

Given a notion of cobordism, it is only natural to consider cobordisms embedded in higher and higher dimensions---i.e., to {\em stabilize} in the sense of topology. For instance, by considering framed cobordisms in $\RR^\infty$, one arrives at the sphere spectrum by the Pontrjagin-Thom construction. The main result of this paper is that the same topological stabilization process, modified to accommodate symplectic geometry, gives rise to the {\em algebraic} notion of stability defined above.

Consider the symplectic manifold $T^*\R^n_x$ with fixed support Lagrangian the zero section 
$\R^n_x \subset T^*\R^n_x$. We equip $T^*\R^n_x$ with canonical exact structure (arising from the fact that it is a cotangent bundle). 
For notational simplicity, we set 
$$\Lag^{\dd n}_{\Lambda}(M) = \Lag^{\dd 0}_{\Lambda \times \R^n}(M\times T^*\R^n),
$$
which for $n=0$, agrees with our previous notation.
The assignment
$$
L\mapsto L \times T^*_{\{0\}} \R_x \subset M \times T^*\R_x,
\qquad
P\mapsto P \times T^*_{\{0\}} \R_x \subset M \times T^*\R_x
$$ 
provides  a  tower of faithful functors
$$
\xymatrix{
\cdots\ar[r] &  \Lag^{\dd n}_{\Lambda}(M)
 \ar[r] & \Lag^{\dd n+1}_{\Lambda}(M )\ar[r] & \cdots
}.
$$

\begin{definition}
Given $ \Lambda\subset M$, 
we define the $\oo$-category $ \Lag_{\Lambda}(M)$ to be the increasing union
$$
\Lag_{\Lambda}(M) = \cup_{n \to \oo}  \Lag^{\dd n}_{\Lambda }(M ).
$$
\end{definition}

\begin{remark}
In the Joyal model structure, fibrant-cofibrant objects are quasi-categories, and cofibrations are level-wise injections. This means $\Lag_\Lambda(M)$ is equivalent to the the homotopy colimit. I.e.,
$$
\Lag_{\Lambda}(M) \simeq \colim_{n \to \oo}  \Lag^{\dd n}_{\Lambda}(M)
$$
in the $\oo$-category of $\oo$-categories.
\end{remark}

The following is our main result.

\begin{theorem} \label{thm. stable}
 The $\oo$-category $\Lag_\Lambda(M)$ is stable.
\end{theorem}

\begin{remark}
While we have used the term {\em stabilization} to refer to the construction of $\Lag_\Lambda(M)$, the same term is used in \cite{LurieHigherAlgebra} as a general procedure for producing stable $\infty$-categories from non-stable ones. We remark that 
$\Lag_\Lambda(M)$ is {\em not} Lurie's universal stabilization of $\Lag^{\dd 0}_\Lambda(M)$---for instance, we will prove that the suspension functor is the usual shift for Lagrangian branes, when branes are endowed with grading structures. Thus the shift functor already exists for each $\lag^{\dd i}$ at the level of objects. One may conjecture that $\lag_\Lambda(M)$ is the universal stabilization with respect to a class of kernel objects already specified, but tackling this question is beyond the scope of this paper.
\end{remark}

\subsubsection{Mapping cones}

One often refers to the cokernel of a morphism  in a stable $\oo$-category as a mapping cone. Part of the value of the above theorem is its proof, in which the mapping cone is concretely constructed.

We briefly summarize here the structure of the mapping cone $C(f)\in  \Lag_\Lambda(M)$ of a morphism $f{\colon\thinspace}L_0 \to L_1$. We will restrict our discussion to objects $L_0, L_1
$ and a morphism $f$ coming from $\Lag^{\dd 0}_\Lambda(M)$. So the objects $L_0, L_1$ are represented by Lagrangian branes in $M$,
and the morphism $f$ is represented by a Lagrangian cobordism $P\subset M\times T^*\R_t$
 living over say $[0,1] \subset \R$.

The mapping cone $C(f)$ will an object of $\Lag^{\dd 1}_{\Lambda}(M) \subset \Lag_\Lambda(M)$. The underlying Lagrangian submanifold is
$$
C(f) = (L_0 \times \Gamma_0) \cup P \cup (L_1 \times \Gamma_1) \subset M \times T^*\R
$$
where $\Gamma_0, \Gamma_1 \subset T^*\R$ are the graphs of closed one-forms $df_0, df_1$
defined
over $(-1, 0], [1, 2)$ respectively. Roughly speaking, $f_0$ is decreasing with $\lim_{x\to -1}f(x) = +\oo, f^{(k)}(0) = 0$,
for all $k\geq 0$,
and  $f_1$ is decreasing with $\lim_{x\to 2} f(x) = -\oo$, $f^{(k)}(1) = 0$ for all $k\geq 0$.

Since $C(f)$ is an extension of $P$ by contractible products of its collars, $C(f)$ admits the structure of a brane.

The following is the main step in the proof of Theorem~\ref{thm. stable}.

\begin{theorem} \label{thm. cones}
The object $C(f)$ is a mapping cone of $f$.
\end{theorem}

It turns out that analyzing the mapping cone for zero morphisms is enough to establish our main theorem. This is due to the following result, found in Theorem 3.11 and Corollary 8.28 of \cite{LurieStab}.

\begin{theorem}[\cite{LurieHigherAlgebra}]
An $\infty$-category $\cC$ is a stable $\infty$-category if and only if the following three properties are satisfied:
\begin{enumerate}
\item There exists a zero object $0\in \cC$.
\item Every morphism in $\cC$ admits a kernel.
\item Let $\Omega{\colon\thinspace} \cC \to \cC$ be the functor induced by sending every object $X$ to the kernel of the zero map $0 \to X$. Then $\Omega$ is an equivalence.
\end{enumerate}
\end{theorem}

In this work, it is the above three properties which we verify.

\subsection{Changing structures}
\label{section.structures}
The geometric set-up we have discussed so far involves all the familiar decoration from Fukaya categories, but our proof of stability only uses the exactness of $M$ and of its branes. Hence, as as in usual cobordism theory, one can choose to affix to $M$ and its branes a wide range of structures---stable framings, stable orientations, no gradings, no Pin structures, et cetera. Given some choice of structure $\cS$, one can define the corresponding category $\lag_\Lambda^\cS(M)$ whose objects are decorated with $\cS$ structures and whose cobordisms respect them.

The reason one can choose arbitrary $\cS$ in the proof of stability is that all our constructions add trivial topologies to our branes; so whatever we mean by structure, it is always respected by our proofs if $\cS$ is sufficiently homotopical in nature. As a particular example, if $\cS$ demands an $N$-fold lift of the Lagrangian phase map $L \to S^1$ (as opposed to the full $\ZZ$-fold lift, $\alpha{\colon\thinspace} L \to \RR$), one obtains an $N$-periodic structure on the shift functor. 

\subsection{The ground case}\label{sect. conj point}
Here we discuss the structure of $\Lag_\Lambda(M)$ when $M$ and $\Lambda$ are simply a point.
Because of our stabilization procedure, this is a highly non-trivial case (for instance, it can be shown to
contain the stable Pontrjagin-Thom theory of ordinary cobordisms).

In a stable $\oo$-category $\cC$, let $b[k]$ denote the $k$th shift of an object $b$. Since the shift functor is the kernel of the zero map, one has the equivalence of spaces
	\[
	\Omega^k \hom_\cC(a,b) \simeq
	\hom_\cC(a,b[k]).
	\]
for any pair $a,b$ of objects.
In other words, between any pair of objects one obtains an $\Omega$-spectrum which we denote by $\totHom_\cC(a,b)$. When $a=b$, composition induces the additional structure  of an $E_1$-algebra in spectra, which we denote by $\totEnd_\cC(a)$.

\begin{definition} Let $\cC$ be a stable $\oo$-category.
\begin{enumerate}
	\item An object $a\in\cC$ is {\em compact} if $\Hom_\cC(a, -)$ preserves coproducts.
	\item An object $a\in \cC$ is a {\em generator} if for any $c\in \cC$, ``$\Hom_\cC(a, c[k]) \simeq 0$ for all $k$'' implies $c \simeq 0$.
\end{enumerate}
\end{definition}

If $a\in \cC$ is a compact generator, then we can recover $\cC$ as the image of the Yoneda functor 
$$
\xymatrix{
\cC\ar[r] & {\totEnd_\cC(a)}\module
&
c\ar@{|->}[r] & \totHom_\cC(a, c)
}
$$
where $ {\totEnd_\cC(a)}\module$ denotes the stable, presentable $\oo$-category of ${\totEnd_\cC(a)}$-modules in spectra.

On the one hand, 
if $\cC$ is a stable, presentable $\oo$-category with a compact generator $a\in \cC$, then
the Yoneda functor is an equivalence $\cC\simeq {\totEnd_\cC(a)}\module$. 
On the other hand, suppose $\cC$ is a stable $\oo$-category with a generator $a\in \cC$ such that every object is a
{finite} colimit of  objects of the form $a[k]$, for $k\in \Z$. Then it is immediate that $a$ is compact,
and the Yoneda functor induces an equivalence 
$
{
\cC\simeq  {\totEnd_\cC(a)}\perf
}
$,
where $ {\totEnd_\cC(a)}\perf$ denotes the stable $\oo$-category of perfect ${\totEnd_\cC(a)}$-modules in spectra. (This is the full subcategory of modules which are finite colimits of free modules).

\begin{conjecture} \label{conj: point generates}
Take $M= \Lambda = pt$, consider
$
\Lag_{pt}(pt),
$
and set $\cE = \totEnd_{\Lag_{pt}(pt)}(pt)$.
\begin{enumerate}
	\item The object $pt\in \Lag_{pt}(pt)$ is a generator.
	\item Every object of $\Lag_{pt}(pt)$ is a finite colimit of objects of the form $pt[k]$, for $k\in \Z$.
\end{enumerate}
Consequently, the Yoneda functor induces an equivalence 
$$
\Lag_{pt}(pt)\simeq \cE\perf.
$$
\end{conjecture}

By the conjecture, the study of $\Lag_{pt}(pt)$ reduces to the calculation of the endomorphism algebra 
$\cE = \totEnd_{\Lag_{pt}(pt)}(pt)$.
By construction, $\cE$ is only an $E_1$-algebra, but in fact it is naturally an $E_\oo$-algebra. The intuition is that direct product induces a functor
$$
\xymatrix{
\Lag_{\Lambda_1}(M_1) \times \Lag_{\Lambda_2}(M_2) \ar[r] & 
\Lag_{\Lambda_1\times \Lambda_2}(M_1 \times M_2).
}
$$
When we take $M_1 = M_2 = \Lambda_1 =\Lambda_2 = pt$, this equips $\Lag_{pt}(pt)$
with a symmetric monoidal structure. (And when we only restrict $M_1 = \Lambda_1 = pt$, this equips  
any $\Lag_{\Lambda_2}(M_2)$ with the structure of $\Lag_{pt}(pt)$-module.) Furthermore, it is apparent that the object $pt \in \Lag_{pt}(pt)$ is the unit for this structure. Thus it follows that its endomorphism algebra $\cE$ is in fact an $E_\oo$-ring spectrum. Making all this precise, however, requires some extra care---we leave this for later work.

\begin{conjecture} \label{conj: point endos}
Take $M= \Lambda = pt$,
 consider
$
\Lag_{pt}(pt).
$
The $E_\oo$-ring spectrum $\cE = \totEnd_{\Lag_{pt}(pt)}(pt)$ 
is connective with 
	\[
	\pi_0(\cE) \cong \ZZ.
	\]
If one considers the structure $\cS$ which does not take into account Pin structures, we expect that $\cE^\cS$ is connective with $\pi_0(\cE^\cS) \cong \ZZ/2\ZZ$. 
\end{conjecture}

\begin{remark}
One import of the above conjecture is that we could then form the base change $\Lag_\Lambda(M) \otimes_{\cE} \Q$
to obtain a rational stable $\oo$-category, or equivalently a rational pre-triangulated differential graded or $A_\oo$-category. 
This then could be compared to well-known geometric differential graded or $A_\oo$-categories. See Conjecture~\ref{conj: fukaya} below. 
\end{remark}

\begin{remark}
We expect the analysis of the $E_\oo$-ring spectrum $\cE$ to follow along the lines of ``modern" Pontrjagin-Thom theory as developed in \cite{GMTW, Galatius, Ayala}.
\end{remark}

\subsection{Cotangent bundles}\label{sect. cotangent bundles}

Here we discuss the structure of $\Lag_\Lambda(M)$ when $M = T^*X$ and $\Lambda \subset T^*X$
is a conical Lagrangian.

Let us first discuss the distinguished case when $\Lambda = X$ is the zero section itself.
The above development for $M=\Lambda = pt$ admits a natural generalization
which should also be amenable to similar techniques.

\begin{conjecture} \label{conj: cotangent bundles}
\footnote{During the publication process of the present work, parts (1) and (2) of Conjecture~\ref{conj: cotangent bundles} were established in~\cite{tanaka-generation} for Weinstein manifolds generally (i.e., not just cotangent bundles).}
Take $M= T^*X$, $\Lambda = X$ the zero section, and consider $\Lag_X(T^*X)$. Assume for simplicity that $X$ is connected,
choose a point $x\in X$, and let $T^*_x X\subset T^*X$ be the corresponding conormal Lagrangian.

(1) The object $T^*_x X\in \Lag_{X}(T^*X)$ is a generator.

(2) Every object of $\Lag_{X}(T^*X)$ is a finite colimit of objects of the form $T^*_x X[k]$, for $k\in \Z$.

(3) Assume $X$ is given a Spin structure. Then the endomorphism algebra $\totEnd_{\Lag_X(T^*X)}(T^*_x X)$ is equivalent to the $E_1$-algebra 
$\Omega_x X_* \wedge \cE$ of $\cE$-chains on the based loop space $\Omega_x X$, 
where $\cE = \totEnd_{\Lag_{pt}(pt)}(pt)$, and $\Omega_x X_* = \Omega_x X \amalg \{*\}$.

Consequently, the Yoneda functor induces an equivalence 
$$
\Lag_{X}(T^*X)\simeq \Omega_x X_* \wedge \cE\perf
$$
\end{conjecture}

\begin{remark}
Suppose $X$ is Spin.
Together with Conjecture~\ref{conj: point endos}, this would imply that the base change $\Lag_{X}(T^*X)\otimes_{\cE}  \Q$ is equivalent to the differential graded category of perfect modules over the differential graded algebra of chains $C_{-*}(\Omega_x X, \BQ)$. 
This would be parallel to work ~\cite{AbbSchw1, AbbSchw2, Abouyoneda, Aboucrit, Abougenerate} on the wrapped Fukaya category. See Conjecture~\ref{conj: fukaya} below.
\end{remark}

Now we would like to generalize Conjecture~\ref{conj: cotangent bundles} when we take a more general 
conical Lagrangian $\Lambda\subset T^*X$.
It will be convenient to adopt the point of view on constructible topology developed in the appendix of~\cite{dag6}. Let us first reformulate Conjecture~\ref{conj: cotangent bundles} in these terms.

Let $\Sing(X)$ denote the Kan simplicial set of singular simplices on $X$. We can regard it as an $\oo$-categorical version of the familiar Poincar\'e groupoid. If $X$ is connected with base point $x\in X$, then $\Sing(X)$ is equivalent to the $\oo$-groupoid with a single object $x$ and endomorphisms $\Omega_x X$. 

\begin{definition} Let $\cA$ be an $E_\oo$-ring spectrum.
\begin{enumerate}
	\item The stable $\oo$-category of {\em constructible sheaves} of $\cA$-modules on $X$ is the continuous functor $\oo$-category
	$$
	\Shv(X, \cA) = \Hom(\Sing(X)^{op}, \cA\module).
	$$
	\item The full stable $\oo$-subcategory $\Perf(X) \subset \Shv(X)$ of {\em perfect constructible sheaves} of $\cA$-modules on $X$ is that generated under finite limits and colimits by representable functors.
\end{enumerate}
\end{definition}

\begin{remark}
A caution about terminology: the above notion of perfect constructible sheaves does not imply any finiteness on the fibers. For example, the pushforward of the constant sheaf along a map $x\to X$ is perfect
though its stalk at $x$ will be $\cA$-cochains on the based loop space $\Omega_x X$. In fact,  perfect constructible sheaves are precisely what one can obtain by finite limits and colimits of such sheaves.
\end{remark}

The following conjecture is equivalent to Conjecture~\ref{conj: cotangent bundles}.

\begin{conjecture} \label{conj: cotangent bundles reformulated}
Take $M= T^*X$, $\Lambda = X$ the zero section, and set 
 $\cE = \totEnd_{\Lag_{pt}(pt)}(pt)$
  
 There is a canonical equivalence of stable $\oo$-categories
 $$
\xymatrix{
 \Perf(X, \cE) \ar[r]^-\sim & \Lag_X(T^*X)
} $$
which sends the constant sheaf on a point $x\in X$ to the conormal Lagrangian $T^*_x X\subset T^*X$.
The above is true for any structure $\cS$.
\end{conjecture}

Now we are prepared to 
generalize Conjecture~\ref{conj: cotangent bundles} when we take conical Lagrangians 
$\Lambda \subset T^*X$ of the following general form. Suppose $X$ is Whitney stratified by strata $X_\alpha\subset X$, for $\alpha \in A$, and take $\Lambda\subset T^*X$
to be the union of the conormals
$$
\Lambda = \coprod_{\alpha\in A} T^*_{X_\alpha} X.
$$

Following~\cite{dag6}, let $\Sing^A(X)$ denote the $\oo$-category of singular simplices adapted to the strata $X_\alpha \subset X$, for $\alpha \in A$. We can regard it as an $\oo$-categorical version of the ``exit path category" of \cite{treumann} with morphisms paths exiting singular strata for less singular strata.

\begin{definition} Let  $X_\alpha \subset X$, for $\alpha \in A$, be a Whitney stratification, and
let $\cA$ be an $E_\oo$-ring spectrum.

(1) The stable $\oo$-category of {\em $A$-constructible sheaves} of $\cE$-modules on $X$  is the continuous functor $\oo$-category
$$
\Shv_A(X, \cA) = \Hom(\Sing^A(X)^{op}, \cA\module).
$$

(2) The full stable $\oo$-subcategory $\Perf_A(X) \subset \Shv_A(X)$ of {\em perfect $A$-constructible sheaves} of $\cE$-modules on $X$ 
is that generated under finite limits and colimits by representable functors.
\end{definition}

\begin{conjecture} \label{conj: stratified}
Take $M= T^*X$, $\Lambda = \coprod_{\alpha\in A} T^*_{X_\alpha} X$ the union of conormals to a Whitney stratification $X_\alpha \subset X$, for $\alpha \in A$.
Then there is a canonical equivalence of stable $\oo$-categories
	$$
	\xymatrix{
	\Perf_A(X, \cE) \ar[r]^-\sim & \Lag_\Lambda(T^*X).
	} 
	$$
\end{conjecture}

\begin{remark}
Together with Conjecture~\ref{conj: point endos}, this would imply that the base change $\Lag_{\Lambda}(T^*X)\otimes_{\cE}  \BQ$ is equivalent to the differential graded category of perfect $A$-constructible sheaves of $\BQ$-vector spaces.
This would be parallel to work ~\cite{NZ, N} on the Fukaya category of the cotangent bundle. See Conjecture~\ref{conj: fukaya} below.
\end{remark}

Finally, before continuing, let us state a refined microlocal version of the above conjecture.
We continue with
$\Lambda = \coprod_{\alpha\in A} T^*_{X_\alpha} X$ the union of conormals to a Whitney stratification $X_\alpha \subset X$, for $\alpha \in A$.

Consider a conical open subset  $U\subset T^*X$.
Then we have two $\oo$-categories which reflect the local geometry of $U$ and are independent
of the ambient cotangent bundle.
On the one hand, we have the stable $\oo$-category $\Lag_{U\cap \Lambda} (U)$.
On the other hand, 
we have the  $\oo$-category $\Micro_\Lambda(U, \cE)$ of perfect microlocal sheaves of $\cE$-modules supported along $U\cap \Lambda$. It can be constructed by quotienting sheaves on $X$ with singular support in $\Lambda\cup (T^*X \setminus U)$ by those with singular support in $T^*X\setminus U$.
(See~\cite{KS, KSmicrolocal} for a comprehensive discussion.)
Then one must impose finiteness conditions akin to the representability in our notion of perfect
constructible sheaf.

\begin{conjecture} \label{conj: microlocal}
Take $M= T^*X$, $\Lambda = \coprod_{\alpha\in A} T^*_{X_\alpha} X$ the union of conormals to a Whitney stratification $X_\alpha \subset X$, for $\alpha \in A$. Let
$U\subset T^*X$ be a conical open subset,
and set 
 $\cE = \totEnd_{\Lag_{pt}(pt)}(pt)$. 
 
 There is a canonical equivalence of stable $\oo$-categories
 $$
\xymatrix{
 \Micro_\Lambda(U, \cE) \ar[r]^-\sim & \Lag_{U\cap \Lambda}(U).
} $$
\end{conjecture}

\subsection{Locality}\label{sect. locality}
Here we discuss the possibility of recovering $\Lag_\Lambda(M)$ by gluing together 
$\Lag_{U\cap \Lambda}(U)$, for small open sets $U\subset M$. Such a local-to-global sheaf property is at the technical core of homological algebraic geometry but unfortunately unavailable in much of
homological symplectic geometry, in particular for the partially wrapped Fukaya category
(with the exception of some tantalizing developments~\cite{FO, Vitfunct, KO2, KO1, Kap, KW, NZ, N, Kontproposal, AbouSeidel, Seidsheaf, Nlocal} among others).
As we understand it, the main obstruction is the presence of instantons (or long distance interactions) in the form of pseudoholomorphic disks, or 
solutions to related elliptic PDEs.
While Lagrangian cobordisms can cover long distances, the idea that they may be assembled from short distance Lagrangian cobordisms is far more elementary than any similar statement for  global elliptic PDEs. For example, we expect to have
the following evident functoriality of which an analogue in Floer theory is a sophisticated question (for example, see
the recent progress~\cite{AbouSeidel}).

By construction, we have the following:
\begin{proposition}
Suppose $U\subset M$ is a conical open subset. Then inclusion induces a functor
$$
\xymatrix{
\Lag_{\Lambda\cap U}(U)\ar[r] & \Lag_\Lambda(M).
}
$$
\end{proposition}

Let $\Open^+(M)$ be the category of conical open subsets $U\subset M$, with morphisms given by  inclusions. Observe
that the intersection of two such open sets is again such an open set.

By a cover $\coprod_{\alpha\in I} U_\alpha \to M$, we mean a finite collection of 
 open sets $U_\alpha \subset M$ such that $M$
is the colimit of the corresponding \v{C}ech simplicial space with $k$-simplices
$
\coprod_{\underline \alpha\in I^k} U_{\underline\alpha},
$
where $U_{\underline\alpha} = U_{\alpha_1} \cap \cdots U_{\alpha_k}$, for $\underline\alpha = (\alpha_1, \ldots, \alpha_k)$.

\begin{conjecture} \label{conj: cosheaf}
Let $\StCat$ denote the $\oo$-category of stable $\oo$-categories.
Then the functor
$$
\xymatrix{
\Open^+(M) \ar[r] & \StCat & U \ar@{|->}[r] &  \Lag_{U\cap \Lambda}(U)
}
$$ 
forms a cosheaf. In particular, $\Lag_{ \Lambda}(M)$ can be calculated locally: for 
$\coprod_{\alpha\in I} U_\alpha \to M$ a cover by objects of $\Open^+(M)$,
$\Lag_{\Lambda}(M)$ is the colimit of the corresponding \v{C}ech simplicial stable $\oo$-category with $k$-simplices
$$
\xymatrix{
\coprod_{\underline \alpha\in I^k} \Lag_{U_{\underline\alpha}\cap \Lambda}(U_{\underline\alpha}),
}
$$
where $U_{\underline\alpha} = U_{\alpha_1} \cap \cdots U_{\alpha_k}$, for $\underline\alpha = (\alpha_1, \ldots, \alpha_k)$.
\end{conjecture}

\subsection{Partially wrapped Fukaya categories}\label{sect. relation to fukaya}

Finally, we arrive at one of our primary motivations for considering Lagrangian cobordisms.

We continue with $M$ an exact symplectic manifold with support Lagrangian $\Lambda \subset M$.
We will proceed informally and write $WF_\Lambda(M)$ for the corresponding partially wrapped Fukaya category of $M$
(with say rational coefficients).
Its construction and basic properties are well-known to experts and explained in~\cite{Aur, AurICM}
building on the more well-known wrapped Fukaya category~\cite{FSSsurvey, AbouSeidel}.
It is a fundamental object in mirror symmetry and geometric representation theory but still
mysterious and technically challenging.

Returning to the starting point of Floer theory as Morse theory of path spaces, it appears quite reasonable that Lagrangian cobordisms should play a role in understanding wrapped Floer theory. Roughly speaking, we expect an analogy along the lines:
$$
\begin{array}{c}
\mbox{Witten's Morse complex: singular cochains ::}  \\
\mbox{ Floer's instanton complex: families of Lagrangian cobordisms}
\end{array}
$$
Independently of such vague ideas, if we fix a base ring $R$, then by equipping all branes and Lagrangian cobordisms with appropriate tangential structures, one can write a pairing
	\begin{equation}\label{eqn.pairing}
	\xymatrix{
	\Lag_\Lambda(M) \times F_\Lambda(M) \ar[r]
	& \on{Chain}_R
	}
	\end{equation}
with the Fukaya category, taking values in the category of chain complexes over $R$. (We refer to~\cite{tanaka-pairing, tanaka-exact} for details. A compact, monotone variant of this was written by Biran and Cornea in~\cite{biran-cornea,biran-cornea-2}.)
In fact, the idea that there should be such a functor is the origin of our study of non-characteristic Lagrangian cobordisms. We expect to be able to exhibit the pairing~\eqref{eqn.pairing} as arising from an honest functor to the wrapped Fukaya category: 

\begin{conjecture} \label{conj: fukaya}
Let $M$ be an exact symplectic manifold with Lagrangian support $\Lambda\subset M$. 
The functor given by continuation maps induces an equivalence of stable $\oo$-categories
$$
\xymatrix{
   \Lag_\Lambda(M) \otimes_{\cE} \BQ  \ar[r]^-\sim & WF_\Lambda(M)
}$$
\end{conjecture}

\begin{remark}
Our Conjectures~~\ref{conj: point endos},~\ref{conj: cotangent bundles}, ~\ref{conj: stratified}, and~\ref{conj: fukaya} are compatible thanks to
the work~\cite{AbbSchw1, AbbSchw2, NZ, N, Abouyoneda, Aboucrit, Abougenerate}.
\end{remark}

\subsection{Relation to the work of Biran and Cornea}
During the preparation of this work, Paul Biran and Octav Cornea independently constructed a category of Lagrangians and cobordisms between them~\cite{biran-cornea, biran-cornea-2}. We note that there are several notable differences making our collective works of independent interest. 

First, Biran and Cornea work in a compact, monotone setting while the present work lives in the non-compact, exact setting. The compactness of Biran and Cornea's branes also accounts for the absence of $\Lambda$ in their work, which only appears here when regulating non-compact cobordisms.

Second, $\lag_\Lambda(M)$ detects the space of Lagrangian cobordisms and likewise detects higher homotopical data of the Fukaya category.  In contrast, the work of Biran and Cornea does not define any space of cobordisms, and instead detects formal sequences of mapping cones, essentially by constructing maps between non-derived version of Waldhausen's s-dot construction. But the constructions here are most likely compatible with the constructions from~\cite{biran-cornea}. For example, there is a functor from $\lag_\Lambda(M)$ to modules over the Fukaya category of $M$~\cite{tanaka-pairing}, and this functor is exact~\cite{tanaka-exact}---it is expected that the induced map on the s-dot constructions of each category reproduces the map from~\cite{biran-cornea-2}.

Third, the setting of this work is manifestly stable---both in the sense of topology and of category theory. Higher cobordisms, for instance, detect higher natural transformations between modules over the Fukaya category, and Lagrangians embedded in $M \times T^*\RR^N$ for large $N$ give geometric constructions of mapping cones. The work in~\cite{biran-cornea} does not stabilize---geometrically nor algebraically---so this higher geometry goes undetected. As an illustration of this difference: the techniques used in~\cite{tanaka-pairing} allow one to exhibit a map from the higher homotopy groups of $Ham$ to Hochschild cohomology groups of $\lag_\Lambda(M)$. In contrast,~\cite{charette-cornea} can only detect such a map at the 1-groupoid level, essentially only seeing $\pi_1$ of $Ham$.

\subsection{A word on exactness}

In the present work, we have chosen to work with exact symplectic manifolds and exact Lagrangians. Let us give a few motivations for this. First, our morphisms and our stabilizations naturally live in non-compact symplectic manifolds, and exactness imposes natural constraints on behaviors ``near infinity.'' A second reason is not internal to the theory of Lagrangian cobordisms: In the exact setting, it is technically simpler to compare the theory of non-compact Lagrangian cobordisms to Floer theory. 

A third reason is more categorical. Exactness arises in the study of symplectic geometry---without reference to bubbling of holomorphic disks---by looking for fixed points of a natural $\RR$-action on the category of symplectic manifolds (see the following paragraphs). Moreover, at least in Floer theory, these fixed points allow us to recover information about non-fixed points using sane amounts of data. For example, many symplectic manifolds $M$ have Fukaya categories arising as deformations of Fukaya categories of exact symplectic manifolds $M \setminus D$ obtained by removing an appropriate divisor $D \subset M$. And in many exact symplectic manifolds, arbitrary branes may be written as extensions built from exact branes. (We do not know if Lagrangian cobordism theory in the non-exact setting admits such deformation-theoretic descriptions.)

Let $\symp$ be the category whose objects are symplectic manifolds, and whose morphisms are Lagrangian correspondences. (That is, a morphism from $M$ to $M'$ is a Lagrangian in $(M \times M', \omega \oplus -\omega')$.) As usual, one must address transversality issues arising in composition, and whether or not to allow immersed Lagrangians as morphisms (as opposed to formal sequences of correspondences), but let us ignore this point for this heuristic discussion. We will simply say that, following Weinstein, some version of $\symp$ is expected to be a natural domain for expressing the functoriality of symplectic invariants.

There is an action of the discrete group $(\RR,+)$ on this category: An element $t \in \RR$ sends an object $(M,\omega)$ to the object $(M,e^t \omega)$. A morphism $L \subset M \times M'$ is sent to the {\em same} Lagrangian $L$. (A Lagrangian for $\omega \oplus -\omega'$ is Lagrangian for $e^t \omega \oplus -e^t \omega$.)

Natural examples of homotopy fixed points of this $\RR$-action arise as exact symplectic manifolds that are complete with respect to their Liouville flows. Indeed, a choice of Liouville vector field $X_\theta$ induces a flow on $M$, and the graph of the time $t$ flow is an invertible Lagrangian correspondence between $(M,\omega)$ and $(M,e^t\omega)$. That the composition of these correspondences respects the group operation of $\RR$ is straightforward. So, for example, Liouville manifolds are examples of homotopy fixed points of the $\RR$-action.

Let us choose a complete Liouville vector field on a symplectic manifold $M$. This homotopy fixed point data induces an action of $\RR$ on the collection of Lagrangians of $M$ as follows: $t \in \RR$ sends a Lagrangian in $M$ to the time $-t$ flow of that Lagrangian under the Liouville flow. Indeed, a Lagrangian in $M$ is a correspondence $L$ from the point $\ast$ to $(M,\omega)$. The action of $t \in \RR$ on $\symp$ sends this Lagrangian to the same subset $L$ of $\ast \times (M,e^t\omega)$. The homotopy fixed point data (the Liouville flow for time $-t$) composes with $L$ to yield the time $-t$ flow of $L$ as a new Lagrangian in $(M,\omega)$.

What does it mean to exhibit a homotopy fixed point structure on a Lagrangian?

To exhibit fixed point data on $L$, one must know what an equivalence of two Lagrangian correspondences is. At the very least, a Hamiltonian isotopy of Lagrangians should be such an thing. Moreover, a function $f: L \to \RR$ satisfying $df=\theta|_L$ precisely exhibits a Hamiltonian vector field matching the Liouville flow of $L$ as a subset of $M$. (Indeed, for any extension $\tilde f: M \to \RR$ of $f$, we have that $\omega(X_{\tilde f},-)$ and $-\omega(X_\theta,-)$ agree when restricted to $L$. This implies that the components of $X_{\tilde f}$ and $X_\theta$ normal to $L$ agree at all times.)

The above comprises an informal argument outlining why exact symplectic manifolds, and exact Lagrangians, arise naturally in symplectic geometry (without any consideration of holomorphic disks) as natural examples of homotopy fixed points.

\subsection{Outline of the paper}

We first take care of the geometric preliminaries and introduce Lagrangian branes in Section~\ref{section.setup}.  For any technical assumptions, the reader can consult this section.

In Section~\ref{sect. lag cobs}, we introduce the notions of {Lagrangian cobordism} and {higher Lagrangian cobordism}. In particular we explain the role of the reference Lagrangian $\Lambda \subset M$, defining the notion of {\em non-characteristic} Lagrangians. We also define the collaring notions needed to define the $\infty$-category $\Lag^{\dd 0}_\Lambda(M)$.

In Section~\ref{sect. lag}, we define the $\infty$-category $\Lag^{\dd 0}_\Lambda(M)$ organizing Lagrangian cobordisms, and prove it is a weak Kan complex with a zero object. We also explain the notion of compositions via time concatenation. We then construct the {\em stabilization} $\Lag_\Lambda(M)$.

We then construct a kernel $\sK$ for any morphism $\sP{\colon\thinspace} \sL_0 \to \sL_1$. The construction is carried out in Section~\ref{sect. kernel}, and the proof of its universal property is given in Section~\ref{sect: universal property}.

Finally, in Section~\ref{sect. loop}, we study the shift functor $\Omega{\colon\thinspace} \Lag_\Lambda(M) \to \Lag_\Lambda(M)$ to prove that $\Lag_\Lambda(M)$ is indeed stable.

In the appendix, we collect useful background material on $\infty$-categories.

\subsection{Acknowledgements}

We would like to thank Denis Auroux, Kevin Costello and Paul Seidel for their interest, and Paul Biran and Octav Cornea for fruitful correspondences and a productive week at the IAS.
We would also like to thank the anonymous referees for their thoughtful comments and suggestions.

The first author was supported by NSF grant DMS-0901114 and a Sloan Research Fellowship.
The second author was supported by an NSF
Graduate Research Fellowship and by Northwestern University's Presidential Fellowship.

\clearpage
\section{Geometric setup}\label{section.setup}

\subsection{Symplectic manifold}\label{section.symplectic-manifold}
We will restrict to the geometric background of convex symplectic manifolds, which we learned from the well-known references~\cite{EliashGromov, Eliash}.

Throughout this paper, $M$ will be an exact symplectic manifold with symplectic form $\omega_M$, fixed Liouville form $\theta_M$ so that $\omega_M = d\theta_M$, and Liouville vector field $Z_M$ defined by $\omega_M(Z_M, -) = \theta_M(-)$. We assume that this flow collars the boundary of $M$---this means the flow is transverse to the ``boundary at infinity'' $\del M$, and that the Liouville flow defines a diffeomorphism
	\[
	(0,1] \times \del M \cong U
	\]
on some open neighborhood of $\del M$. For the purposes of having a notion of distance, we will assume a Riemannian metric (or a compatible almost complex structure) which respects the collaring on $U$.

\begin{example}[Cotangent bundles]
Fix a smooth (not necessarily compact) base manifold $X$, and let $M = T^*X$ be its cotangent bundle equipped with $\theta=-pdq$ the canonical one-form. The Liouville vector field $Z_M$ generates fiberwise dilation.
More generally, we could take a conical open subset $U\subset M = T^*X$.
\end{example}
\begin{example}[Weinstein manifolds] 
Suppose  $M$ admits a proper, 
bounded-below Morse function $h{\colon\thinspace}M\to \R$ such that $Z_M$ is gradient-like with respect to $h$. Such data naturally arises
when $M\subset \C^n$ is a Stein or smooth affine complex variety: one can take $h$ to be a perturbation of the restriction of the standard unitary norm on  $\C^n$, and $\theta_M$ to correspond to the gradient of $h$ via the symplectic form. More generally, we
could consider a conical open subset $U\subset M$.
\end{example}

\begin{remark}
It is a useful level of generality to assume that $Z_M$ is complete, so the flow generated by $Z_M$ provides an $\R_{>0}$-action on $M$. It will also be useful to consider structures that are {\em conical} (i.e., $\R_{>0}$-invariant) when relating $\lag_\Lambda(M)$ to the Fukaya category~\cite{tanaka-pairing,tanaka-exact}, but we will not need such assumptions here.
\end{remark}

We will write $M^\oo = \ol M \setminus M$ for the boundary of $\ol M$, and given a subset $A\subset M$, we will write $\ol A$ for the closure of $A$ inside of $\ol M$, and $A^\oo = \ol A\setminus A = \ol A \cap M^\oo$ for the part of $\ol A$ lying in the boundary $M^\oo$.

\begin{remark}
It will be useful from a technical perspective to assume that $\ol M$ embeds as a subanalytic subset of some real analytic manifold $N$ (which without loss of generality we could take to be $\R^n$ for some large $n$). We will then only consider subsets of $M$ whose closures in $N$ are themselves subanalytic. This obviates the possibility of any ``wild" behavior and keeps us in the world of ``tame" topology~\cite{BM, vdDM}.
\end{remark}

Finally, in order to work with Lagrangian branes inside of $M$, one must consider the following standard background structures.
Given a compatible almost complex structure  $J\in\on{End}(TM)$, we can speak about the complex canonical line bundle 
$\kappa = (\wedge^{\dim M/2} T^{hol} M)^{-1}$. A bicanonical trivialization $\eta$ is an identification of the tensor-square
$\kappa^{\otimes 2}$ with the trivial complex line bundle. The obstruction to a bicanonical trivialization is twice the first Chern class $2c_1(M)\in H^2(M, \Z)$, and all bicanonical trivializations form a torsor over the group  $\on{Map}(M, S^1)$.
Note that there is a bundle over the space of all compatible $J$ whose fibers are given by bicanonical trivializations; we assume we have fixed a section.

One must also fix a background class $[c] \in H^2(M , \Z/2\Z)$
so as to be able to speak about pin structures relative to $[c]$.

\subsection{Support Lagrangian}\label{ssect: reference lagrangians}
 Throughout this paper, $\Lambda \subset M$ will be a (not necessarily compact, possibly singular) closed Lagrangian subvariety. To be more precise, we suppose the closure $\ol \Lambda\subset \ol M$ is subanalytic, and hence admits a Whitney stratification into smooth submanifolds. Then we insist that $\Lambda$ (though not necessarily irreducible) is pure of dimension $\dim M/2$,
 and each stratum of $\Lambda$ is isotropic with respect to the symplectic form $\omega_M$.

\begin{example}[Conical Lagrangians in cotangent bundles]
Take $\Lambda \subset M = T^*X$ to be any closed conical Lagrangian subvariety. In particular, we could take $\Lambda$ to be the zero section alone. More generally, if $U\subset M = T^*X$
is a conical open subset, we can take the restriction $\Lambda \cap U \subset U$ as support Lagrangian
inside of $U$.
\end{example}
\begin{example}[Conical Lagrangians in Weinstein manifolds]
Let $\Lambda^{sk}\subset M$ be the ``Lagrangian skeleton" of stable manifolds associated to the exhausting Morse function $h{\colon\thinspace}M\to \R$ (for details, see~\cite{Biran}).
We can then take $\Lambda$ to be the union of $\Lambda^{sk}$ together with the ``cone" with respect to dilation over any Legendrian subvariety in $M^\oo$.
More generally, if $U\subset M$
is a conical open subset, we can take the restriction $\Lambda \cap U \subset U$ as support Lagrangian
inside of $U$.
\end{example}

\begin{remark}\label{remark: any subset can be Lambda}
Although we have not discussed any non-conical examples, there is nothing to exclude such a choice. Likewise, the proof of stability does not rely on the dimension, or the Lagrangian-ness, of $\Lambda$. Any subset $\Lambda \subset M$ will do, including the case $\Lambda = M$.
\end{remark} 

\begin{remark}
In this work, $\Lambda$ only plays the role of restricting the types of cobordisms we allow.
\end{remark}

\subsection{Lagrangian branes}\label{ssect: lagrangian branes}
A Lagrangian brane inside of $M$ is an exact Lagrangian submanifold $L\subset M$ equipped with the following:
\begin{enumerate}
\item
 $f{\colon\thinspace}L\to \R$ is a function realizing the exactness of $L$ with respect to the given primitive $\theta_M$;
 \item 
$\alpha{\colon\thinspace}L\to \R$ is a grading with respect to the given trivialization of the complex bicanonical line bundle $\kappa^{\otimes 2}$;
\item
$q$ is a relative pin structure with respect to the given background class 
$[c]\in H^2(M, \Z/2\Z)$. 
\end{enumerate}
Each structure serves a purpose---$f$ allows for composition of exact cobordisms, $\alpha$ allows for a $\ZZ$-periodic shift functor (or, in the world of Fukaya categories, a $\ZZ$-graded Floer cohomology) and $q$ allows for one to orient moduli spaces of holomorphic polygons in Floer theory. We remind the reader that only the function $f$ is a necessary decoration for the proof of the main theorem, as Floer theory plays no role in this paper. (See Section~\ref{section.structures}.)

\begin{remark}
Note that
$f$ is determined
up to translation by an element of $H^0(L, \R)$;
$\alpha$ is determined up to translation by an element of $H^0(L, \Z)$;
and $q$ is determined up to translation by an element of $H^1(L, \Z/2\Z)$.
\end{remark}

\begin{remark}
Finally, one can also choose to consider Lagrangians with varying properties to define $\lag_\Lambda(M)$. One could require objects to have a compact projection to $\Lambda$, to be conical near the boundary of $M$, or only be asymptotically conical near $\del M$, and consistently define an $\infty$-category consisting of objects and (higher) morphisms with these properties. 
\end{remark}

\subsection{Product structures}\label{ssect: products} All of the structures we have introduced
in the preceding sections behave naturally with respect to Cartesian products. For $i=1,2$, given exact symplectic manifolds $M_i$, $\omega_i$, $\theta_i$, we have the exact symplectic manifold
$M_1 \times M_2$, $\omega_1 + \omega_2$, $\theta_1 + \theta_2$. Likewise, given exact Lagrangians $L_1, L_2$, one defines  a primitive on their product by $f(x_1,x_2) = f_1(x_1) + f_2(x_2)$. 

Likewise, given almost complex structures $J_i$, bicanonical trivializations $\tau_i$, and
background classes  $[c_i] \in H^2(M_i , \Z/2\Z)$, we take the 
almost complex structure $J_1 \oplus J_2$, bicanonical trivialization $\tau_1\wedge\tau_2 $, and
background class  $[c_1] + [c_2] \in H^2(M_1\times M_2 , \Z/2\Z)$.

Similarly, for $i=1,2$, given Lagrangian branes  $L_i \subset M_i$ with brane structures $(f_i , \alpha_i, q_i)$, we have the Lagrangian brane $L_1 \times L_2 \subset M_1 \times M_2$
with brane structure
$(f_1 + f_2, \alpha_1 + \alpha_2, q_1 + q_2)$.

\subsection{Euclidean space}\label{ssect: cotangent bundles}

We will often want to consider the symplectic manifold $T^*\R^n$, and so spell out here what structures 
we equip it with to fit into our setup.

Let $\pi{\colon\thinspace}T^*\R^n \to \R^n$ denote the canonical projection.
We have a canonical identification $T^*\R^n \simeq \R^n \times (\R^n)^\vee$
where $(\R^n)^\vee$ denotes the linear dual of $\R^n$.
We write $x_j$ for the standard coordinates on $\R^n$, and $\xi_j$ for the dual coordinates on $(\R^n)^\vee$. 

We work with the standard symplectic form 
on $T^*\R^n$ given in coordinates by $\omega = \sum_j d\xi_j dx_j$ with standard primitive
$
\theta = \sum_j \xi_i dx_j.
$
The Liouville vector field
$
Z =   \sum_j \xi_j \partial_{\xi_j} 
$
generates the standard fiberwise dilation.
We will usually take the zero section $\R^n\subset T^*\R^n$ as our support Lagrangian $\Lambda$.

To enrich $T^*\RR^n$ with the necessary background structures, we take the standard complex structure $\C^n \simeq T^*\R^n$ with holomorphic coordinates $z_j = \xi_j + ix_j$. We trivialize the bicanonical bundle $\kappa^{\otimes 2}$ by taking the real line subbundle of the canonical bundle $\kappa$ induced by the zero section $\R^n \subset T^*\R^n$.
In particular, we can identify possible gradings on the zero section with the integers $\Z$ (rather than some arbitrary $\Z$-torsor). There is no choice of background class
 $[c] \in H^2(T^*\R^n , \Z/2\Z)$ since $T^*\R^n$ is contractible.

Here are some of our favorite branes inside of $T^*\R^n$.

\begin{example}[Conormals]
Recall that the conormal $T^*_X Y$ to a smooth $X \subset Y$ is the set of all $(x,\xi)\in T^*Y$ for which $\xi|_{T_x X} = 0$.
The $\alpha$-graded {\em conormal brane} $\sN_{x}[\alpha] $ based at $x\in \R^n$ is the Lagrangian submanifold 
$$
T^*_{\{x\}} \R^n \subset T^*\R^n
$$ 
for some $x\in \R^n$, equipped with the brane structure
given by the primitive $f= 0$, grading $\alpha \in \Z + n/2$, and trivial relative pin structure.

More generally, 
the $\alpha$-graded {\em conormal brane} $\sN_X[\alpha] $ along a $k$-dimensional submanifold $X\subset \R^n$ is the Lagrangian submanifold 
$$
T^*_X \R^n \subset T^*\R^n
$$ 
equipped with the brane structure
given by the primitive $f= 0$, grading $\alpha \in \Z + (n-k)/2$, and trivial relative pin structure.

\end{example}

\begin{example}[Graphs]
Given a function $h{\colon\thinspace}\R^n\to \R$, the {\em graph} of its differential is the Lagrangian brane
$$
\Gamma_h \subset T^*\R^n
$$ 
with brane structure $(f_h, \alpha_h, [0])$,
where $f_h$ is
 the pullback of $h$, $\alpha_h$ is transported from the zero grading on the zero section via an isotopy through graphs, and 
 $[0]$ is the trivial relative pin structure.
\end{example}

\subsection{Constructions of branes}
Here we give various constructions of simple Lagrangian branes.

\subsubsection{Translation and carving}\label{section.carving}
Suppose
$h{\colon\thinspace}\R^n\to \R$ is a function with graph brane
$
\Gamma_h \subset T^*\R^n.
$ 
We can regard $\pi{\colon\thinspace}T^*\R^n \to \R^n$ as a relative group, and consider the fiberwise translation given by adding $\Gamma_h$.

Given another Lagrangian brane $\sL\subset T^*\R^n$, with brane structure $(f_\sL, \alpha_\sL, [c_\sL])$,
we can form the {\em translated} Lagrangian brane 
$$
\sL_h = \sL + \Gamma_h= \{ (x, \xi) \in T^*\R^n \, | \, (x, \xi - dh_x ) \in \sL\}
$$
 with brane structure $(f_{\sL, h}, \alpha_{\sL, h}, [c_\sL])$, where $f_{\sL, h}$ is the sum of $f_\sL$ and the pullback of $h$, and
 $\alpha_{\sL, h}$ is transported from the grading $\alpha_\sL$ via an isotopy through graphs.

More generally, given an exact symplectic manifold $M$, and a Lagrangian brane $\sL\subset M \times T^*\R^n$ with brane structure $(f_\sL, \alpha_\sL, [c_\sL])$,
we can form the translated Lagrangian brane 
$$
\sL_h = \sL + \Gamma_h= \{ (m, x, \xi) \in T^*\R^n \, | \, (m, x, \xi - dh_x ) \in \sL\}
$$
 with brane structure $(f_{\sL, h}, \alpha_{\sL, h}, [c_\sL])$,
 where $f_{\sL, h}$ is the sum of $f_\sL$ and the pullback of $h$, and
 $\alpha_{\sL, h}$ is transported from the grading $\alpha_\sL$ via an isotopy through graphs.
 
More generally still, we will often apply the above construction with a function $h{\colon\thinspace} U \to \RR$ defined only on an open subset $U\subset \R^n$.
Consider the associated  graph brane
$
\Gamma_h \subset T^*U.
$ 
Given  a Lagrangian brane $\sL\subset M \times T^*\R^n$, with brane structure $(f_\sL, \alpha_\sL, [c_\sL])$,
we can form the {\em carved} Lagrangian brane 
$$
\sL_h = \sL + \Gamma_h= \{ (m, x, \xi) \in M \times T^*U \, | \, (m, x, \xi - dh_x ) \in \sL|_U\}\subset M \times T^*U
$$
 with brane structure $(f_{\sL, h}, \alpha_{\sL, h}, [c_\sL])$,
 where $f_{\sL, h}$ is the sum of $f_\sL$ and the pullback of $h$, and
 $\alpha_{\sL, h}$ is transported from the grading $\alpha_\sL$ via an isotopy of $\Gamma_h$  to the zero section through graphs. We say that $\sL_h$ is carved from $\sL$ by the function $h$. 
 
Clearly the connectedness of the resulting brane depends on the behavior of $h$. In any event, we use the term ``carved" since we envision that we are using $h$ to remove any part of $\sL$
that does not live above $U\subset \R^n$---this will be apparent from our examples of $h$, which will have derivatives tending to $\infty$ near $\del U$.

Of course, one needs to be careful and impose conditions to ensure the carved  brane $\sL_h$  is reasonable. In particular, one needs to be careful that $h$, $\sL$, and their interaction  are reasonable at the boundary of $U \subset \R^n$. Rather than trying
to systematize this, we will comment as necessary   in particular cases.

\subsubsection{Pullbacks}
There are many general things to say about the functoriality of branes, but we will restrict our remarks
to what we specifically need.

First, suppose $U\subset \R^m, V\subset \R^n$ are open domains, and $p{\colon\thinspace}U \to V$ is a submersion.
This provides a correspondence
$$
\xymatrix{
T^*U & \ar[l]_-i T^*V \times_V U \ar[r]^-q & T^*V  
}
$$
in which $i$ is an embedding.

Given a Lagrangian brane $\sL\subset T^*V$, 
we can form the {\em pullback} Lagrangian brane 
$$
p^*\sL = i(q^{-1}(\sL)) \subset T^*U
$$
with the evident brane structure pulled back along $q$.

More generally, given an exact symplectic manifold $M$, and a Lagrangian brane $\sL\subset M \times T^*V$, 
we can similarly form the {\em pullback} Lagrangian brane 
$$
p^*\sL = i(q^{-1}(\sL))\subset M \times T^*U
$$
with the evident brane structure pulled back along $s$.

 Suppose in addition there is a section $s{\colon\thinspace}V \to U$ to the submersion $p$. Then we have the left-inverse $s^*$
 to the pullback $p^*$, and we refer to its result as the {\em restricted} Lagrangian brane
 $$
 s^*(p^*(\sL)) = \sL \subset M \times T^*V.
 $$
 In this case, we will often say that $p^*\sL$ is {\em collared} by $\sL$ along the subspace $s(V)\subset U$
 in the direction of the map $p$.

 \begin{example}[Orthogonal projection]
Suppose  $p{\colon\thinspace}\R^{m+k}\to  \R^m$ is the orthogonal projection.

Given an exact symplectic manifold $M$, and a Lagrangian brane $\sL\subset M \times T^*\R^m$, 
the {pullback} Lagrangian brane  is simply the product with the zero section
$$
p^*\sL = \sL \times T^*_{\R^k} \R^k \subset M \times T^*\R^{m+k}.
$$

 The restriction along the section  $s{\colon\thinspace}\R^{m}\to  \R^{m+k}$ is simply the original brane
 $$
s^*(\sL \times T^*_{\R^k} \R^k) = \sL \subset M \times T^*\R^{m}.
$$

 \end{example}
 
 \begin{example}[Rotation via angular projection]\label{example: rotation}
Let $\|\cdot\|{\colon\thinspace}\R^n\to \R$ be the Euclidean length function.
Fix $0<a< b$, and consider the annulus and interval 
$$
U = \{v\in \R^n \, |\, a< \|v\| < b\}
\qquad
V = (a, b)\subset \R
$$

Suppose $p{\colon\thinspace}U \to V$ is the restriction of $\|\cdot\|$.

Given an exact symplectic manifold $M$, and a Lagrangian brane $\sL\subset M \times T^*V$, 
we can form the  {\em rotated} Lagrangian brane
$$
p^*\sL =\sL \times T^*_{S^{n-1}}S^{n-1}  \subset M \times T^*U.
$$

Suppose in addition there is a small $\epsilon >0$ such that over the open intervals 
$$
\xymatrix{
V_a = (a, a+ \epsilon)
& V_b = (b-\epsilon, b)
}$$
we have identifications
$$
\xymatrix{
\sL|_{V_a} = \sL_a \times T^*_{V_a} V_a
&
\sL|_{V_b} = \sL_b \times T^*_{V_b} V_b
}$$
for some Lagrangian branes $\sL_a, \sL_b \subset M$.

Then we can form a single brane 
$$
\sR \subset M \times T^*\R^n
$$ by gluing together $p^*\sL$ over the annulus $U$
 with the product branes
$$
\xymatrix{
\sL_a \times T^*_{W_a}W_a
&
\sL_b \times T^*_{W_b}W_b
}$$
defined over the domains
$$
\xymatrix{
W_a = \{v\in \R^n \, |\,  \|v\| < a\}
&
W_b = \{v\in \R^n \, |\, b< \|v\| \}
}$$ 
\end{example}

Finally, there are myriad variations on this theme. For instance, a useful further construction is to consider the restriction
$\sR|_Q$ of the rotated Lagrangian brane $\sR$ to a quadrant  $Q\subset \R^n$. With a slight tweak of the original function $p$, we can achieve that $\sR|_Q$ is collared along the boundary of the quadrant in the normal directions. 

See Figure~\ref{fig. P_rot} in the case $n=2$ and $Q = \{(x, y) \in \R^2 \, | \, x, y > 0\}$. By a small perturbation, we have arranged so that there is $\epsilon > 0$ so that over the boxes
$$
\xymatrix{
(a, b) \times (0, \epsilon)
&
(0, \epsilon) \times (a, b)
}
$$
the rotated brane $\sR|_Q$ is the collared product
$$
\xymatrix{
\sL \times T^*_{(0, \epsilon)} (0, \epsilon)
&
T^*_{(0, \epsilon)} (0, \epsilon) \times \sL
}
$$
This particular example will arise in our constructions of kernels in Section~\ref{sect: kernel homotopy}.

\begin{figure}[ht]
	$$\xy
	\xyimport(10,10)(0,0){\includegraphics[height=3in]{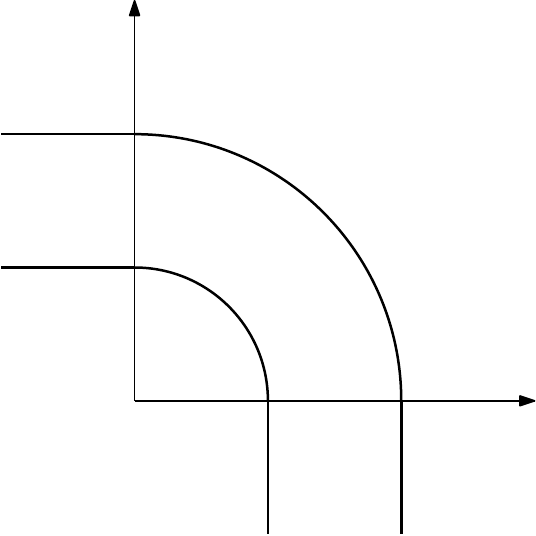}}
		,(5.2,5.2)*{\sL}
		,(1,1)*{\sL_a}
		,(7,7)*{\sL_b}
		,(10,2)*{x}
		,(3,10)*{y}
		,(1,6.2)*{\sL}
		,(6.2,1)*{\sL}
	\endxy$$
	\begin{image}\label{fig. P_rot} The brane $\sR|_Q$, depicted as living over the quadrant $Q = \{(x,y) \in \RR^2 \, | \, x,y > 0 \}$; we have arranged so that $\sR|_Q$ is collared along the axes of $Q$. In Section~\ref{sect: kernel homotopy}, we will apply the construction of $\sR|_Q$ when our brane $\sL$ happens to be a Lagrangian cobordism $\sP$.
		\end{image}
	\end{figure}

\subsection{Pictures of 1-forms} \label{ssect: graphs}
Via the usual metric on $\R^n$, we can transform one-forms into vector fields,
and exact one-forms into gradient vector fields.
In our initial proofs, we will write fairly explicit formulas  for one-forms,
while
in later proofs, we will often draw qualitatively meaningful vector fields and leave to the reader the choice of their precise realization.

Let us fix an auxiliary smooth, non-decreasing function $\tau{\colon\thinspace} \R \to \R$ such that 
	$$
	\xymatrix{
	\tau^{-1}(0) = (-\infty,0]
&
	\tau^{-1}(1) = [1,\infty).
	}$$
In Figure~\ref{figure.1-forms} we depict the 1-form $dh$ for various functions $h(t,x) $ of two variables. In the rest of this work, we will draw large arrows to denote places where a vector field shoots off to infinity in the indicated direction, and (when necessary) small arrows to highlight where the vector field is non-zero in the indicated direction.

\begin{figure}[ht]
\includegraphics[width=2in]{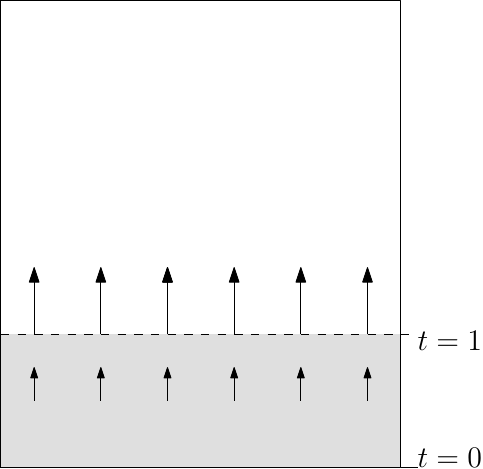}
\qquad
\includegraphics[width=2in]{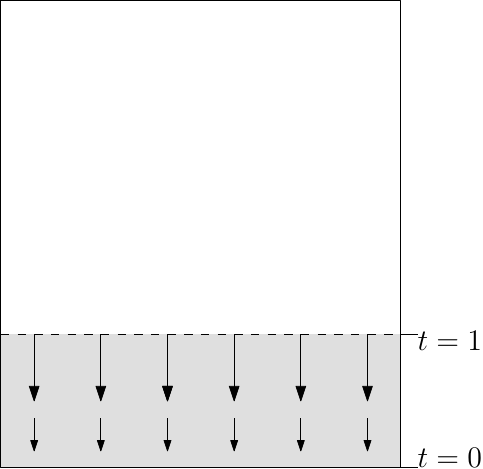}
\vspace{1cm}
\\
\includegraphics[width=2.3in]{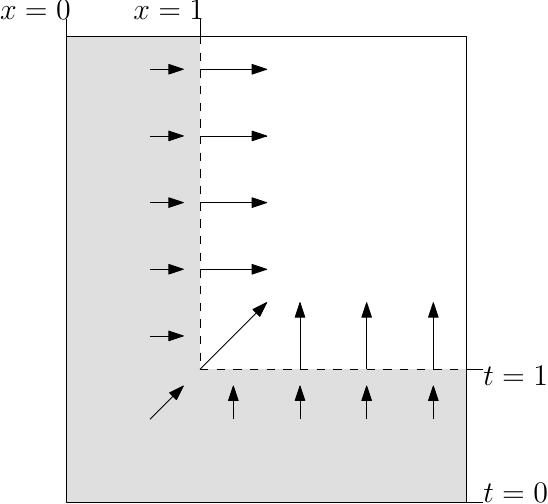}
\qquad
\includegraphics[width=2.3in]{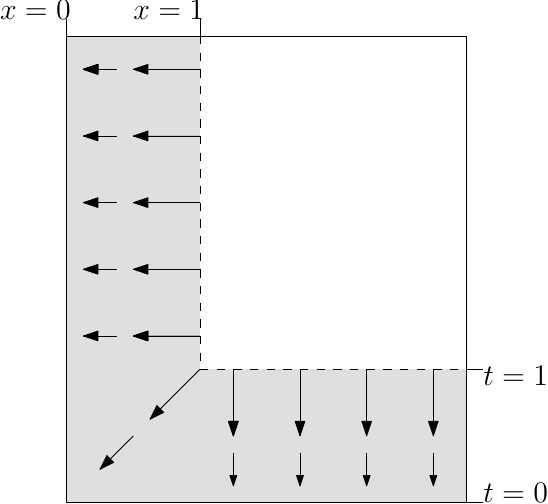}

\begin{image}\label{figure.1-forms}
On the upper-left, a drawing of $dh$, where $h(t, x) =-\log(1-\tau(t))$.
On the upper-right, a drawing of $dh$ for $h(t, x) =\log(1-\tau(t))$.
On the lower-left, a drawing of $dh$ for $h(t, x) =-\log(1-\tau(t)\tau(x))$
On the lower-right, a drawing of $dh$ for $h(t, x) =\log(1-\tau(t)\tau(x))$.
\end{image}
\end{figure}

\begin{remark}[Corners and rounding]\label{remark.corners}
Figure~\ref{figure.1-forms} has ``corners'' in the domains over which the Lagrangians are supported. One could just as easy taken a model where, rather than corners, one has a smooth boundary. This latter option is better suited for a version of $\lag_\Lambda(M)$ in which all objects and (higher) morphisms are eventually conical. 
\end{remark}

\def\fC{\mathfrak C}
\def\sC{\mathsf C}

\def\sSet{\on{\it sSet}}
\def\sCat{\on{\it sCat}}

\def\st{\mathsf t}
\def\sa{\mathsf a}
\def\sb{\mathsf b}
\def\sc{\mathsf c}

\clearpage
\section{Lagrangian cobordisms}\label{sect. lag cobs}

\subsection{Cubical collaring}

We introduce a notion of {\em collaring} which will allow us to define the face and degeneracy maps for the $\infty$-category $\Lag$. For a motivation of these collaring conditions, we refer the reader to Appendix~\ref{appendix. nerve} where we discuss the coherent nerve of a category, which inspired our definitions.

Fix $n$, a pair of vectors $\sa = (a_1, \ldots, a_n), \sb = (b_1, \ldots, b_n) \in \R^n$ with $a_i < b_i$ for all $i$. Consider the 
$n$-dimensional rectangular solid
$$
R = [a_1, b_1] \times\cdots\times [a_n, b_n] \subset \R^n.
$$
We will write $R^\circ\subset R$ for its interior, and $T^*R^\circ$ for the cotangent bundle of $R^\circ$.

The facets $R_\beta \subset R$
correspond to maps
$$
\xymatrix{
\beta{\colon\thinspace}\{1, \ldots, n\} \ar[r] & \{a, *, b\}
}
$$  
by the rule $x\in R$ lies in $R_\beta$ if and only if $x_i = \beta(i)_i$ when $\beta(i)$ is equal to $a$ or $b$,
or $x_i \in (a_i, b_i)$ when $\beta(i) = *$. For example, the interior $R^\circ$ is simply $R_\beta$ when $\beta$ is the constant map to $*$.

Given a facet $R_\beta \subset R$, by a {\em collaring neighborhood} $U_\beta \subset R$ we will mean an open subset containing 
$R_\beta$ such that $\pi_\beta(U_\beta) = R_\beta$, where $\pi_B$ is the orthogonal projection restricted to $U_\beta$.

\begin{definition}[Cubical collaring]
Suppose given a Lagrangian brane 
$$\sP \subset M \times T^* R^\circ.
$$

 We say that $\sP$ is {\em collared} if for each facet $R_\beta \subset R$, there is a 
  Lagrangian brane 
  $$
  \sP_\beta \subset M \times T^* R_\beta,
  $$
and  collaring neighborhood $U_\beta\subset R$
 such that we have an identity of branes
 $$
 \sP |_{R^\circ\cap U_\beta} = \pi_\beta^*\sP_\beta=
 \sP_\beta \times T^*_{N_\beta} N_\beta
 $$ 
  where $N_\beta$
is the affine space orthogonal to the facet $F_\beta$.
 \end{definition}

Observe that if a Lagrangian brane $\sP \subset M \times T^*R^\circ$ is collared, then the collaring branes 
$\sP_\beta \subset M \times T^*R_\beta$ are uniquely determined. When working with collared
Lagrangian branes, we will often abuse notation and write $\sP \subset T^*R$ rather than indicating the interior $R^\circ \subset R$.

\subsection{Combinatorics of cubes and simplices}
In order to refine the behavior of Lagrangian branes living over a rectangular solid, we will need to gather
some preliminaries.

Let $[n]$ denote the ordered set $\{0, 1, \ldots, n\}$. We will think of elements of $[n]$ as the vertices of the $n$-simplex $\Delta^n$.

Let $\cP(n)$ denote the power set of $[n]$. We will think of elements of $J \in \cP(n)$ as the vertices of a facet $\Delta(J)\subset \Delta^n$. We will write $J^{min}$ and $J^{max}$ for the minimum and
maximum elements of $J$ respectively.

Let $\cQ(n)\subset \cP(n)$ denote those subsets
$I \subset [n]$ such that $0, n\in I$.
Here are two alternative ways to think of an element $I\in \cQ(n)$.
On the one hand, $I$ encodes a path (directed sequence of vertices) in $\Delta^n$ from the initial vertex $0$ to the terminal vertex $n$. On the other hand, $I$ encodes a vertex $v(I)$ of the $(n-1)$-cube 
$$
c(n-1) = [0,1]^{n-1} \subset \R^{n-1}
$$
by setting $v(I)_j$ equal to $1$ if $j\in I$ and $0$ if $j\not \in I$. Note that any surjection $[n'] \to [n]$ of ordered sets give rise to a map $\cQ(n') \to \cQ(n)$, and in turn a continuous map $c(n'-1) \to c(n-1)$ which is a submersion on each facet. 

Consider the incidence set 
$$
\cQ(n)_* = \{ (I, k) \in \cQ(n) \times [n] \, | \, k\in I\}.
$$
On the one hand, an element of $\cQ(n)_*$ encodes a path $I$ in the $n$-simplex  $\Delta^n$ together with a marked vertex $k$ on the path.
On the other hand,  $\cQ(n)_*$ encodes a point $v(I, k) = (v(I), k)$ of the $n$-dimensional rectangular solid
$$
C(n) = c(n-1) \times [0,n] = [0,1]^{n-1}\times [0,n] \subset \R^{n} \cong \RR^{n-1} \times \RR_t.
$$
As the notation indicates, we will often think of the last coordinate, $t$, as the time direction. 
Let us next turn to the facets of the $(n-1)$-cube $c(n-1)$. 

We represent $I\in  \cQ(n)$ by a path of $1$-subsimplices
	$$
	\{0,i_1\}, \{i_1,i_2\},\ldots, \{i_{\ell-1}, n\} \in \cP(n),
	\quad
	\mbox{where $I = \{0<i_1< i_2<\cdots <i_{\ell-1}< n\}.$}
	$$
More generally, a {chain} $\beta$ of subsimplices is a sequence of subsimplices
	$$
	J_1, J_2, \ldots, J_\ell \in \cP(n)
	$$
satisfying the following:
\begin{enumerate}
	\item 
	$|J_a| > 1$ for any $a = 1,\ldots, \ell$,
	\item 
	$J_1^{min} = 0$, $J^{max}_\ell = n$,
	\item
	$J_{a}^{max} = J_{a+1}^{min}$, for any $a = 1,\ldots, \ell-1$,
	\item 
	if $a<b$, then $j_a \leq j_b$, for any $j_a \in J_a$, $j_b\in J_b$.
\end{enumerate}

The facets of the $(n-1)$-cube $c(n-1)$ are in bijection with such chains: given a chain $\beta$,
the corresponding facet $f(\beta) \subset c(n-1)$
is the convex hull of the vertices $v(I)$ corresponding to paths $I$ through the subsimplices in $\beta$. More precisely,
$v(I)$ is a vertex of $f(\beta)$ if
for each $i_a\in I$, there  is some subsimplex $J_b$ in the chain $\beta$
such that $i_a, i_{a+1}\in J_b$. The dimension of the facet $f(\beta)$ is the sum
$$
\dim f(\beta)  = \sum_{a=1}^\ell (|J_a|-1).
$$

By a vertical facet 
of the rectangular solid $C(n) =c(n-1) \times [0,n]$, we mean a facet of the form 
$$F(\beta) = f(\beta) \times [0,n].
$$
Given a subsimplex $J$ appearing in a chain $\beta$,
form the rectangular solid
$$
F(\beta, J) = \{ (x, t) \in F(\beta) \, |\, t\in [J^{min},J^{max}] \}.
$$
We will call $F(\beta, J)$ a refined facet of the rectangular solid $C(n)$. 
As we vary $J$, the collection of refined facets $F(\beta, J)$ form a refinement of the vertical facet $F(\beta)$.

Finally, we will introduce the trivial directions within a refined facet $F(\beta, J)$.
Let us write $C(J)$ for the $(|J|-1)$-dimensional rectangular solid
$$
C(J) =  [0,1]^{|J|-2}\times [J^{min}, J^{max}] \subset \R^{|J|-1}.
$$
By construction, we have a natural orthogonal projection
$$
\xymatrix{
p(\beta, J){\colon\thinspace}F(\beta, J) \ar[r] & C(J)
}$$
respecting projection onto the last coordinate direction. It forgets the variable directions $x_i$ for those indices $i \in\{1, \ldots, n-1\}$ such that $i\not \in J$.


\subsection{Factorization}

Now we are ready to refine the behavior of Lagrangian branes living over a rectangular solid.  
We work with the $(n-1)$-cube 
$$
c(n-1) = [0,1]^{n-1} \subset \R^{n-1}
$$
and the
$n$-dimensional rectangular solid
$$
C(n) = c(n-1) \times [0,n] = [0,1]^{n-1}\times [0,n] \subset \R^{n} \cong \RR^{n-1} \times \RR_t.
$$

Fix Lagrangian branes $\sL_0, \sL_n \subset M$.

\begin{definition}[Factorized Lagrangian $n$-cobordism]
A factorized Lagrangian $n$-cobordism from $\sL_0$ to $\sL_n$ is a collared Lagrangian brane
$$
\sP \subset M \times T^*C(n)
$$
satisfying the following. 

(1) Along the face 
$
c(n-1) \times \{0\} \subset C(n),
$
$\sP$ is collared by 
$$
\sL_0 \times T^*_{c(n-1)} c(n-1),$$
and along the face 
$
c(n-1) \times \{n\} \subset C(n),
$
$\sP$ is collared by $$\sL_n \times T^*_{c(n-1)} c(n-1).
$$

(2) For all $J\in \cP(n)$, there is a collared 
Lagrangian brane
$$
\sP_J \subset M \times T^*C(J)
$$
such that for all vertical facets $F(\beta)\subset C(n)$, the restriction of the collaring brane $\sP_\beta$
to the refined facet $F(\beta, J) \subset F(\beta)$ is equal to the pullback brane
$$
\sP_\beta|_{F(\beta, J)} = p(\beta, J)^*\sP_J.
$$
\end{definition}

A factorized Lagrangian $n$-cobordism $\sP$ from $\sL_0$ to $\sL_n$ contains a lot of information that may not be apparent at first glance. 
For instance, it contains the information of  intermediate
 Lagrangian branes 
 $$\sL_1, \ldots, \sL_{n-1} \subset M.
 $$ Namely, for $k\in \{1, \ldots, n-1\}$, 
 consider the vertex $v(I) \in c(n-1)$ corresponding to $I = \{0, k, n\} \subset [n]$. Then
 the intermediate brane $\sL_k\subset M$ appears as a collaring brane
 at the point $(v(I), k)$.
 
\begin{example}
For instance, when $n=4$, the facet defined by the function
\[
\beta(i) = \left\{ \begin{array}{ll} (0,1) & i=1,3 \\ \{1\} & i=2 \end{array}\right.
\]
can be split into two parts by $t \in [0,2]$ and $t \in [2,4]$. One sees that this facet is described completely by extending the Lagrangians associated to $J=[012] \subset \Delta^4$ and $J'=[234] \subset \Delta^4$, composing them appropriately.
\end{example}

Some low-dimensional examples are illustrated in Figures~\ref{image: 2-simplex-cobordism} and~\ref{image: 3-simplex}.

\begin{figure}
    \begin{tikzpicture}
		\draw (0,0) -- (4,0) -- (4,8) -- (0,8) -- cycle;
		\draw (0,0) circle (0.1 em);
			\draw[dashed](0,0) -- (-1,-1) node[anchor=north east] {$\sL_0$};
		\draw (4,0) circle (0.1 em);
			\draw[dashed](4,0) -- (5,-1) node[anchor=north west] {$\sL_0$};
		\draw (4,4) circle (0.1 em);
			\draw[dashed](4,4) -- (5,4) node[anchor=west] {$\sL_1$};
		\draw (4,8) circle (0.1 em);	
			\draw[dashed](4,8) -- (5,9) node[anchor=south west] {$\sL_2$};
		\draw (0,8) circle (0.1 em);		
			\draw[dashed](0,8) -- (-1,9) node[anchor= south east] {$\sL_2$};
		\draw (0,4) node[anchor = east] {$\sP_{02}$};
		\draw (4,2) node[anchor = west] {$\sP_{01}$};
		\draw (4,6) node[anchor = west] {$\sP_{12}$};
		\draw (2,4) node {$\sP$};
    \end{tikzpicture}
	
\begin{image}\label{image: 2-simplex-cobordism} 
A factorized 2-cobordism $\sP$ from $\sL_0$ to $\sL_2$, projected onto $C(2) \subset\R^2$. The indicated vertices are collared by the Lagrangians $\sL_i$ for $i=0,1,2$ and the indicated facets are collared by $P_J$ for $J=\{0,1\},\{0,2\},\{1,2\}$. One should think of $\sP$ as a homotopy between the ``composition'' $\sP_{12} \circ \sP_{01}$ and the Lagrangian $\sP_{02}$.
	\end{image}
\end{figure}

\begin{figure}
\includegraphics[height=3in]{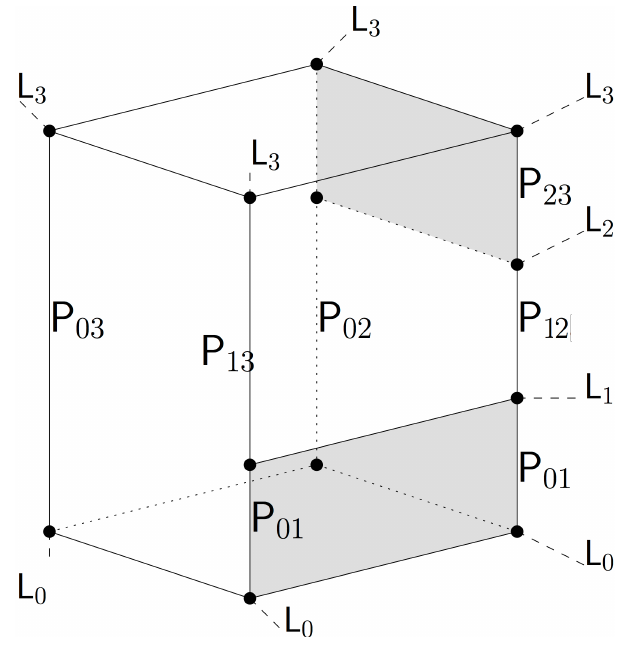}
	\begin{image}\label{image: 3-simplex}
	A factorized 3-cobordism $\sP$ from $\sL_0$ to $\sL_3$, projected on to $C(3) \subset \R^3$. We have indicated how the various edges and vertices are collared. Note that grey shaded regions are identical to $P_{01} \times \R$ and $P_{23} \times \R$.
	\end{image}
\end{figure}


\subsection{Time symmetries}
Set $\Diff(0)$ to be the trivial group, and let $\Diff(1)$ denote 
 the 
group  of orientation-preserving diffeomorphisms of the interval $[0,1]$ that are the identity
in some neighborhood of the end points $\{0, 1\} \subset [0,1]$.

Here is an inductive generalization for $n>1$. Let $\Diff(n)$ denote the 
group  of orientation-preserving diffeomorphisms of the $n$-dimensional
rectangular solid
$$
\xymatrix{
\varphi{\colon\thinspace} C(n)\ar[r]^-\sim &  C(n)
}
$$ 
that satisfy the following:

(1) {\em time symmetry}: $\varphi$ respects the orthogonal projection to the $(n-1)$-cube $c(n-1)$ in the sense that
there is a map
$$
\xymatrix{
\gamma{\colon\thinspace}C(n) \ar[r] & [0,n]
}
$$
such that we have
$$
\xymatrix{
\varphi(x, t) = (x, \gamma(x, t)),
& (x, t) \in C(n).
}$$
In particular, $\varphi$ restricts to a diffeomorphism of each vertical facet 
$
F(\beta) \subset C(n).
$

(2) {\em preserves collaring}: (a) for each vertical facet $F(\beta) \subset C(n)$, there is a 
   collaring neighborhood $U(\beta)\subset C(n)$
 such that $\varphi$ respects the orthogonal projection
 $$
\xymatrix{
\pi(\beta) {\colon\thinspace} U(\beta) \ar[r] &  F(\beta)
}
$$
in the sense that
$$
\xymatrix{
\pi(\beta)\circ \varphi = \varphi \circ \pi(\beta).
}$$ 

(b) there are open neighborhoods of the faces 
$$
c(n-1) \times \{0\}, c(n-1) \times \{n\} \subset C(n)
$$
on which $\varphi$ is the identity.

(3) {\em preserves refined facets}: $\varphi$ restricts to a diffeomorphism of each refined facet 
$$
F(\beta, J) \subset C(n).
$$

(4) {\em preserves trivial directions of refined facets}: for each refined strata 
$$
F(\beta, J) \subset C(n),
$$
there is an element 
$$\varphi(\beta, J)\in \Diff(|J|-1)
$$ 
such that we have
$$
\varphi|_{F(\beta,J)} = p(\beta, J)^*\varphi(\beta, J)
$$
under the projection 
$$
\xymatrix{
p(\beta, J){\colon\thinspace}F(\beta, J) \ar[r] & C(J).
}$$

We will refer to elements $\varphi\in \Diff(n)$ as {\em time reparametrizations}.
Any time reparametrization $\varphi \in \Diff(n)$ induces a symplectomorphism
$$
\xymatrix{
\varphi^*{\colon\thinspace}M \times T^*C(n) \ar[r]^-\sim & M \times T^*C(n)
}
$$
respecting all of the structures in play. 
Thus for $\sL_0, \sL_n \subset M$ fixed Lagrangian branes,
$\varphi^*$ induces a bijection on the set of 
factorized Lagrangian $n$-cobordisms from $\sL_0$ to $\sL_n$.

\begin{definition}[Lagrangian $n$-cobordism]
A {\em Lagrangian $n$-cobordism} is an equivalence class of factorized Lagrangian $n$-cobordisms
under the action of the group $\Diff(n)$.
\end{definition}

It follows that the notion of Lagrangian $n$-cobordism is independent of the particular choice of timescale.
More precisely, if we produce a factorized Lagrangian $n$-cobordism with some alternative ordered embedding
of the set $[n] = \{0, 1, \ldots, n\}$ inside of the real line $\R$, we can unambiguously transport it  to a 
 Lagrangian $n$-cobordism living over the standard rectangular solid $C(n)$.
As we will see in Section~\ref{sect. weak kan} below, one immediate consequence is that Lagrangian cobordisms naturally form simplicial sets.

\begin{remark}
It is possible to define a version of $\lag$ in which one does not mod out by time symmetries. Instead, one can define an $n$-simplex to be a concrete brane inside $M \times T^*\RR^{n}$, collared above some rectangle, and whose faces are equipped with specified smooth embeddings of $\RR^{n-1}$ into the various facets of the rectangle. 
\end{remark}



\subsection{Non-characteristic propagation}

\nc\Flow{\on{Flow}}

Now we introduce the Lagrangian support $\Lambda\subset M$ into the story. Identifying $T^*\RR_t \cong \RR_t \times \RR_t^\vee$, let 
	\[
	\pi_{\RR_t^\vee}{\colon\thinspace} M \times T^*C(n) \to \RR_t^\vee
	\qquad\text{and}\qquad
	\pi_M{\colon\thinspace} M \times T^*C(n) \to M
	\] 
be the projections to the cotangent coordinate of the $t$ variable, and to the $M$ coordinate, respectively. We assume we have fixed a metric on $M$, and denote by $\dist(-,-)$ the (minimum) distance between two sets.

\begin{definition}[Non-characteristic cobordisms]\label{def. nonchar}
A {\em $\Lambda$-non-characteristic} factorized Lagrangian $n$-cobordism
is a factorized Lagrangian $n$-cobordism 
$$\sP\subset M \times T^*C(n)
$$ 
for which there exists $T \in \RR_t^\vee$ such that
	\[
	\inf_{\substack{x \in P \\ \pi_{\RR_t^\vee}(x) < T}} \dist_M(\pi_M(x),\Lambda) > 0.
	\]
The notion is invariant under the action of $\Diff(n)$, and hence descends to provide the notion
of a {\em $\Lambda$-non-characteristic Lagrangian $n$-cobordism}. We will synonymously call such a $\sP$ {\em $\Lambda$-avoiding}.
\end{definition}



\subsection{Basic examples}

\begin{example}[The identity $1$-cobordism]\label{ex: identity}
For any Lagrangian brane $\sL\subset M$, we have the Lagrangian $1$-cobordism 
$$
 \sL \times T^*_\R\R \subset M \times T^*\R
$$  
obtained by taking the product of $\sL$ with the zero section equipped with its canonical brane structure.
We refer to $ \sL \times T^*_\R\R$ as the {\em identity cobordism} of the brane $\sL$.
\end{example}

\begin{example}[Zero morphisms]\label{example: zero maps}
 Choose  $0 < \tau_0 < \tau_1 < 1$, and a function
 $\tau{\colon\thinspace}\R\to \R$ that is smooth, non-decreasing, and satisfies
$$
\tau(t) = 
 \left\{
 \begin{array}{ll}
 0, & t\leq \tau_0 \\
 1, & t\geq \tau_1
\end{array}
\right.
$$
Consider the associated functions 
$$
\sigma_{\es\to}(t)  = \log \tau(t)
\qquad \sigma_{\to\es}(t)  = -\log (1-\tau(t)),
$$ and the corresponding graph branes
$$
\Gamma_{\sigma_{\es\to}}, \Gamma_{\sigma_{\to\es}}\subset T^*\R.
$$
Let $\zerolag$ denote the empty Lagrangian. We define the {\em zero morphisms} 
$$z_{\emptyset\to}{\colon\thinspace}\zerolag \to  \sL
\qquad
 z_{\to\es}{\colon\thinspace}  \sL\to \zerolag $$ 
 to be the equivalence classes under time reparametrizations of the product branes
$$
\sL \times \Gamma_{\sigma_{\es\to}}, \sL \times \Gamma_{\sigma_{\to\es}}  \subset M \times T^*C(1).
$$
Note that while these cobordisms are not compact subsets of $T^*\RR_t$, they do not approach $-\infty dt$, and are hence $\Lambda$-non-characteristic.
\end{example}

\begin{remark}
If one is interested in the version of $\lag_\Lambda(M)$ in which every brane and cobordism is eventually conical, one may replace $\Gamma$ with a brane which, outside some compact set, equals $(-\infty,t_0] \coprod \{t_1\} \times [\tau,\infty) \subset \RR_t \coprod T^*_{t_1} \RR_t$ for some $t_0 < t_1$. One can choose such a brane whose primitive equals zero outside a compact set.
\end{remark}

\begin{example}[Hamiltonian isotopies]
Fix a Lagrangian brane $\sL \subset M$. Let $H{\colon\thinspace} \RR \times M \to \RR$ be a time-dependent Hamiltonian on $M$ such that $H(t,x) = 0$ for $t \leq0$ and for $t\geq1$. Let $\phi_t^H$ be its time $t$ Hamiltonian flow. Then the (equivalence class of the) brane
	\[
	\{(\phi_t^H(x), t, -H_t(x))\} \subset M \times \RR_t \times \RR_t^\vee \cong M \times T^*\RR_t,
	\qquad
	x \in \sL
	\]
is a Lagrangian cobordism from $\sL$ to $\phi_{1}^H(\sL)$. One can compute that its primitive is given by
	\[
	f(x,t) = f_\sL(x) + \int_0^t H_t \circ \phi_t^H(x) dt.
	\]
\end{example}

\clearpage
\section{\texorpdfstring{$\oo$}{oo}-category of Lagrangian cobordisms}\label{sect. lag}
Let $\Delta$ be the category whose objects are finite, non-empty, totally ordered sets, and whose morphisms are order-preserving maps. Recall that a simplicial set is a functor $\Delta^{op} \to \on{Sets}$. 
In this section, we build a simplicial set whose $n$-simplices label Lagrangian $n$-cobordisms and prove that it is an $\infty$-category.
We conclude the section by verifying basic properties of this $\infty$-category and defining basic constructions.

\subsection{Defining a weak Kan complex}\label{sect. weak kan}
We continue with the geometric setup introduced in Section~\ref{section.setup}. 
We fix the symplectic manifold $M$ and support Lagrangian $\Lambda$, setting
$
\Lag^{\dd 0} = \Lag^{\dd 0}_{\Lambda}(M).
$

\begin{definition}[Simplices]
We define the set $ \Lag^{\dd 0}(0)$ of $0$-simplices to consist of Lagrangian branes $\sL \subset M$.
For $n> 0$, we define the set $ \Lag^{\dd 0}(n)$ of $n$-simplices to consist of $\Lambda$-non-characteristic Lagrangian $n$-cobordisms. 
\end{definition}
As is well known, to define a simplicial set one need only understand the face and degeneracy maps.
Fix $n$, and consider the order-preserving face inclusion
	$$
	d_i{\colon\thinspace}[n-1] \to [n],
	\qquad
	\mbox{for $i = 0, \ldots, n$,}  
	$$
for which $i\in [n]$ is not in the image,
and the order-preserving degeneracy surjection
	$$
	s_i{\colon\thinspace}[n] \to [n-1],
	\qquad
	\mbox{for $i = 0, \ldots, n-1$,}
	$$
for which $i\in[n-1]$ has two pre-image elements.

\begin{definition}[Face maps]
For $n=1$, we define the face maps 
$$
\xymatrix{
d^*_0,d^*_1 {\colon\thinspace} \Lag^{\dd 0}(1) \to \Lag^{\dd 0} (0)
&
d^*_0(\sP) = \sL_0, d^*_1(\sP) = \sL_1.
}
$$
where $\sL_0, \sL_1 \subset M$ are the collaring branes.

Fix $n > 1$, and any $i \in [n]$. Set $J = [n]\setminus i$. 
We define the face map
$$
\xymatrix{
d^*_i {\colon\thinspace} \Lag^{\dd 0}(n) \to \Lag^{\dd 0} (n-1)
&
d^*_i(\sP) = \sP_J
}
$$
where $\sP_J\subset M \times T^*C(J)$ is the collaring brane. By applying a time reparametrization, we may identify
$C(J)$ with the standard rectangular solid $C(|J| - 1)$, and hence consider $\sP_J$ as a 
 Lagrangian $(|J|-1)$-cobordism.

\end{definition}

\begin{remark}
In general, for an arbitrary monomorphism $\sigma{\colon\thinspace} J \hookrightarrow [n]$ in $\Delta$, we can take $\sigma^*(\sP)$ to be a time reparametrization
of the collaring brane $\sP_J \subset M \times T^* C(J)$.
\end{remark}

To define degeneracy maps $s_i^*$ for a $(n-1)$-cobordism $\cP \subset M \times T^*C(n-1)$, we first dilate {\em time} at $t=i$ to construct a cobordism $\cP_i'$ living over $c(n-1) \times[0,n]$. We then pull back along a map $c(n) \to c(n-1)$ to obtain the $s_i^*(\cP)$. Specifically:

\begin{enumerate}
	\item For $0<i<n-1$, choose a small $\epsilon >0$, and a smooth embedding of intervals
		$$
		\xymatrix{
		p_{i}{\colon\thinspace}[0,n] \ar@{^(->}[r] & [0, n-1]
		}
		$$
	such that  $p_i(t) = t$ for $t< i-\epsilon$, and $p_i(t) = t-1$ for $t> i+1 + \epsilon$. Then define
		\[
		\sP'_i = (\id_{c(n-1)} \times p_i)^*(\sP).
		\]
	When $i=0$, we can produce a similar brane 
		$$
		\sP'_0\subset M \times T^*(c(n-1) \times [0,n])
		$$
	by translating $\sP$ to the time interval $[1,n]$, and then gluing to it the constant brane
		$$
		\sL_0 \times T_{c(n-1) \times [0,1]}^*(c(n-1) \times [0,1]).
		$$
	When $i=n-1$, we can produce a similar brane 
		$$
		\sP'_{n-1}\subset M \times T^*(c(n-1) \times [0,n])
		$$
	by gluing to $\sP$ the constant brane
		$$
		\sL_{n-1} \times T_{c(n-1) \times [n-1,n]}^*(c(n-1) \times [n-1,n]).
		$$
	\item
	Recall the standard triangulation of the $(n-1)$-cube $c(n-1)$: The vertices
$v(I) \in c(n-1)$ correspond to subsets $I\subset [n]$ containing both $0, n\in [n]$. With this indexing, the $k$-simplices of the standard triangulation
of $c(n-1)$ have vertices indexed by increasing sequences of subsets $I_1 \subset I_2 \subset \cdots I_k \subset [n]$.

	Since a surjection $\sigma{\colon\thinspace}[m]\to [n]$ induces a map from these subsets of $[m]$ to $\{I \subset [n]\}$, one obtains a simplicial map 
	\eqn
	\nonumber
	\xymatrix{
	c(\sigma){\colon\thinspace}c(m-1) \ar[r] &  c(n-1)
	}
	\eqnd
with respect to the standard triangulations. In general, the map $c(\sigma)$ is not a simple orthogonal projection (and not even smooth), but linear on each simplex in the domain.

\end{enumerate}

\begin{definition}[Degeneracy maps]
Fix any $n > 0$, and any $i \in [n-1]$. 

We define the degeneracy map
$$
\xymatrix{
s^*_i {\colon\thinspace} \Lag^{\dd 0}(n-1) \to \Lag^{\dd 0} (n)
&
s^*_i(\sP) = c(s_i)^*(\sP'_i)
}
$$
where as introduced above the brane
$$
\sP_i' \subset M \times T^*(c(n-1) \times [0,n])
$$
is the $i$th stretched time reparametrization of $\sP$ for sufficiently small $\epsilon>0$, and the simplicial map
$$
\xymatrix{
c(s_i){\colon\thinspace} c(n) \times [0,n] \ar[r] &  c(n-1) \times [0,n]
}
$$ is that induced by $s_i$.
By the pullback $c(s_i)^*(\sP'_i)$, we mean the union of the pullbacks along each simplex on which $c(s_i)$ is linear.
The collaring conditions on $\sP$, and hence on $\sP'_i$, ensure that the union glues together to a well-defined Lagrangian $n$-cobordism.
\end{definition}

\begin{proposition}
\label{prop: simplicial-set}
With the face and degeneracy maps as above, $\Lag^{\dd 0}$ is a simplicial set.
\end{proposition}

\begin{proof}[Proof of Proposition~\ref{prop: simplicial-set}]
There are four simplicial relations to check.
\begin{enumerate}	
	\item That $d_i^* d_j^* = d_{j-1}^*d_i^*$ for $i<j$ is obvious from the definition of removing the $i$th element from an ordered set. 
	\item The assignment $\sigma \mapsto c(\sigma)$ respects compositions---because $c(-)$ is induced by the map of the subsets, for instance. Further, expanding time at $t=i$, and then expanding at $t=j+1$, is equivalent to expanding time at $t=j$ and then at $t=i$ (so long as $i \leq j$). So it follows that $s_i^* s_j^* = s_{j+1}^* s_i^*$ for $i \leq j$.
	\item $d_i s_j = s_{j-1} d_i$ for $i<j$ and $d_i s_j = s_j d_{i-1}$ for $i> j+1$---this follows from the same reasons as above: The operations $d_i$ and $s_j$ respect the maps $c(s_i), c(d_i)$ induced on the $c(n)$ coordinates, and obviously respect the simplicial relations in the $t$ coordinates. 
	\item Finally, if $i=j$ or $i=j+1$, we see that $d_i^* s_j^*$ is the identity. 
\end{enumerate}
\end{proof}

Let $\Delta^n$ be the simplicial set represented by the object $[n]=\{0<1<\ldots<n\}$. Recall that for $k=1, \ldots, n-1$, the $k$th {\em inner horn} $\Lambda^n_k\subset \Delta^{n}$ is the subsimplicial set obtained by deleting the interior of $\Delta^{n}$ and the codimension one face opposite the $k$th vertex.
A simplicial set  $X$ is called a {\em weak Kan complex} if any map from an inner horn $\Lambda^n_k \to X$ extends to a map $\Delta^{n} \to X$.

\begin{proposition}[Weak Kan property]
\label{prop: kan}
The simplicial set $\Lag^{\dd 0}$ is a weak Kan complex.
\end{proposition}

\begin{proof}[Proof of Proposition~\ref{prop: kan}]
Let $0<i<n$ and consider the face $C(n)_i$ of $C(n)$ spanned by the facets for which $i \not \in J$. By definition, an inner horn $\Lambda^n_i \to \lag^{\dd 0}$ defines a Lagrangian living over every face of $(\del c(n) \times [0,n])$ except for $C(n)_i$---note that one must choose a coherent collection of space dilations to ensure that cobordisms defined by an $(n-1)$-simplex are pulled back to a cobordism living over $c(n-1) \times [0,n]$. One can then choose a retraction $C(n) \to (\del c(n) \times [0,n]) - C(n)_i$ which is a stratum-wise submersion, and pull back. The collaring condition ensures that the pull-back is smooth and collared appropriately.
\end{proof}

\subsection{Composition via concatenation}
In an $\oo$-category, there is no absolute notion of composition, but one is defined up to contractible choice. Specifically, given an $\oo$-category $\sC$, objects $x_0, x_1, x_2$, and morphisms $f_{01}{\colon\thinspace} x_0 \to x_1, f_{12}{\colon\thinspace}x_1\to x_2$, we are guaranteed by the weak Kan property a $2$-simplex 
	$$
	\xymatrix{
	\ar[dr]^-{f_{02}}\ar[d]_-{f_{01}} x_0  & \\
	x_1 \ar[r]_-{f_{12}} & x_2
	}
	$$
with the labeled boundary values. The ``composite" map $f_{02}{\colon\thinspace}x_0 \to x_2$ and the $2$-simplex are not unique, but the above $2$-simplex and any other such $2$-simplex will form the partial boundary of a $3$-simplex which we again will be able to fill, and so on. Thus one can show that the space of possible ``compositions" is contractible:

\begin{proposition}
Let $\Lambda^2_1 \subset \Delta^2$ be the simplicial set given by the union of the edge $0\to 1$ and $1 \to 2$. Let $X$ a simplicial set, and let $X^{\Lambda^2_1}$ and $X^{\Delta^2}$ be the mapping simplicial sets. Then the obvious map
		\[
		X^{\Delta^2} \to X^{\Lambda^2_1}
		\]
	is a trivial fibration if and only if $X$ is a quasi-category.
\end{proposition}
See for instance Joyal's notes from Barcelona~\cite{joyal-lectures}. 

In the $\oo$-category $\Lag^{\dd 0}$, the situation is more concrete.
Let $\sP_{01}{\colon\thinspace}\sL_0\to \sL_1$ and  $\sP_{12}{\colon\thinspace}\sL_1\to \sL_2$ be morphisms in $\Lag^{\dd 0}$.
Up to time reparametrization, we can view them as Lagrangian $1$-cobordisms inside the intervals
$$
\xymatrix{
\sP_{01} \subset M \times T^*C(1) = M \times T^*(0,1)
&
\sP_{12} \subset M \times T^*C(1) = M \times T^*(1,2)
}$$

\begin{definition}[Composition via concatenation]

  We define the {\em composition} 
  $$
  \sP_{02} = \sP_{12} \circ\sP_{01}{\colon\thinspace} \sL_0 \to \sL_2
  $$ to be the equivalence class under time reparametrization of the {\em concatenation}
  $$
  \sP_{01} \cup\sP_{12} \subset  M \times T^*(0,2).
  $$
\end{definition}

Returning to the language of $\oo$-categories, we can think of the composition $\sP_{02}$ as the third edge of the $2$-simplex
$$
\xymatrix{
\ar[dr]^-{\sP_{02}}\ar[d]_-{\sP_{01}} \sL_0  & \\
\sL_1 \ar[r]_-{\sP_{12}} & \sL_2
}
$$
represented by the Lagrangian $2$-cobordism 
$$
\sP \subset M \times T^*C(2) = M \times T^*(0,2) \times T^*(0,1)
$$ 
given simply by taking the product of the concatenation 
$$
 \sP_{01} \cup\sP_{12} \subset M \times T^*(0, 2)  
$$
with the zero section 
$$
T^*_{(0,1)} (0, 1) \subset T^*(0,1).
$$ 
We in fact used this technique implicitly in the proof of Proposition~\ref{prop: kan}.

More generally, let $\sQ'$ and $\sQ''$ be an $n'$-simplex and an $n''$-simplex respectively. We assume $n',n'' \geq 1$. Assume that the $n'$th vertex of $\sQ'$ and the 0th vertex of $\sQ''$ are the same object $\sL$. We define the notion of their {\em composition} $\sQ'' \circ \sQ'$, which is an $(n'+n''-1)$-simplex.

Take the splitting
\[
c(n' + n''-2)  = c(n'-1) \times c(n''-1)
\]
and consider the projections
\[
\pi' {\colon\thinspace} c(n'+n''-2) \to c(n'-1),
\qquad
\pi'' {\colon\thinspace} c(n'+n''-2) \to c(n''-1).
\]
For instance, $\pi''$ acts by forgetting the first $(n'-1)$ coordinates of a point in $c(n'+n''-2)$.
We extend $\sQ', \sQ''$ by zero sections by defining
    \[
    \bar \sQ ' 
    {\colon\thinspace}= \pi'^* \sQ' \cup \left(\sL \times \{n'\} \times c(n'+n''-2) \right) 
    \subset M \times T^*(0,n'] \times T^*c(n'-1) \times T^*c(n''-1)
    \]
and
    \[
    \bar \sQ'' 
    {\colon\thinspace}= \left(\sL \times \{0\} \times c(n'+n''-2)\right) \cup \pi''^* \sQ''
    \subset M \times T^*[0,n'') \times T^*c(n'-1) \times T^*c(n''-1).
    \]
By time reparametrization we can realize $\sQ''$ as a Lagrangian brane in the manifold
    \[
    M \times T^*[n',n'+n'') \times T^*c(n''-1).
    \]
Since $\sQ''$ and $\sQ'$ are both collared by $\sL$ at time $t = n'$, we concatenate:

\begin{definition}[Composition of higher simplices]\label{def. higher composition}
We define the composition 
\[
\sQ '' \circ \sQ'
\]
of $\sQ'$ with $\sQ''$ to be the equivalence class (under time reparametrization) of the Lagrangian brane
whose underlying Lagrangian submanifold is the concatenation
\[
\bar \sQ' \cup \bar \sQ'' \subset M \times T^*(0,n' + n'') \times c(n' + n'' -2).
\]
\end{definition}

\begin{proposition}
$\sQ'' \circ \sQ'$ is an $(n'+n''-1)$-simplex.
\end{proposition}

\begin{proof}
We proceed by induction. The base case is $n'=n''=1$, where the result is obvious, and we induct on the value $n'+n''-2$.

Consider the front $i$th face---i.e., the locus where the $i$th coordinate of $c(n-1)$ equals zero. This face of $c(n-1)$ consists of those $I\subset[n]$ such that $\{0,i,n\} \subset I$. If $i < n'$, the face of the rectangle $c(n-1) \times [0,n]$ is collared by $\sQ'' \circ \del_i \sQ'$ by definition. By induction, this is indeed a simplex. If $i\geq n'$, the face is collared by $\del_{i-n'} \sQ'' \circ \sQ'$, which is also a simplex by induction. A similar argument shows that the other facets of $C(n)$ are appropriately collared. 
\end{proof}

Just as there is a 2-simplex in $\Lag^{\dd 0}$ whose boundary contains $\sP_{12}$, $\sP_{01}$ and $\sP_{12} \circ \sP_{01}$, one can produce a $(k+l)$-simplex relating the composition $\sQ'' \circ \sQ'$ with the cobordisms $\sQ''$ and $\sQ'$:

\begin{proposition}
There exists an $(n'+n''+1)$-simplex $\sQ$ which has $\sQ'' \circ \sQ'$ as a codimension 1 face and $\sQ'$ and $\sQ''$ as faces.
\end{proposition}

That is, $\sQ'' \circ \sQ'$ has reason to be called a composition.

\begin{proof}
One simply extends $\sQ'' \circ \sQ'$ to a higher cube, for instance by taking $s_i^*(\sQ'' \circ \sQ')$ for some $i$. Then $\sQ''$ and $\sQ'$ collar the obvious subfacets of this cube.
\end{proof}

\subsection{Zero object}\label{ssect: zero object}
Recall that for $\sC$ an $\oo$-category and two objects $x,x'\in \sC(0)$, one can model the {\em mapping space} $\sC(x,x')$ as a simplicial set whose $n$-simplices are $(n+1)$-simplices in $\sC$ whose $0$th face is the degenerate face at $x'$ and whose $0$th vertex is $x$.

Recall that a zero object $z\in \sC(0)$ is an object such that $\sC(x, z)$ and $\sC(z, x)$ are contractible for all $x\in \sC(0)$.

\begin{definition}[Empty brane]
Let $\zerolag\in\Lag^{\dd 0}(0)$ denote the empty Lagrangian equipped with its unique brane structure.
\end{definition}

\begin{theorem}[Zero object]\label{thm. zero} The empty brane $\zerolag\in \Lag^{\dd 0}(0)$ is a zero object: for any object $\sL\in \Lag^{\dd 0}(0)$, the mapping spaces
$\Lag^{\dd 0}(\zerolag, \sL)$ and $\Lag^{\dd 0}(\sL, \zerolag)$ are contractible.
\end{theorem}

\begin{proof}
The proof for  $\Lag^{\dd 0}(\sL, \zerolag)$ is parallel to that for $\Lag^{\dd 0}(\zerolag, \sL)$, so we will focus on the former.
Since $\Lag^{\dd 0}(\sL, \zerolag)$ is a Kan complex, 
it suffices to show that
any map of simplicial sets 
$$
\xymatrix{
b{\colon\thinspace}\partial \Delta^k \ar[r] & \Lag^{\dd 0}(\sL, \zerolag)
}
$$ extends to a map
$$
\xymatrix{
\Delta^k\ar[r] &  \Lag^{\dd 0}(\sL, \zerolag).
}
$$ 
Note that Example~\ref{example: zero maps} handled the case $k=0$, showing that the mapping space is non-empty. 

For convenience, set $n =k+1$. The data of the map $b$ provides a compatible collection of Lagrangian branes $\sP_b$ over the boundary of the $n$-dimensional
rectangular solid 
$$C(n) = [0,1]^{n-1} \times [0,n].$$ 
In particular, since $L_\emptyset$ is the collaring brane along the face 
$$[0,1]^{n-1} \times \{n\} \subset C(n),
$$
there exists $\epsilon >0$ such that $\sP_b$ is empty over the open subset 
$$
[0,1]^{n-1} \times (n-\epsilon, n] \subset C(n).
$$

Using our collaring conventions, there exists a canonical extension of $\sP_b$ to a Lagrangian brane $\sP_B$
over an open neighborhood $B\subset C(n)$ of the boundary $\partial C(n)$.

We will construct an extension of $\sP_B$ to a Lagrangian $n$-cobordism $\sP$ over all of $C(n)$.
To this end, it is convenient to fix a small $\epsilon >0$, and choose  once and for all an elementary function
 $\tau{\colon\thinspace}\R\to \R$ that is smooth, non-decreasing, and satisfies
$$
\tau(t) = 
 \left\{
 \begin{array}{ll}
 0, & t\leq \epsilon/2\\
 1, & t\geq \epsilon
\end{array}
\right.
$$

Next define the function
$$
\xymatrix{
\sigma{\colon\thinspace}C(n) = [0,1]^{n-1} \times [0,n] \ar[r] & \R &
\sigma(x, t) = \tau(t)\prod_{i=1}^n \tau(x_i)\tau(1-x_i)
}$$
Consider the partially defined one form
$$
\kappa = -d\log(1-\sigma) = \frac{d\sigma}{1-\sigma}
$$
Finally, form the fiberwise translation
$$
\sP = \sP_B + \kappa \subset M \times T^*C(n)
$$

For $\epsilon$ small enough, it is straightforward to verify that $\sP$ is a well-defined 
 Lagrangian $n$-cobordism and provides the desired extension of $b$. In particular, 
  the fact that $\sP_b$ is $\Lambda$-non-characteristic implies the same for $\sP$.
\end{proof}


\subsection{Non-characteristic objects}\label{sect: non-characteristic}

\begin{definition}
An object $\sL \in \Lag^{\dd 0}(0)$ is said to be {\em non-characteristic} if 
$$
\sL\cap \Lambda = \emptyset.
$$
\end{definition}

\begin{proposition}
If $\sL \in \Lag^{\dd 0}(0)$ is non-characteristic, then it is a zero object.
\end{proposition}

\begin{proof}
We will show that the identity morphism $\id_\sL$ (see Example~\ref{ex: identity}) factors through the zero object $\zerolag$.
More precisely, we will construct  a $2$-simplex in $\Lag^{\dd 0}$ as follows:
$$
\xymatrix{
\ar[dr]^-{\id_\sL}\ar[d]_-{} \sL  & \\
\zerolag \ar[r]_-{} & \sL.
}
$$
Choose  $0 < \tau_0 < \tau_1 < {\frac {1} {2}} < \tau_2 < \tau_3 < 1$, and a smooth bump function
 $\tau{\colon\thinspace}\R\to \R$ such that
$$
\tau(t) = 
 \left\{
 \begin{array}{ll}
 0, & t\leq \tau_0 \\
 1, & \tau_1 \leq t \leq \tau_2 \\
  0, & t\geq \tau_3
\end{array}
\right.
$$
with $\tau$ non-decreasing over $(\tau_0, \tau_1)$ and non-increasing over $(\tau_2, \tau_3)$.

Choose   $0<\epsilon_0 < \epsilon_1 < 1$, and a smooth, non-increasing function
 $\sigma{\colon\thinspace}\R\to \R$ such that
$$
\sigma(s) = 
 \left\{
 \begin{array}{ll}
 0, & s\leq \epsilon_0 \\
 1, & s\geq \epsilon_1
\end{array}
\right.
$$

Form the partially defined  function
$$
\xymatrix{
q{\colon\thinspace}C(2) = [0,1] \times [0,2] \ar[r] &   \R &
q(s, t) = -\log(1-\sigma(s) \tau(t))
}$$
and consider the corresponding graph brane
$$
\Gamma_q \subset T^*C(2) = T^*((0, 1) \times (0,2))
$$

Now we define the sought after $2$-simplex to be represented by the Lagrangian $2$-cobordism given by the translated brane
$$
\sQ = (\sL \times T^*_{C(2)} C(2)) + \Gamma_q \subset M \times T^*C(2)
$$

It is straightforward to verify that  $\sQ$ is a well-defined 
 Lagrangian $2$-cobordism and has the prescribed boundary behavior.
In particular, since $\sL$ is assumed non-characteristic,
$\sQ$ is non-characteristic.
\end{proof}

\begin{figure}[ht]
\includegraphics[height=3.3in]{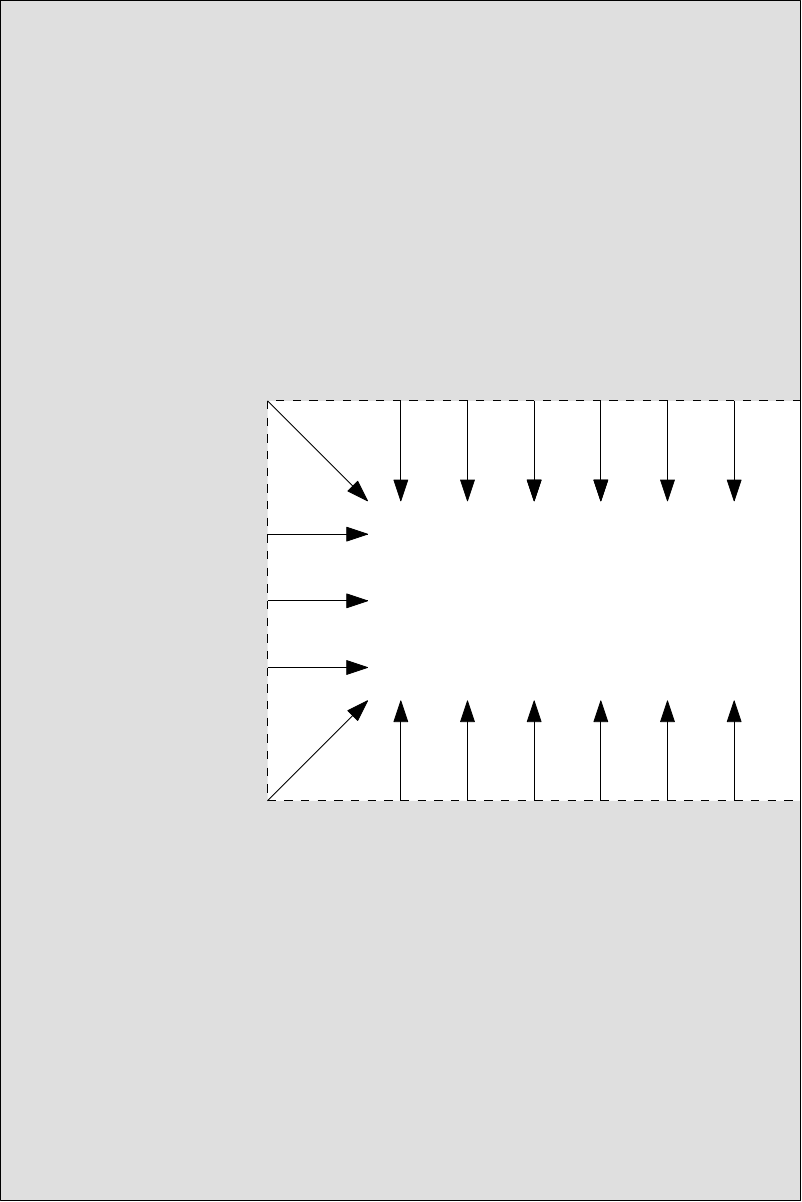}

\begin{image}\label{figure: gamma_q}
 The graph brane $\Gamma_q$ of the proposition.
 \end{image}
\end{figure}

There is a simple but  useful generalization of the above proposition. Fix a morphism $\sP{\colon\thinspace}\sL_0 \to \sL_1$ and a time $t_0\in (0,1)$.
Consider the natural correspondence
$$
\xymatrix{
M \times T^*C(1) & \ar@{_(->}[l]_-i M \times T^*_{\{t_0\}} C(1) \ar@{->>}[r]^-p & M 
}
$$
and the resulting subset
$$
\sP_{t_0} = p(i^{-1}(\sP)) \subset M.
$$
Note that $\sP_{t_0} \subset M$ is not the naive restriction $\sP|_{t_0} = i^{-1}(\sP) \subset M \times T^*_{\{t_0\}} C(1)$ 
but rather the Hamiltonian reduction along the equation $t= t_0$.

\begin{definition}
A morphism  $\sP{\colon\thinspace}\sL_0 \to \sL_1$ is said to admit a {\em non-characteristic cut} at time $t_0\in (0,1)$ 
if 
$$
\sP_{t_0} \cap \Lambda = \emptyset.
$$
\end{definition}

\begin{example}
If an object $\sL$ is non-characteristic, then its identity morphism $\id_\sL$ admits a non-characteristic cut at any time $t_0\in (0,1)$.
\end{example}

\begin{proposition}\label{prop. cuts}
If a morphism $\sP{\colon\thinspace}\sL_0 \to \sL_1$ admits a non-characteristic cut at some time $t_0 \in (0,1)$, then it factors through the zero object $\zerolag$: 
there is a $2$-simplex in $\Lag^{\dd 0}$ with boundary values
$$
\xymatrix{
\ar[dr]^-{\sP}\ar[d]_-{} \sL_0  & \\
\zerolag \ar[r]_-{} & \sL_1
}
$$
\end{proposition}

\begin{proof}
We will write $\sP'$ for the result of time reparametrizing $\sP$ so that it lies inside of $M\times T^*(0,2)$.
Without loss of generality, we may assume that $\sP'$ admits a non-characteristic cut at time $t_0 = 1$. 

Let us return to the proof of the preceding proposition, and in particular the graph brane
$$
\Gamma_q \subset T^*C(2) = T^*((0, 1) \times (0,2)).
$$
 Consider the Lagrangian $2$-cobordism represented by the translated brane
$$
\sQ = (\sP' \times T^*_{(0,1)} (0,1)) + \Gamma_q \subset M \times T^*C(2) = M \times T^*((0,1) \times (0,2))
$$

Observe that admitting a {non-characteristic cut} is an open condition in the time variable $t_0$.
Thus choosing all of the parameters
 $0 < \tau_0 < \tau_1 < {\frac {1} {2}} < \tau_2 < \tau_3 < 1$ in the construction of the graph brane $\Gamma_q$
 sufficiently close to ${\frac {1} {2}}$, the Lagrangian $2$-cobordism $\sQ$ is non-characteristic.
\end{proof}

\begin{remark}\label{remark: cuts}
We will eventually use not only the above proposition but also the specific $2$-simplex constructed in its proof.

Thanks to the weak Kan property and the fact that $\zerolag$ is a zero object, the above $2$-simplex is equivalent up to a sequence of contractible choices with a $2$-simplex of the more concrete form
$$
\xymatrix{
\ar[dr]^-{\sP}\ar[d]_-{z_{\to \es}} \sL_0  & \\
\zerolag \ar[r]_-{z_{\es\to}} & \sL_1
}
$$
where $\sP$ is factored into a composition of the zero morphisms of Example~\ref{example: zero maps}.
\end{remark}

\subsection{Stabilization}
The main construction of this paper is a concrete geometric stabilization. The following definition encompasses familiar phenomena from stable homotopy theory and homological algebra, lifting the notion of a triangulated category to the $\infty$-categorical setting: 

\begin{definition}[\cite{LurieStab}]\label{defn.stability}
 An $\oo$-category $\sC$ is said to be {\em stable} if it satisfies the following:
\begin{enumerate}
\item There exists a zero object $0\in \sC$.
\item Every morphism in $\sC$ admits a kernel and a cokernel.
\item Any commutative diagram in $\sC$ of the form
$$
\xymatrix{
a\ar[d] \ar[r] & b \ar[d]\\
0 \ar[r] & c
}
$$
is Cartesian if and only if it is cocartesian.
\end{enumerate}
\end{definition}

The third condition states that the diagram presents $a$ as a kernel of $b\to c$
if and only if it presents $c$ as a cokernel of $a\to b$.

We have seen that the empty Lagrangian $\zerolag$ is a zero object of $\Lag^{\ 0}_{\Lambda}(M)$.
To guarantee the existence of kernels and cokernels,
we consider the following symplectic version of a stabilization procedure familiar from algebraic topology.

Consider the symplectic manifold $T^*\R^n$ with fixed support Lagrangian the zero section 
$\R^n \subset T^*\R^n$. We also equip $T^*\R^n$ with its canonical background structures (arising from the fact that it is a cotangent bundle). 
For notational simplicity, we set 
$$\Lag^{\dd n}_{\Lambda}(M) = \Lag^{\dd 0}_{\Lambda \times \R^n}(M\times T^*\R^n),
$$
which for $n=0$ agrees with our previous notation.

Recall that for $x\in \R$, and $\alpha\in \Z + {\frac {1} {2}}$, we have the $\alpha$-graded conormal brane 
$$\sN_{x}[\alpha] \subset T^*\R.
$$ 
Let us focus on the ${\frac {1} {2}}$-graded conormal brane based at $x=0$.
Successively taking the product with it provides  a  tower of functors
\begin{equation}\label{eqn.lag-n}
\xymatrix{
\cdots\ar[r] &  \Lag^{\dd n}_{\Lambda}(M)
 \ar[r] & \Lag^{\dd n+1}_{\Lambda}(M )\ar[r] & \cdots\
 & \sL \ar@{|->}[r] & \sL \times \sN_0[{\frac {1} {2}}].
}
\end{equation}
Note that each functor is a monomorphism of simplicial sets (i.e., an injection at every dimension).

\begin{definition}
Given $ \Lambda\subset M$, 
we define the $\oo$-category $ \Lag_{\Lambda}(M)$ to be the union
$$
 \Lag_{\Lambda}(M) {\colon\thinspace}= \cup_{n \to \oo}  \Lag^{\dd n}_{\Lambda }(M ).
$$
of $\oo$-categories.
\end{definition}

\begin{remark}
Note that this union is again an $\infty$-category, because any inner horn to the union must factor through $\lag^{\dd n}_\Lambda(M)$ for some finite $n$.
\end{remark}

While the above definition is concrete, there is a more universal characterization:

\begin{lemma}
$\lag_\Lambda(M)$ is a homotopy colimit of the sequential diagram~\eqref{eqn.lag-n}. More specifically, the natural map from the homotopy colimit 
$$
	\hocolim_n \lag^{\dd n}_\Lambda(M) \to \Lag_{\Lambda}(M)
$$
is an equivalence of $\infty$-categories.
\end{lemma}

\begin{proof}
This follows from generalities about Reedy model structures and the Joyal model structure. Since sequential diagrams are cofibrant in the Reedy model structure, a sequential (honest) colimit of cofibrations is always a homotopy colimit in a left proper model category. One concludes by observing that monomorphisms are cofibrations in the Joyal model structure for $\infty$-categories, and that the Joyal model structure is left proper. (In fact, even if the Joyal model structure were not left proper, one would only need that the first object in the sequential diagram be cofibrant.)
\end{proof}

The following theorem is our main result. The remaining sections of the paper are devoted to its proof,
and in particular a construction of the required mapping cones.

\begin{theorem} \label{thm. main}
 The $\oo$-category $\Lag_\Lambda(M)$ is stable.
\end{theorem}

Strictly speaking, we do not prove properties (2) and (3) of Definition~\ref{defn.stability}. Rather, we verify three equivalent properties
given by the following theorem of Lurie.

\begin{theorem}[\cite{LurieHigherAlgebra}]\label{stablealternative}
An $\infty$-category $\sC$ is a stable $\infty$-category if and only if the following three properties are satisfied:
\begin{enumerate}
\item There exists a zero object $0\in \sC$.
\item Every morphism in $\sC$ admits a kernel.
\item Let $\Omega{\colon\thinspace} \sC \to \sC$ be the functor induced by sending every object $\sL$ to the kernel of the zero map $0 \to \sL$. Then $\Omega$ is an equivalence.
\end{enumerate}
\end{theorem}

\begin{remark}
While we have used the term {\em stabilization} to refer to the construction of $\Lag_\Lambda(M)$, the same term is used in \cite{LurieHigherAlgebra} as a general procedure for producing stable $\infty$-categories from non-stable ones. $\Lag_\Lambda(M)$ is {\em not} the universal stabilization of $\Lag^{\dd 0}_\Lambda(M)$; for instance, the shift functor already exists in $\lag^{\dd 0}_\Lambda(M)$. 
\end{remark}

\subsection{Space symmetries}\label{sect: isotopies of domain}
Before continuing on to the proof of theorem,
we state here an easily checked invariance of our constructions under isotopies.

Let $\Diff(\R^k)$ denote the topological group of compactly supported diffeomorphisms of $\R^k$. 
It is useful to recognize that $\Diff(\R^k)$ deformation retracts onto the general linear group.
In particular,  $\Diff(\R^k)$ consists of two components, depending on whether an element is orientable.

Let  $P(\R^k)$ denote the Poincar\'e $\oo$-groupoid of $\Diff(\R^k)$. The group structure on 
$\Diff(\R^k)$ turns $P(\R^k)$ into a Picard $\oo$-groupoid.
We will write  $P(\R^k)(n)$ for its set of $n$-simplices, and in particular $P(\R^k)(0)$ for its discrete group of $0$-simplices. 

By construction, 
$P(\R^k)(0)$ naturally acts on $\Lag^{\dd k}_\Lambda(M)$, or in other words, there is a morphism
$$
\xymatrix{
P(\R^k)(0) \ar[r] & \Aut(\Lag^{\dd k}_\Lambda(M))
}
$$
where we write $\Aut$ to denote the Picard $\oo$-groupoid of $\oo$-category automorphisms.

We leave the following 
assertion to the interested reader.

\begin{proposition}\label{prop: diff action}
The natural action of the discrete group $P(\R^k)(0)$ on the $\oo$-category $\Lag^{\dd k}(M)$ lifts to an action of 
the full Picard $\oo$-groupoid $P(\R^k)$.
\end{proposition}

\begin{example}[Isotopy of Domain]\label{example: isotopy of domain}
Here is the simplest  general import of the previous proposition. 

Suppose we begin with an element $g\in P(\R^k)(0)$ 
which  for simplicity we assume represents an orientation-preserving diffeomorphism. Let $p \in P(\R^k)(1)$ represent a path in $\Diff(\R^k)$ from the identity  to the element $g$. 

Suppose we take an object $\sL \in \Lag^{\dd k}(0)$. Then the proposition asserts that there is an invertible morphism
$\sP \in \Lag^{\dd k}_\Lambda(M)(1)$ from $\sL$ to $g(\sL)$. 
\end{example}

\begin{example}
Here is a very special but useful application of the previous proposition (which could easily be deduced in an ad hoc manner). 

Recall that for $x\in \R$, and $\alpha\in \Z + {\frac {1} {2}}$, we have the $\alpha$-graded conormal brane 
$$\sN_{x}[\alpha] \subset T^*\R.
$$ 
In particular, in our stabilization tower, we worked with the ${\frac {1} {2}}$-graded conormal brane based at $x=0$.

The above proposition confirms that our constructions are independent of the choice of $x = 0$. The functor of taking the product with $\sN_{0}[{\frac {1} {2}}]$ is equivalent to taking the product with $\sN_{x}[{\frac {1} {2}}]$ for any other $x\in \R^k$.
\end{example}

We will not use the following assertion (and in fact it follows from our main theorem), but we state it here since it can be deduced rather
formally from
the above proposition.

Fix two objects $\sL_0, \sL_1\in \Lag^{\dd k}(M)(0)$. Using the above proposition, we can assume they are disjoint
$
\sL_0 \cap \sL_1 = \emptyset,
$
and hence their disjoint union provides a well defined object
$$
\sL_0 \coprod \sL_1 \in  \Lag^{\dd k}_\Lambda(M)(0)
$$

\begin{corollary}
Disjoint union of branes provides a natural $\cE_k$-structure on the $\oo$-category $\Lag^{\dd k}(M)$.
Taking $k\to \oo$, we obtain an $\cE_\oo$-structure on the $\oo$-category $\Lag_\Lambda(M)$.
\end{corollary}

Our proof that $\Lag_\Lambda(M)$ is stable shows in particular that disjoint union is the coproduct
in $\Lag_\Lambda(M)$.

\def\Tilt{\mathsf{Tilt}}
\subsection{Tilting}\label{ssec: tilt}
 We present here a simple but particularly useful variation
of Proposition~\ref{prop: diff action} which could be skipped until needed later. It can be thought of as the ``Fourier transform"
of the operation of translation.

Recall that for $x\in \R$, and $\alpha\in \Z + {\frac {1} {2}}$, we have the $\alpha$-graded conormal brane 
$$\sN_{x}[\alpha] \subset T^*\R.$$

More generally, consider an $n$-dimensional submanifold $X\subset \R^k$, a contractible open neighborhood $U\subset X$, and a smooth function $b{\colon\thinspace}X\to \R$ that is strictly positive on $U$ and zero elsewhere. 
Consider the logarithm $\log b{\colon\thinspace}U \to \R$, and for $\beta\in \Z + (k-n)/2$, the corresponding translated conormal brane 
$$
\sT_U[\beta] = \sN_X[\beta] + \Gamma_{\log b} \subset T^*\R^k
$$

\begin{definition}
For $\alpha\in \Z + {\frac {1} {2}}$, a {\em  tilt} of the conormal brane $\sN_x[\alpha]$ is any such brane $\sT_U[\beta] = \sN_X[\beta] + \Gamma_{\log b}$ with $\beta = \alpha - (k-n)/2$.
\end{definition}

The source of the name tilt is that one can obtain  $\sT_U[\beta] $ from the conormal brane $\sN_x[\alpha]$ by a time-dependent Hamiltonian symmetry of $T^*\R^k$.
We leave the following easy lemma to the reader.

\begin{lemma}\label{lemma: tilts}
Time-dependent Hamiltonian symmetries provide a contractible space of identifications of all  {tilts} of the conormal brane
$\sN_x[\alpha]$.
\end{lemma}



\clearpage

\section{Construction of kernel}\label{sect. kernel}
The remainder of the paper is devoted to the proof of Theorem~\ref{thm. main}. 
Fix once and for all a morphism $\sP{\colon\thinspace}\sL_0 \to \sL_1$ between objects $\sL_0, \sL_1\in \Lag_\Lambda(M)$.
We will construct a kernel object $\sK\in  \Lag_\Lambda(M)$ and a morphism $\sk {\colon\thinspace}\sK \to \sL_0$,
then show that the composition $\sP\circ \sk{\colon\thinspace}\sK \to \sL_1$ is homotopic to a zero map. We will refer to the 2-simplex of $\Lag_\Lambda(M)$ realizing this homotopy as the kernel homotopy. The construction of the kernel object, kernel morphism, and kernel homotopy are the subjects of Sections~\ref{sect: kernel object}, \ref{sect: kernel map}, and~\ref{sect: kernel homotopy}, respectively.
We verify the above data satisfies the universal property of the kernel in Section~\ref{sect: universal property}.

Recall that $\Lag_\Lambda(M)$ is the union of the $\oo$-subcategories 
$$
\Lag^{\dd k}_\Lambda(M) = \Lag^{\dd 0}_{\Lambda \times\R^k} (M \times T^*\R^k).
$$
Hence the morphism $\sP{\colon\thinspace}\sL_0 \to \sL_1$ comes from $\Lag^{\dd k}_\Lambda(M)$, for some $k$. It is convenient
 to reindex our objects once and for all, and write $M$ in place of $M \times T^*\R^k$, and $\Lambda$ in place of $\Lambda \times\R^k$. With this understanding, we view $\sP{\colon\thinspace}\sL_0 \to \sL_1$ as a morphism in $\Lag^{\dd 0}_\Lambda(M)$. The kernel $\sK$ will then be an object of $\Lag^{\dd 1}_\Lambda(M)$.

In any stable $\oo$-category, the kernel and cokernel of a morphism differ by a shift (properly speaking, the shift functor is deduced from the axiom that squares are Cartesian if and only if cocartesian). Thus our construction of a candidate kernel $\sK$ just as well
provides a candidate cokernel $co\sK = \sK[1]$, where as usual $[1]$ denotes the shift of grading. We will refer to the underlying Lagrangian submanifold 
$$
C(\sP) \subset M \times T^*\R
$$ of each of these branes as the {\em cone} of $\sP$. Often
in what follows we will focus on the kernel, though all of our arguments have precise analogues for the cokernel. 

\subsection{Kernel object}\label{sect: kernel object}
In this section, we give a specific model of $\sK$ depending on the choice of realization of the  cone $C(\sP)$.
In particular, we will fix once and for all a concrete representative brane
 $$
  \sP\subset M \times T^*C(1)
 $$ 
for the morphism $\sP$ (rather than an equivalence class under time reparametrization). We will make further choices to construct the model of the kernel---in verifying the universal property of the kernel, we will have shown that the space of such choices is contractible.

In what follows, the subscripts $\ell,r$ stand for left and right.

\begin{construction}[Cone]\label{construction: cone}

Fix once and for all a smooth non-decreasing function $\tau_\ell{\colon\thinspace}\R\to \R$ and a smooth non-increasing function $\tau_r{\colon\thinspace}\R\to \R$ 
such that
\[
\tau_\ell(t) = 
 \left\{
 \begin{array}{ll}
 0, & t\leq -1\\
 1, & t\geq 0
\end{array}
\right.
\qquad
\tau_r(t) = 
 \left\{
 \begin{array}{ll}
 1, & t\leq 1\\
 0, & t\geq 2
\end{array}
\right.
\]
Consider the corresponding exact one-forms
\[
\kappa_\ell = d\log \tau_\ell(t) \in \Omega^1(-1, \oo)
\qquad
\kappa_r = d\log \tau_r(t)\in \Omega^1(-\oo, 2).
\]
Finally, consider the exact one-form given by the linear combination
\[
\kappa = -\kappa_\ell + \kappa_r \in  \Omega^1(-1,2).
\]
We will only appeal to the fact that $\kappa$ has the qualitative behavior
\begin{enumerate}
 \item 
 $\kappa$ is identically zero over $[0,1]$,
 \item
$\kappa$ has no zeros outside of $[0,1]$,
 \item
 $\kappa(\partial_t) \to-\infty$ as $t\to -1$,
 \item
 $\kappa(\partial_t) \to-\infty$ as $t\to 2$.
 \end{enumerate}
Now let us return to the Lagrangian brane
 $$
  \sP\subset M \times T^*C(1).
 $$ 
Thanks to its collaring by $\sL_0, \sL_1 \subset M$ at the end points $\{0\}, \{1\}$ of the interval $C(1) = [0,1]$, we can extend $\sP$ to a Lagrangian brane over a longer interval
 \[
  \tilde \sP = (\sL_0 \times T^*_{(-1, 0)}(-1, 0) ) \cup \sP \cup  (\sL_1 \times T^*_{(1, 2)}(1, 2))\subset M \times T^*(-1,2).
 \]
We define the {\em cone} of $\sP$ to be the translated Lagrangian submanifold
$$
C(\sP) =   \tilde \sP + \Gamma_\kappa \subset M \times T^*(-1, 2).
$$
Alternatively, we could restrict $\kappa$ to the intervals $(-1, 0), (1, 2)$, and then form the glued Lagrangian submanifold
$$
C(\sP) =   (\sL_0 \times \Gamma_\kappa|_{(-1, 0)} ) \cup \sP \cup  (\sL_1 \times \Gamma_\kappa|_{(1, 2)} ) \subset M \times T^*(-1, 2).
$$
See Figure~\ref{fig:cone}.

 Observe that $C(\sP)$ deformation retracts onto the original exact brane 
 $$\sP = C(\sP)|_{[0,1]}.
 $$ Thus all homotopical aspects of the geometry
 of $C(\sP)$
 coincide with those of $\sP$. In particular,
it is evident that $C(\sP)$ is a closed exact Lagrangian submanifold of $M \times T^*\R$.
\end{construction}

	\begin{figure}[ht]
		$$
		\xy
		\xyimport(8,8)(0,0){\includegraphics[width=3.3in]{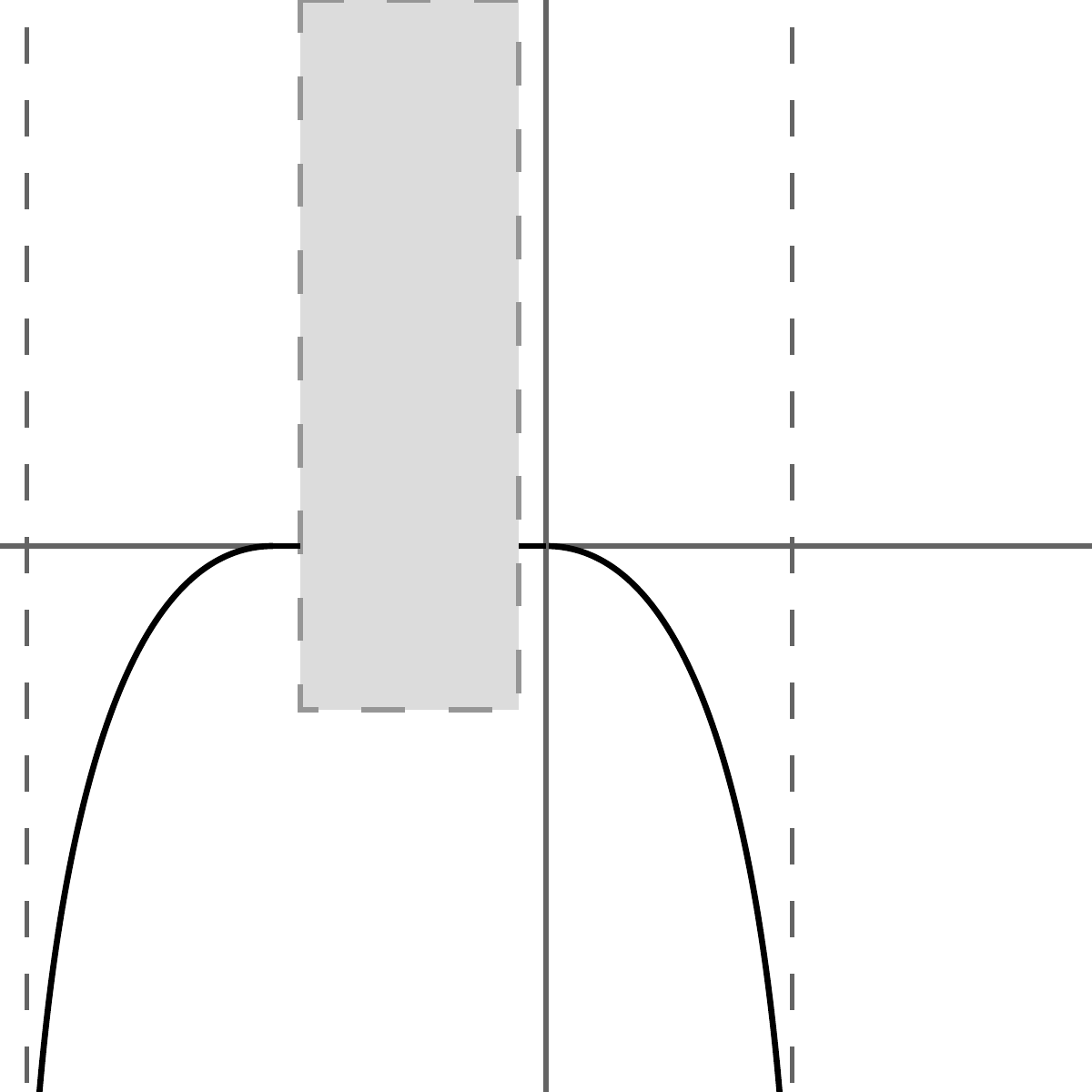}}
	,(4,8)*+!L{\R^\diamond}
	,(8,4)*+!U{\R_x}
	,(3,6)*+{\mbox{\huge $\sP$}}
	,(5.4,2)*+!L{\mbox{\LARGE $\sL_1 \times \kappa_r$}}
	,(0.5,2)*+!R{\mbox{\LARGE $\sL_0 \times (-\kappa_\ell)$}}
	,(0.1,7)*+!R{x=-2}
	,(5.9,7)*+!L{x=1}
		\endxy
		$$
		\begin{image}\label{fig:cone}
		The projection of the cone $C(\sP)$ onto $T^*\R$. 		\end{image}
	\end{figure}

\begin{definition}[Kernel and cokernel]\label{def. kernel}
We define the {\em kernel} 
$
\sK$ of the morphism $\sP$ to be the Lagrangian brane with underlying Lagrangian submanifold the cone 
$$
C(\sP) \subset M \times T^*\R,
$$ and brane structure $(f_{\sK}, \alpha_{\sK}, [c_{\sK}])$ uniquely determined by the fact that its restriction to $\sP \subset C(\sP)$ coincides with the brane structure $(f_{\sP}, \alpha_{\sP}, [c_{\sP}])$ of $\sP$.

The {\em cokernel} $co\sK$ is the shift $co\sK = \sK[1]$. Explicitly, $co\sK$ has the structure $(f_\sK, \alpha_\sK+1, [c_\sK]).$
\end{definition}

\subsection{Kernel morphism}\label{sect: kernel map}

Next we define the kernel morphism $\sk{\colon\thinspace} \sK \to \sL_0$.  

Let us begin with the identity morphism $\id_{\sK}$ represented by the Lagrangian $1$-cobordism
$$
\id_{\sK} = \sK \times T^*_{C(1)} C(1) \subset M \times T^*\R \times T^*C(1).
$$

To construct $\sk$, we will carve $\id_{\sK}$ in the sense of Section~\ref{section.carving} with a function 
$$
h{\colon\thinspace}\R \times C(1)\to \R.
$$
We define $h$ as follows: Fix once and for all a small $\epsilon >0$, and smooth non-decreasing functions $\tau_s{\colon\thinspace}\R\to \R, \tau_t{\colon\thinspace}C(1)\to \R$ 
such that
$$
\tau_s(s) = 
 \left\{
 \begin{array}{ll}
 0, & s\leq -1 + \epsilon\\
 1, & s\geq 0
\end{array}
\right.
\qquad
\tau_t(t) = 
 \left\{
 \begin{array}{ll}
 0, & t\leq \epsilon\\
 1, & t\geq {\frac {1} {2}}
\end{array}
\right.
$$
Form the product $\tau(s, t) = \tau_s(s)\tau_t(t){\colon\thinspace}\R\times C(1) \to \R$, and finally consider the function
$$
h(s, t) = - \log(1- \tau(s, t))
$$
over its natural domain of definition 
\[
D = ((-1, 2) \times (0,1)) \setminus ([0, 2) \times [{\frac {1} {2}}, 1)) \subset \R \times C(1).
\]

\begin{definition}[Kernel morphism]
We define the {\em kernel morphism} $\sk{\colon\thinspace}\sK\to \sL_0$ to be the carved Lagrangian brane
\[
\sk =   \id_{\sK} + \Gamma_h \subset M \times T^*\R \times T^* C(1).
\]
\end{definition}

\begin{figure}[ht]
		\[
		\xy
		\xyimport(8,8)(0,0){\includegraphics[height=3in]{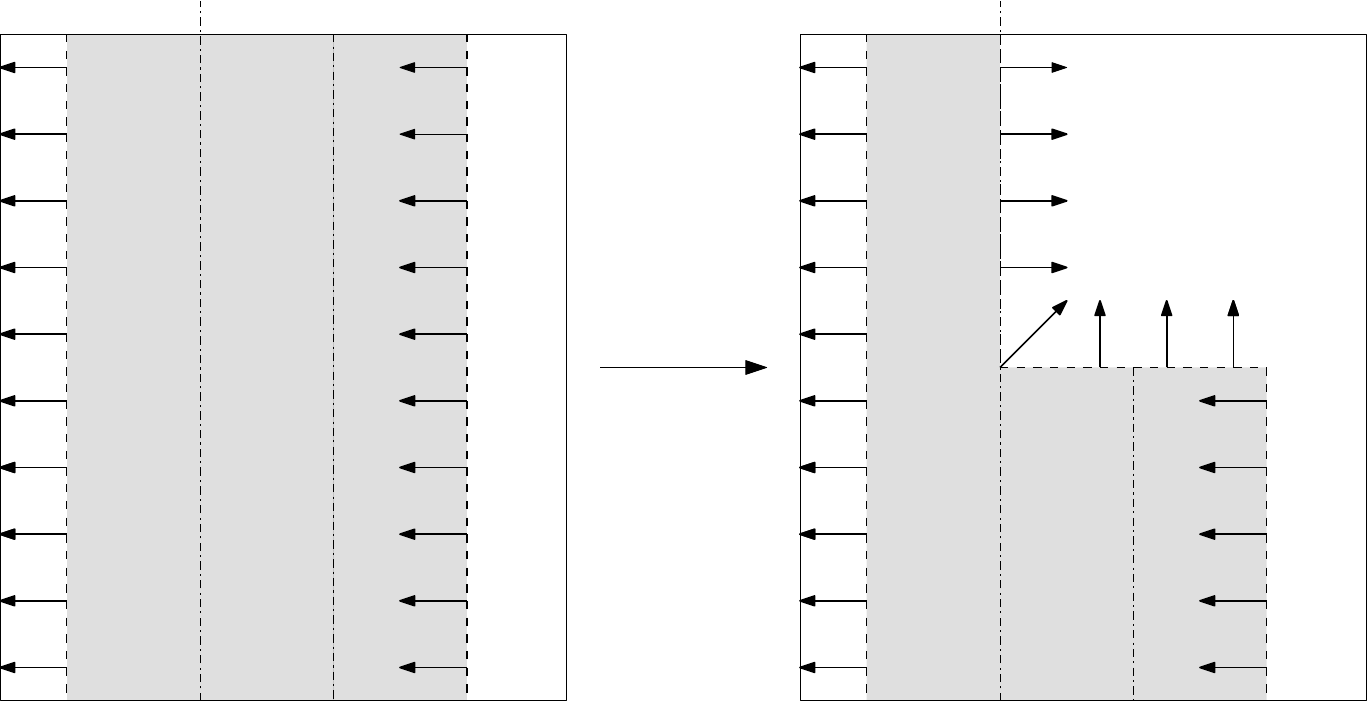}}
	,(1.8,-0.4)*+{(a)}
	,(6.4,-0.4)*+{(b)}
	,(0.8,4)*+{ \mbox{\huge $\sL_0$} }
	,(1.6,4)*+{\mbox{\huge $\sP$}}
	,(2.4,4)*+{\mbox{\huge $\sL_1$}}
	,(5.5,4)*+{\mbox{\huge $\sL_0$}}
	,(6.2,2.5)*+{\mbox{\huge $\sP$}}
	,(7,2.5)*+{\mbox{\huge $\sL_1$}}
	,(4,4.2)*+{\mbox{\huge $+\Gamma_h$}}
	,(1.2,8)*+!R{s=-1}
	,(5.9,8)*+!R{s=-1}
		\endxy
		\]
		\begin{image}\label{figure: kernel map}
		(a) The projection of $id_{\sK}$ onto $\R \times C(1)$. (b) shows the effect of carving by $dh$; the resulting cobordism $\sk = id_{\sK} + \Gamma_{dh}$ is a morphism from $\sK$ to (a tilt of) $\sL_0$.
		\end{image}
\end{figure}

\begin{remark}
The kernel morphism $k$ is indeed a morphism to a tilt (in the sense of Section~\ref{ssec: tilt}) of the product brane 
$$
\sL_0 \times \sN_0[{\frac {1} {2}}] \subset M \times T^*\R.
$$
To view this as a morphism to $\sL_0$ itself, note that $\sL_0 \times \sN_0[{\frac {1} {2}}]$  is the image of $\sL_0$ under the map
$$
\Lag^{\dd 0}_\Lambda(M) \to \Lag^{\dd 1}_\Lambda(M),
$$ and so equivalent to $\sL_0$ inside of $\Lag_\Lambda(M)$. On the other hand, thanks to Lemma~\ref{lemma: tilts}, there is a contractible space of such tilts, and they are all equivalent
to $\sL_0 \times \sN_0[{\frac {1} {2}}]$ inside of $\Lag^{\dd 1}_\Lambda(M)$. Thus we can unambiguously interpret $k$ as a morphism to $\sL_0$.
\end{remark}

\begin{remark}
As we noted in Remark~\ref{remark.corners}, one can choose instead to construct a model for $\sk$ whose domain does not contain a corner, and instead smooth it out. We have just chosen $\tau$ (and hence $h$) to produce a corner because the present formulas are easier to write down.  
\end{remark}

\subsection{Kernel homotopy}\label{sect: kernel homotopy}
Now that we have the kernel object $\sK$ and the kernel map $\sk{\colon\thinspace} \sK \to \sL_0$, we will construct a homotopy from the composition
\[
\xymatrix{
 \sK \ar[r]^-\sk & \sL_0 \ar[r]^-\sP &  \sL_1
}
\] 
to a zero map. Recall the  representative zero morphisms $z_{\to \es}, z_{\es\to}$ constructed in Example~\ref{example: zero maps}.

We will construct a diagram of $2$-simplices
$$
\xymatrix{
\ar[dr]\ar[d]_-{z_{\to \es}} \sK \ar[r]^-\sk & \sL_0\ar[d]^-\sP \\
\zerolag \ar[r]_-{z_{ \es\to}} & \sL_1
}
$$
extending our previous constructions. Thanks to Proposition~\ref{prop. cuts} (see also Remark~\ref{remark: cuts}), it suffices to find a $2$-simplex
$$
\xymatrix{
\ar[dr]_-\sh  \sK \ar[r]^-\sk & \sL_0\ar[d]^-\sP \\
 & \sL_1
}
$$
such that the boundary edge $\sh$ admits a non-characteristic cut. But our construction of $\sk$ is a map to a tilt of $\sL_0$, rather than to $\sL_0^\dd$. As such we will instead construct a 2-simplex with boundary edges
\[
\xymatrix{
 \sK \ar[r]^-\sk \ar[dr]_\sh& {\sL_0\times\Gamma_g}\ar[d]^-{\sP\times \Gamma_g} \\
 & \sL_1\times\Gamma_g.
}
\]

\begin{remark}
Here, $\Gamma_g \subset T^*\R$ is the graph brane associated to the tilting function
$$
g{\colon\thinspace}(-1, 0)\to \R
\qquad
g(s) = -\log(\tau_\ell(s) (1-\tau_s(s))).
$$
Note that by Lemma~\ref{lemma: tilts}, we lose no generality in replacing $\sP{\colon\thinspace} \sL_0 \to \sL_1$ by a tilt.
\end{remark}

\begin{proposition}[Kernel homotopy]\label{prop:kernel composes}
There is a Lagrangian $2$-cobordism
$$
\sH \subset M \times T^*\R \times T^*C(2) = M\times T^*\R \times T^*((0,2) \times (0,1))
$$
with collared boundaries over the facets
$$
\xymatrix{
\sH|_{ (0,1) \times \{0\}} =   k
&
\sH|_{ (1,2) \times \{0\}} =  (\sP\times \Gamma_g) 
&
\sH|_{ (0,2) \times \{1\}}  = \sh
}
$$
such that the Lagrangian $1$-cobordism $\sh$ admits a non-characteristic cut.
\end{proposition}

\begin{remark}
It will be important that we understand the precise shape of $\sH$ to later verify that the resulting kernel diagram
is universal. Therefore, while we formulate the following as a proposition, we will eventually
refer to properties established  during its proof.
\end{remark}

\begin{proof} We break up the proof into  the construction of $\sh$, then that of $\sH$.

\begin{construction}[of the morphism $\sh$]\label{construction: h}
The construction of $\sh$ is summarized in Figure~\ref{figure: kernel_homotopy_sh}. It proceeds in four steps.

\begin{figure}
		\[
		\xy
		\xyimport(8,8)(0,0){\includegraphics[width=6in]{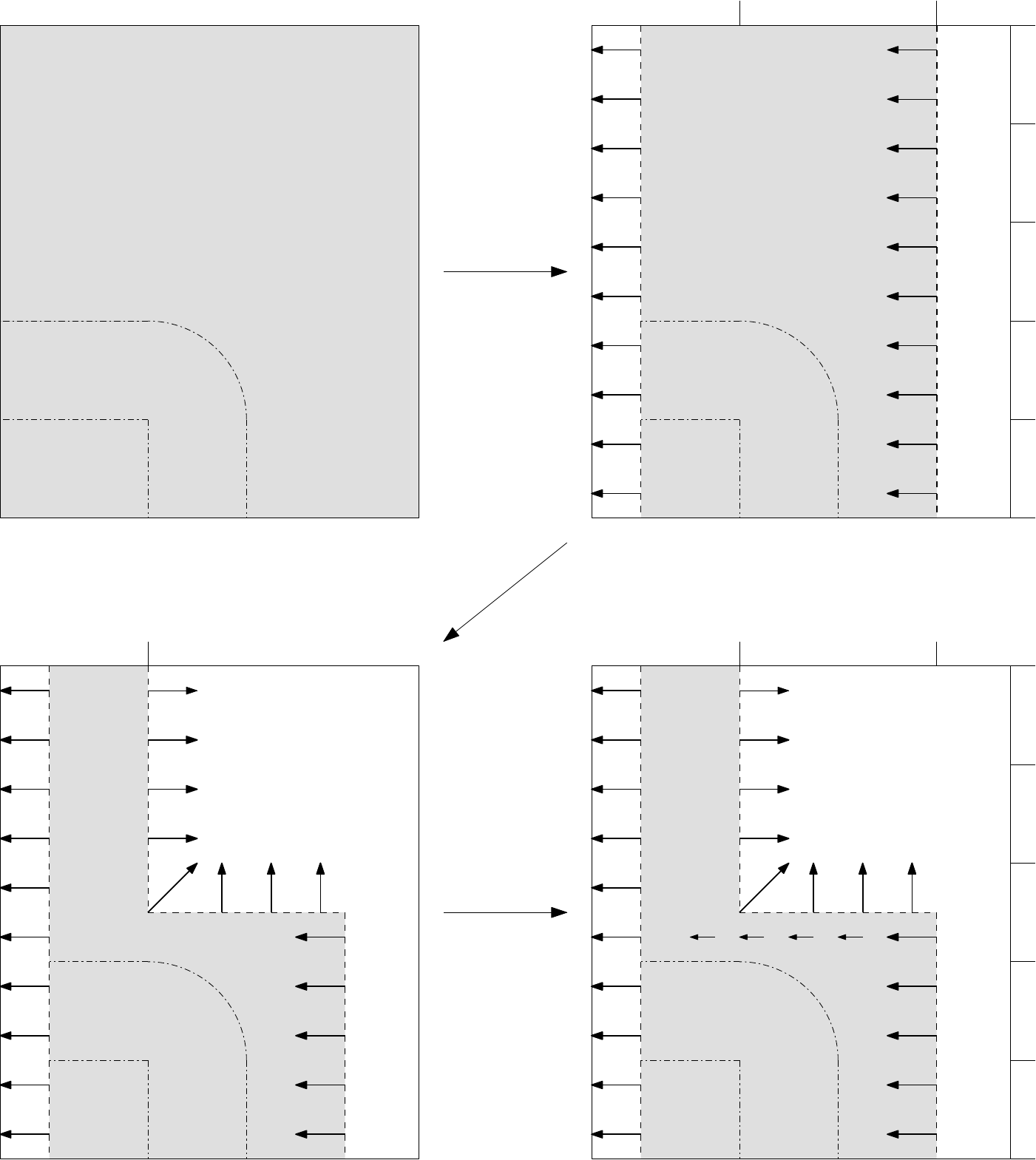}}
	,(1.5,4.4)*+!U{(a)}
	,(6.0,4.4)*+!U{(b)}
	,(1.5,0.0)*+!U{(c)}
	,(6.0,0.0)*+!U{(d)}
	,(3.9,6.2)*+!D{ \mbox{\huge $+\Gamma_{\kappa}$} }
	,(3.8,4.1)*+!D{ \mbox{\huge $+\Gamma_{h}$} }
	,(3.8,1.8)*+!D{ \mbox{\huge $+\Gamma_{g}$} }
	,(0.6,4.8)*+!L{ \mbox{\huge $\sL_0$} }
	,(0.6,5.4)*+!L{ \mbox{\huge $\sP$} }
	,(1.4,4.8)*+!L{ \mbox{\huge $\sP$} }
	,(2.2,6.6)*+!L{ \mbox{\huge $\sL_1$} }
	,(0.6,0.4)*+!L{ \mbox{\huge $\sL_0$} }
	,(0.6,1)*+!L{ \mbox{\huge $\sP$} }
	,(1.3,0.4)*+!L{ \mbox{\huge $\sP$} }
	,(2.1,1.3)*+!L{ \mbox{\huge $\sL_1$} }
	,(5.2,4.8)*+!L{ \mbox{\huge $\sL_0$} }
	,(5.2,5.4)*+!L{ \mbox{\huge $\sP$} }
	,(5.9,4.8)*+!L{ \mbox{\huge $\sP$} }
	,(6.3,6.6)*+!L{ \mbox{\huge $\sL_1$} }
	,(5.2,0.3)*+!L{ \mbox{\huge $\sL_0$} }
	,(5.2,1.0)*+!L{ \mbox{\huge $\sP$} }
	,(5.9,0.3)*+!L{ \mbox{\huge $\sP$} }
	,(6.6,1.3)*+!L{ \mbox{\huge $\sL_1$} }
	,(7.8,0)*+!D!L{ \mbox{$t=-2$} }
	,(7.8,0.65)*+!D!L{ \mbox{$t=-1$} }
	,(7.8,1.35)*+!D!L{ \mbox{$t=0$} }
	,(7.8,2.05)*+!D!L{ \mbox{$t=1$} }
	,(7.8,2.75)*+!D!L{ \mbox{$t=2$} }
	,(7.8,3.4)*+!D!L{ \mbox{$t=3$} }
	,(7.8,4.4)*+!D!L{ \mbox{$t=-2$} }
	,(7.8,5.05)*+!D!L{ \mbox{$t=-1$} }
	,(7.8,5.75)*+!D!L{ \mbox{$t=0$} }
	,(7.8,6.45)*+!D!L{ \mbox{$t=1$} }
	,(7.8,7.15)*+!D!L{ \mbox{$t=2$} }
	,(7.8,7.8)*+!D!L{ \mbox{$t=3$} }
	,(5.75,7.9)*+!D!R{ \mbox{$s=-1$} }
	,(7.3,7.9)*+!D!R{ \mbox{$s=1$} }
	,(5.75,3.5)*+!D!R{ \mbox{$s=-1$} }
	,(7.3,3.5)*+!D!R{ \mbox{$s=1$} }
	,(1.2,3.5)*+!D!R{ \mbox{$s=-1$} }
		\endxy
		\]
 \begin{image}\label{figure: kernel_homotopy_sh}
 A summary of the construction of $\sh$. (a) results from the rotation construction applied to $\sP$. (b) results from carving by the one-form $\kappa(s,t)$. (c) results from carving by the one-form $dh$. (d) is a drawing of $\sh$, where the addition of $dg$ ensures that $\sh$ admits a non-characteristic cut.
 \end{image}
\end{figure}

{\em Step (a)}. Recall the rotation construction from Example~\ref{example: rotation}. We perform this on the cobordism $\tilde \sP$ to obtain a Lagrangian in $M \times T^*\R \times T^*\R$. The result is depicted in Figure~\ref{figure: kernel_homotopy_sh}(a), and inherits a canonical brane structure from that of $\sP$.

{\em Step (b)}. In defining the cone $C(\sP)$, we utilized a 1-form $\kappa \in \Omega^1(\R)$ in Construction~\ref{construction: cone}. Let us consider $\kappa = \kappa(s, t) \in \Omega^1(\R \times \R)$
as a 1-form with no dependence on $t$, and carve the brane from (a) by the 1-form $\kappa= \kappa(s, t)$. The result is a brane in
$
M \times T^*(-2,1) \times T^*\R.
$

{\em Step (c).} Now we carve by the 1-form $dh$ utilized in Section~\ref{sect: kernel map}. Again we remember the resulting brane structure.

{\em Step (d).} Finally, 
to ensure that $\sh$ admits a non-characteristic cut, 
we translate by a 1-form $dg$ which we now define. Fix $\epsilon>0$, and define a function $g_1{\colon\thinspace} \R \to \R$ such that
\[
g_1(s) = \left\{
\begin{array}{ll}
1, & s \leq -2 + \epsilon \\
0, & s \geq 1-\epsilon
\end{array}
\right.
\]
with $g_1$ strictly decreasing on $(-2+\epsilon,1-\epsilon)$. Define  a bump function $g_2{\colon\thinspace} C(1) \to \R$ such that
\[
\xymatrix{
supp(g_2) \subset [{\frac {1} {2}} -3\epsilon, {\frac {1} {2}} - \epsilon] \qquad & \qquad \text{$g_2>0$ on $({\frac {1} {2}} -3\epsilon, {\frac {1} {2}} - \epsilon)$. }
}
\]
Now define $g{\colon\thinspace}\R \times C(1) \to \R$ by
\[
g(s,t) = g_1(s)g_2(t)
\]
and translate the brane from (c) by $dg$.

In principle we are done. But we will need some concrete notation for later on. Define the linear diffeomorphism
\[
\psi{\colon\thinspace} (-2,3) \to (0,1)
\]
and transport the brane from (d) to a brane in
$
M \times T^*\R \times C(1)
$
via the diffeomorphism
$$
\xymatrix{
M \times T^*\R \times (-2,3) \ar[r] & M \times T^*\R \times C(1)  &
   (p,s,t)  \ar@{|->}[r] & (p,s,\psi(t)).
}
$$
We call this Lagrangian brane $\sh$, with all the brane structure induced from the previous steps. It clearly admits a non-characteristic cut at $\psi({\frac {1} {2}} - 2\epsilon)$.
\end{construction}

%
%
\begin{construction}[of kernel homotopy $\sH$]\label{construction: H}
It is clear that $\sK \circ \sP$ can be obtained by carving a portion of $\sh$. The 2-cube
\[
\sH \subset M \times T^*\R \times C(2)
\]
will be a cobordism realizing this carving. We depict the carving in Figure~\ref{figure: kernel_homotopy_sH_2}.

\begin{figure}
\includegraphics[width=6in]{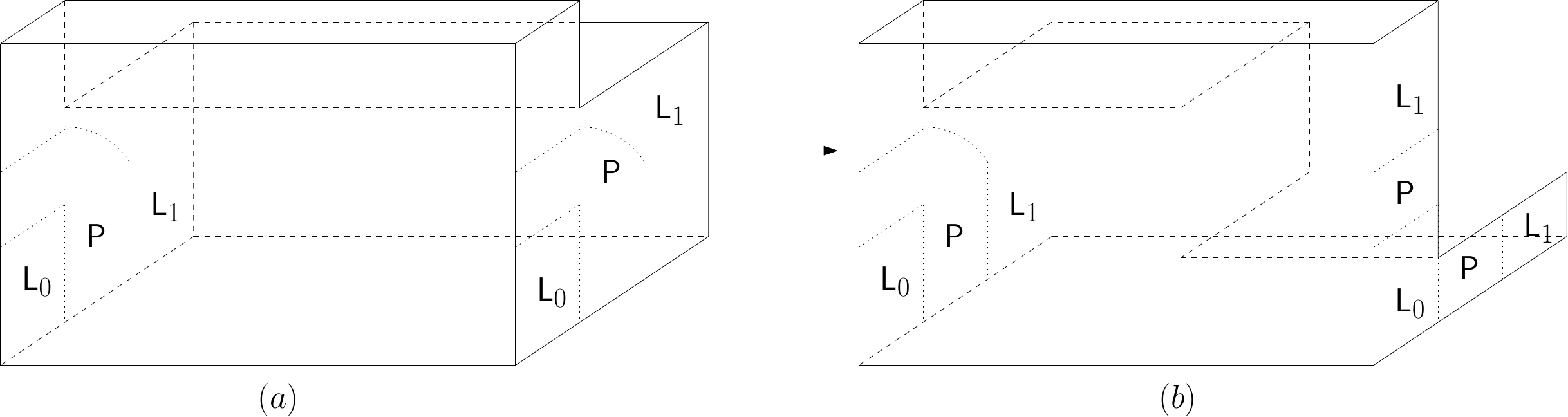}

 \begin{image}\label{figure: kernel_homotopy_sH_2}
 A summary of the construction of $\sH$. (a) is the Lagrangian $p^*(\sh)$. (b) is obtained by carving by the 1-form $d(\tau f)$. Note that the rightmost face of (b) is equal to the composition $\sP \circ \sK$, where $\sP$ has been tilted. The drawing is not to scale, and we have not drawn any vector fields to indicate the behavior of $\sH$ near infinity. \end{image}
\end{figure}

Explicitly, the linear diffeomorphism 
\[
\phi{\colon\thinspace} (0,2) \to (0,1) \subset \R.
\]
induces a map
\[
\begin{array}{ccc}
p{\colon\thinspace}M \times T^*\R \times C(2) &\to & M \times T^*\R \times C(1) \\
   (p,s,(x,t)) &\mapsto & (p,s,\phi(t))
\end{array}
\]
by which we can pull back $\sh$ to obtain
\[
p^*(\sh) \subset M \times T^*\R \times C(2).
\]
This Lagrangian is the identity 2-morphism of $\sh$---up to scaling in the time coordinate, one can think of it simply as
\[
\sh \times T^*_{(0,1)}(0,1).
\]

Recall the function
\[
h(s,t) = -\log(1 - \tau(s,t))
\]
used in Section~\ref{sect: kernel map}. Note that if we define the $x$-independent function
\[
f = h - \phi^*\psi_*h
\]
then the Lagrangian
\[
p^*(\sh) + \Gamma_{df}
\]
is the identity 2-morphism of $\sK \circ \sP$. The goal is thus to parametrize between $f$ and the function 0---we do this by introducing a function $\tau$.

Choose a smooth, non-decreasing function $\chi{\colon\thinspace} (0,1) \to \R$ with boundary conditions
\[
\chi(x) = \left\{ \begin{array}{ll}
0, & x \leq \epsilon \\
1, & x\geq {\frac {1} {2}}
\end{array}
\right.
\]
and define the product
\[
\chi f (x,s,t) = \chi(x) f(s,t).
\]

We define
\[
\sH =p^*(\sh) + \Gamma_{d(\chi f)}.
\]
As usual, the brane structure on $\sH$ is induced by the process of carving and by pull-backs. From the comments above it is clear that $\sH$ satisfies the boundary conditions of Proposition~\ref{prop:kernel composes}. 
\end{construction}

The above constructions complete the proof of Proposition~\ref{prop:kernel composes}.
\end{proof}

\begin{remark}[Stabilization commutes with kernels]\label{remark: stabilizing kernels}
Given a morphism $\sP\in \Lag^{\dd 0} (\sL_0,\sL_1)$, we have defined the kernel to be an object 
$
\sK(\sP) \in \Lag^{\dd 1}.
$
As constructed,  $\sK(\sP)^{\dd},\sK(\sP^{\dd}) \in \Lag^{\dd 2}$ are not identical objects. For example, setting
\[
\Lag^{\dd 2}(M) = \Lag^{\dd 0}(M \times T^*\R^2) = \Lag^{\dd 0}(M \times T^*\R_{s_1} \times T^*\R_{s_2}),
\] 
we see that $\sK(\sP^{\dd})$ is a product with $T^*_{\{0\}}\R_{s_1}$ while $\sK(\sP)^{\dd}$ is a product with $T^*_{\{0\}}\R_{s_2}$.

However, there is a Hamiltonian isotopy of $\R^2$ which rotates $(-2,1) \subset \R_{s_1}$ into $(-2,1) \subset \R_{s_2}$, showing that $\sK(\sP)^{\dd}$ and $\sK(\sP^{\dd})$  are indeed equivalent. In fact, this rotation produces a prism in $\Lag^{\dd 1}$ realizing a homotopy between the two diagrams
\[
\xymatrix{
\ar[dr]_{\sh(\sP)^{\dd}}  \sK(\sP)^{\dd} \ar[r]^{\sk(\sP)^\dd} & \sL_0^{\dd 2}\ar[d]^{\sP^{\dd 2}} \\
 & \sL_1^{\dd 2}
}
\hspace{5em}
\xymatrix{
\ar[dr]_{\sh(\sP^\dd)}  \sK(\sP^{\dd}) \ar[r]^{\sk(\sP^\dd)} & \sL_0^{\dd 2} \ar[d]^{\sP^{\dd 2}} \\
 & \sL_1^{\dd 2}.
}
\]
Here the rightmost edge $\sP^{\dd 2}$ is fixed, since we can choose a rotation which fixes the origin of $ \R^2$, and $\sP^{\dd2}$ lives over the origin.
\end{remark}

\clearpage
\section{Proof of universal property}\label{sect: universal property}

In this section, we prove that the kernel diagram (object, morphism, and homotopy) from the preceding section satisfies the ($\infty$-categorical) universal property of a kernel. In Section~\ref{sect: reduce universal property}, we describe the outline of the proof, reducing it to three key constructions given in Sections~\ref{sect: d'F}, \ref{sect: construction of {\kernelfunctor^i}}, and \ref{sect: I}. 
In Section~\ref{sect: anatomy}, we introduce a useful figure for reference throughout the constructions.
Throughout what follows, we continue with the setup of the preceding section. In particular, we assume $\sP$ is a morphism in $\Lag^{\dd 0}$ without loss of generality.

\subsection{Formal reductions}\label{sect: reduce universal property}

\begin{definition}
Fix an $\infty$-category $\sC$ and a simplicial set $\sD$. A {\em diagram in $\sC$ indexed by $\sD$} is a map of simplicial sets $\sD \to \sC$. Given a diagram $\sD\to \sC$,  the {\em overcategory $\sC_\sD$}  is the unique simplicial set such that 
	$$\sSet(\Delta^n, \sC_\sD) = \sSet_\sD(\Delta^n \star \sD , \sC)$$
where $\Delta^n \star \sD$ denotes the join of $\Delta^n$ and $\sD$, and  $sSet_\sD$ denotes the simplicial set of maps $\Delta^n \star \sD \to \sC$ that restrict to the diagram $\sD \to \sC$.
\end{definition}

\begin{example}
Let $\sD = \Delta^0$ be a point and let the diagram $\sD \to \sC$ specify an object $\sL$ of $\sC$. Then an $n$-simplex in $\sC_\sD$ is an $(n+1)$-simplex in $\sC$ whose final vertex is the object $\sL$.
\end{example}

For details on the join of two simplicial sets and diagram categories, we refer the reader to the Appendix.

\begin{notation}
We denote simplices of $\sC_\sD$ with a subscript $\sD$.  The map
$$
\xymatrix{
\sC_\sD \ar[r] & \sC
&
 \sX_\sD \ar@{|->}[r] &  \sX
}$$
denotes the forgetful map.
\end{notation}

\begin{definition}[$\sD$]
Given the morphism $\sP{\colon\thinspace} \sL_0 \to \sL_1$ let $\sD$ be the diagram
 in $\Lag$ given by
\[
	\xymatrix{
	&\sL_0 \ar[d]^{\sP}	\\
	0 \ar[r]_{z_{\es \to}}
		&\sL_1
	}
\]
\end{definition}

\begin{remark}\label{remark: Lag_D}
An object $\sX_\sD\in \Lag^{\dd i}_\sD$ represents a diagram in $\Lag^{\dd i}$ of the form
\[
	\xymatrix{
	\sX	\ar[r] \ar[d] \ar[rd]
			&\sL_0^{\dd i}	\ar[d]^{\sP^{\dd i}} \\
			0	\ar[r]	
		&\sL_1^{\dd i}.
	}
\]

A higher simplex $\sX_\sD{\colon\thinspace}\Delta^m \to  \Lag^{\dd i}_\sD$ also represents a diagram in $\Lag^{\dd i}$
of a similar form. 
For example, an edge $\sX_\sD{\colon\thinspace}\Delta^1\to  \Lag^{\dd i}_\sD$  with underlying vertices $\sX_0, \sX_1\in  \Lag^{\dd i}$
represents the union of two tetrahedra that heuristically represent compositions
$$
\xymatrix{
\sX_0 \to \sX_1 \to 0 \to \sL_1
&
 \sX_0 \to \sX_1 \to \sL_0 \to \sL_1
} $$ 
glued along their common triangle.  In general, an $m$-simplex $\sX_\sD{\colon\thinspace}\Delta^m \to \Lag^{\dd i}_\sD$ 
 with underlying vertices $\sX_0,\ldots,  \sX_m\in  \Lag^{\dd i}$
 represents a diagram in $\Lag^{\dd i}$ obtained by gluing two $(m+2)$-simplices along a common $(m+1)$-simplex.  
\end{remark}

In the preceding section, we constructed an object $\sK_\sD = \sK_\sD(\sP) \in \Lag_{\sD}^{\dd 1}$ given by the commutative square
\[
	\xymatrix{
	\sK \ar[r]^{\sk} \ar[d]_{z_{\to \es}} \ar[dr]_{\sh}	&\sL_0^{\dd 1} \ar[d]^{\sP^{\dd1}}	\\
	0 \ar[r]_{z_{\es \to}}
		&\sL_1^{\dd1}.
	}
\]
Note that for $i=0,1,2,\ldots,$ the objects 
\[
\sK_\sD(\sP^{\dd i}), \sK_\sD^{\dd i} = \sK_\sD(\sP)^{\dd i} \in \Lag_\sD^{\dd i+1}.
\]
are naturally equivalent by Remark~\ref{remark: stabilizing kernels}, so there is no ambiguity in defining the object $\sK_\sD \in \Lag_\sD$.

Our ultimate aim is to establish the following:

\begin{theorem}\label{thm. kernels exist}
$\sK_\sD$ is a terminal object of the overcategory $\Lag_\sD$. That is, $\sK_\sD$ is a kernel of the map $\sP$.
\end{theorem}

The theorem is a simple corollary of the following lemma:

\begin{lemma}\label{lemma: {\kernelfunctor^i}}
For every object $\sX_\sD \in \Lag_{\sD}^{\dd i}$, the stabilization map 
\[
\xymatrix{
\dd{\colon\thinspace} \Lag_{\sD}^{\dd i}(\sX_\sD,\sK_\sD^{\dd i}) \ar@{->}[r]& \Lag_\sD^{\dd i+1}(\sX_\sD,\sK_\sD^{\dd i+1})
}
\]
induces the zero map on homotopy groups.
\end{lemma}

\begin{proof}[Proof of Theorem~\ref{thm. kernels exist}]
To show the object  $\sK_\sD$ in the $\oo$-category $\Lag_\sD$ is terminal is equivalent 
to showing that the space of morphisms $\Lag_\sD(\sX_\sD,\sK_\sD)$ is contractible, for any object $\sX_\sD \in \Lag_\sD$. 
By definition of $\Lag$, we have
$
\Lag_\sD (\sX_\sD, \sK_\sD) = \colim_j \Lag^{\dd i+j}_\sD(\sX_\sD^{\dd j},\sK_\sD^{\diamond i+ j})
$
as Kan complexes. (We have fixed $i$ such that $\sX_\sD$ originates in $\Lag^{\dd i}_\sD$.)

Consider $f{\colon\thinspace} \partial \Delta^m \to \Lag_\sD(\sX_\sD,\sK_\sD)$. 
Since $\partial \Delta^m$ is a finite simplicial set, $f$ factors through $\Lag^{\dd i+j}_\sD$, for some $j$. Lemma~\ref{lemma: {\kernelfunctor^i}} shows $f$ can be filled by a map $\tilde f {\colon\thinspace} \Delta^m \to  \Lag_\sD^{\dd i+j+1}$.
\end{proof}

We will prove the lemma based on two propositions. To state the propositions, it is convenient to recall the following:

\begin{definition}
Given an $\oo$-category $\sC$, we denote by $\sC^\cone= \sC \star *$ the {\em terminal cone category} given by attaching a terminal vertex $*$ to $\sC$. 
We have the natural inclusion
$
\iota{\colon\thinspace} \sC \to \sC^\cone
$
of the {\em base category} into the terminal cone category.

Given $\sX$ a diagram in  $\sC$, we denote by $\sX^\cone$ the {\em terminal cone diagram} in $\sC^\cone$ obtained by taking the join of $\sX$ and $*$. 
\end{definition}

\begin{proposition}\label{prop. {\kernelfunctor^i}}
For each $i = 0,1,2,\ldots$, there exists a  functor
\[
\xymatrix{
{\kernelfunctor^i}{\colon\thinspace} (\Lag_\sD^{\dd i})^\cone \ar[r] & \Lag_\sD^{\dd i+1}
}
\] 
such that the composition
\[ 
\xymatrix{
  \Lag_{\sD}^{\dd i} \ar[r]^-\iota
  &(\Lag_\sD^{\dd i})^\cone \ar[r]^-{F^i} 
  & \Lag_\sD^{\dd i+1}   
}
\]
is the stabilization functor 
\[
\xymatrix{
\sX_\sD \ar@{|->}[r] & \sX_\sD^{\dd}.
}\]
\end{proposition}

\begin{proposition}\label{prop. I}
For each $i = 0,1,2,\ldots$, the functor $F^i$ of Proposition~\ref{prop. {\kernelfunctor^i}} can be chosen so that when
evaluated on the terminal cone point $*\in (\Lag_\sD^{\dd i})^\cone$,
$$
\xymatrix{
{\kernelfunctor^i}(*) = \sK_\sD(\sP^{\dd i}) \in \Lag_\sD^{\dd i+1},
}$$
and for each $i = 1,2,3,\ldots$, the following morphism is invertible
$$
\xymatrix{
{\kernelfunctor^i}(\sK_\sD(\sP)^{\dd i-1} \to *) \in \Lag_\sD^{\dd i+1}(F^i(\sK_\sD(\sP)^{\dd i-1}), \sK_\sD(\sP^{\dd i}))
}$$
where $\sK_\sD(\sP)^{\dd i-1} \to *$ is the unique morphism.
\end{proposition}	

The proofs of these propositions will occupy Sections~\ref{sect: d'F}, \ref{sect: construction of {\kernelfunctor^i}}, and \ref{sect: I}. For now we prove the lemma assuming the propositions.

\begin{proof}[Proof of Lemma~\ref{lemma: {\kernelfunctor^i}}.]
Recall that in Section~\ref{ssect: zero object} we defined the mapping space $\sC(x,x')$ for any $\infty$-category $\sC$. In the present proof, for convenience we will take a different (but weakly equivalent) model of the mapping space: Namely, we take the $m$-simplices of  $\sC (x, x')$  to be the $(m+1)$-simplices of $\sC$ whose face opposite the $(m+1)$st vertex is a degenerate face at $x$, and whose $(m+1)$st vertex is $x'$. 

Since $\Lag_\sD^{\dd i}(\sX_\sD,\sK^{\dd i}_\sD)$ and $\Lag_\sD^{\dd i+1}(\sX_\sD,\sK_\sD^{\dd i+1})$ are Kan complexes, we merely need to solve the extension problem
\[
\xymatrix{
\partial \Delta^m \ar[d] \ar[r]^-f & \Lag_{\sD}^{\dd i} (\sX_\sD, \sK_\sD^{\dd i}) \ar[d]^-{\diamond} \\
\Delta^m \ar@{-->}[r]^-{\tilde f}    & \Lag_\sD^{\dd i+1}(\sX_\sD,\sK_\sD^{\dd i+1}).
}
\]

By definition, $f$ yields a map $\partial \Delta^{m+1} \to \Lag^{\dd i}_\sD$, which in turn induces a canonical map $f^\cone {\colon\thinspace} (\partial \Delta^{m+1})^\cone \to (\Lag_\sD^{\dd i})^\cone$. 
Applying the functor $F^i$ of Proposition~\ref{prop. {\kernelfunctor^i}}, we observe that ${\kernelfunctor^i}(f^\cone)$ is an {\em outer} horn in $\Lag_\sD^{\dd i+1}$, with the face opposite the $(m+2)$nd vertex deleted. 

By Proposition~\ref{prop. I}, the edge between the $(m+1)$st and $(m+2)$nd vertices is invertible. Hence we can exchange the ordering of the two vertices, obtaining an {\em inner} horn corresponding to a $(m+2)$-simplex with the $(m+1)$st face deleted. (Strictly speaking, we must post-compose with the inverse of the edge to obtain an inner horn whose boundary agrees with the outer horn.) 

By the weak Kan property of $\Lag_\sD^{\dd i +1}$, this inner horn can be filled. The resulting $(m+1)$st face is a map $\tilde f{\colon\thinspace} \Delta^m \to \Lag_\sD^{\dd i+1}(\sX_\sD,\sK_\sD^{\dd i+1})$ whose boundary is $f$.
\end{proof}

\begin{remark}[The need for Proposition~\ref{prop. I}]\label{remark: F^i is a lifting problem}
Let $\sX_\sD {\colon\thinspace} \Delta^m \to \Lag_{\sD}^{\dd i}$ be an $m$-simplex. 
Then the functor ${\kernelfunctor^i}$ will map the  $(m+1)$-simplex given by the join $\sX_\sD^\cone
= \sX_\sD \star *$ to an $(m+1)$-simplex $F^i(\sX_\sD^\cone)$ with terminal vertex ${\kernelfunctor^i}(*) = \sK_\sD(\sP)^{\dd i+1}$ and initial opposite face $\sX_\sD^{\dd 1}$.

In other words,  to each $m$-simplex in $\Lag_{\sD}^{\dd i}$, the functor ${\kernelfunctor^i}$ compatibly produces an $m$-simplex in $(\Lag_\sD^{\dd i+1})_\sK$ solving the lifting problem
\[
\xymatrix{
 \Delta^m \ar[d]_{\sX_\sD} \ar@{-->}[r]  & (\Lag_\sD^{\dd i+1})_\sK	\ar[d]  \\
 \Lag_{\sD}^{\dd i} \ar[r]_{\dd} & \Lag_\sD^{\dd i+1}.
}
\]
If $\Lag_\sD$ and $(\Lag_\sD)_\sK$ were Kan complexes (rather than weak Kan complexes), such a lift would suffice to show that $\sK$ is a terminal object of $\Lag_\sD$ by standard topological arguments. However, since we are dealing with $\oo$-categories (rather than groupoids), we need Proposition~\ref{prop. I} regarding the invertibility of ${\kernelfunctor^i} (\sK_\sD(\sP)^{\dd i} \to *)$. This makes the proof of Lemma~\ref{lemma: {\kernelfunctor^i}} possible.
\end{remark}

%
%
%
%

\subsection{Anatomy of lift}\label{sect: anatomy}
In this interlude between arguments, we introduce some notation and conventions for what follows.
 Since many of our constructions will take place in $\R^4$, we hope these will ease the presentation of our proofs. 
 In particular, we introduce Figure~\ref{figure: {\kernelfunctor^i} projection} which will be frequently referenced during our constructions. We will proceed assuming we have constructed $F^i$ and develop some pictures of its structure. These
 will guide us in its construction.

Throughout what follows, we will adopt a notational shorthand: We will represent a simplex $\sX{\colon\thinspace} \Delta^m \to \Lag$ by concatenation of its vertices, $\sX_0 \sX_1 \cdots \sX_m$. The higher morphisms represented by $\sX$ are to be understood.

As discussed in Remark~\ref{remark: Lag_D}, an $m$-simplex $\sX_\sD{\colon\thinspace}\Delta^m\to \Lag^{\dd i}_\sD$ represents a diagram in $\Lag^{\dd i}$ of the form
\[
	\xymatrix{
	\sX	\ar[r] \ar[d] \ar[rd]
			&\sL_0^{\dd i}	\ar[d]^{\sP^{\dd i}} \\
			0	\ar[r]	
		&\sL_1^{\dd i}.
	}
\]
This commutative diagram should not be taken too ally: There are more vertices, edges, and higher simplices than naively drawn.  In this pictorial shorthand, any edge drawn from $\sX$ to another vertex should be interpreted as a terminal cone from $\sX$ to that vertex.

To construct the functor ${\kernelfunctor^i}$, we must compatibly construct  a diagram of the form
\[
	\xymatrix{
	\sX \ar[rrd] \ar[ddr] \ar[dr]
		&
		&
		\\
		&\sK \ar[r] \ar[d] \ar[rd]
		&\sL_0	^{\dd i+1}\ar[d]^{\sP^{\dd i+1}} \\
		&0	\ar[r]	
		&\sL_1^{\dd i+1}.
	}
\]

This diagram comprises a pair of $(m+3)$-simplices in $\Lag^{\dd i+1}$ glued along a common $(m+2)$-simplex.
If $\sX_0,\ldots,  \sX_m\in  \Lag^{\dd i}$ denote the underlying vertices of $\sX_\sD$, then
the two $(m+3)$-simplices are  built upon compositions 
 \[
 \xymatrix{
 \sX_0 \cdots \sX_m \sK  0  \sL_1 & 
 \sX_0 \cdots \sX_m\sK   \sL_0  \sL_1.
 }
 \]
The common $(m+2)$-simplex is
built upon a composition
 \[
 \xymatrix{
 \sX_0 \cdots \sX_m \sK   \sL_1.
 }
 \]
As with our pictures, we will denote such compositions by the respective expressions
 \[
 \xymatrix{
 \sX \sK  0  \sL_1 & 
 \sX \sK   \sL_0  \sL_1
 & 
  \sX \sK   \sL_1 
 }
 \]
with the understanding that $\sX$ is a general $m$-simplex.

 Let us now return to the geometry of $\Lag$  and examine what such simplices encode.
 Recall that an $(m+3)$-simplex of $\Lag^{\dd i+1}$ represents a cobordism
 in 
 \[
 M \times T^*\R^{i+1}_s \times T^*\R_t \times T^*\R^{m+2}_x
 \]
where $\R^{i+1}_s$,   $\R_t $, and $\R^{m+2}_x$ denote the stabilizing, time, and space directions.

In order to draw pictures, let us project onto the final two space coordinates
\[
\xymatrix{
M \times T^*\R^{i+1}_s \times T^*\R_t \times T^*\R^{m+2}_x \ar[r] &  \R_{x_{m+1}} \times \R_{x_{m+2}}.
}
\]
Then we can depict the cobordisms represented by the $(m+3)$-simplices $\sX \sK  0  \sL_1$ and  $\sX \sK   \sL_0  \sL_1$ as living over squares,
and that represented by the common $(m+2)$-simplex  $\sX \sK   \sL_1$ as living over a common edge. 
This is depicted in Figure~\ref{figure: {\kernelfunctor^i} projection}.

\begin{figure}[ht!]
\includegraphics[height=3in]{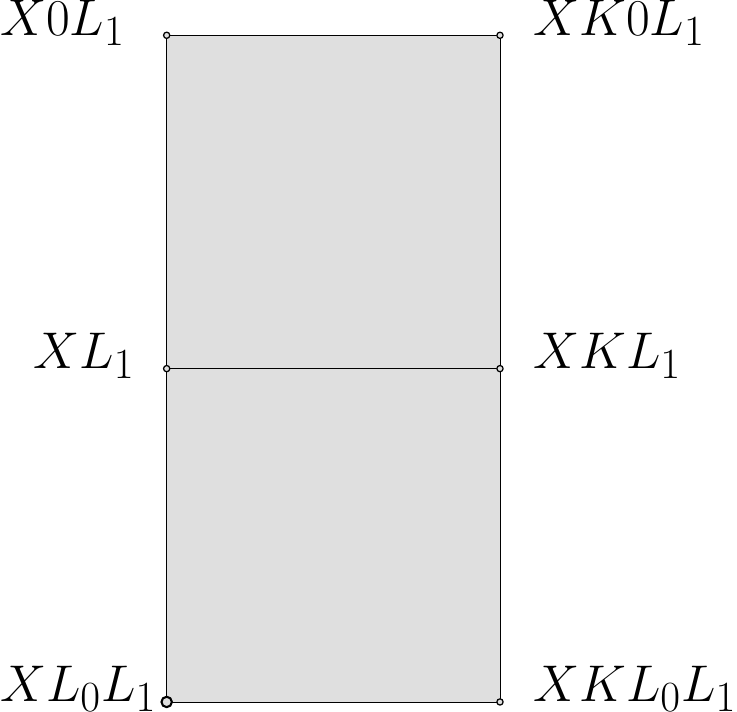}
\begin{image}\label{figure: {\kernelfunctor^i} projection}
A schematic picture of the cobordism comprising $F^i(\sX^\cone_\sD)$ projected onto the final two space coordinates $\R_{x_{m+1}} \times \R_{x_{m+2}}.$ 
The horizontal direction is the $x_{m+1}$-direction and the vertical is the $x_{m+2}$-direction; the lower left vertex $\sX \sL_0 \sL_1$ is placed at $(0,0)$, and the other vertices are at integral points.
\end{image}
\end{figure}

Looking at Figure~\ref{figure: {\kernelfunctor^i} projection}, 
along the lefthand edge $x_{m+1}=0$, we see that $\kernelfunctor^i(\sX_{\sD}^\cone)$ is collared by the original diagram $\sX_\sD$. 

Along the righthand edge $x_{m+1}=1$, we see that $\kernelfunctor^i(\sX_{\sD}^\cone)$ is collared by a composition:
it represents a map from $\sX$ to the kernel diagram $\sK_\sD$. To produce $\kernelfunctor^i(\sX_\sD^\cone)$, a cobordism living over this rectangle, is to produce a homotopy parametrizing between the two edges.

Throughout what follows, we will refer to Figure~\ref{figure: {\kernelfunctor^i} projection} as a guide. 
Our first step, taken in Section~\ref{sect: d'F}, will be to construct the map from $\sX$ to the kernel
diagram $\sK_\sD$, producing the righthand edge $x_{m+1}=1$.

\subsection{Map to kernel}\label{sect: d'F}
In what follows, we assume without loss of generality that $i=0$. Accordingly we will simply denote the functor $\kernelfunctor^i$ by $\kernelfunctor$.

Recall that an $m$-simplex $\sX_\sD{\colon\thinspace} \Delta^m \to \Lag^{\dd 0}_\sD$
with underlying $m$-simplex $\sX{\colon\thinspace}\Delta^m\to \Lag^{\dd 0}$ represents a diagram
\[
	\xymatrix{
	\sX	\ar[r] \ar[d] \ar[rd]
			&\sL_0^{\dd 0}	\ar[d]^{\sP^{\dd 0}} \\
			0	\ar[r]	
		&\sL_1^{\dd 0}.
	}
\]
Given the above data, we will construct here a natural map 
\[
\xymatrix{
\sX \ar[r] &  \sK
}
\]
or more precisely, a functor
\[
\xymatrix{
F_1{\colon\thinspace}\Lag_\sD^{\dd 0} \ar[r] & \Lag_\sK^{\dd 1}.
}\]
The notation $F_1$ reflects the fact that this will yield the boundary value of the functor $F$ along the edge 
\[
\xymatrix{
\sX\sK0\sL_1	\ar@{-}[d]\\
\sX\sK\sL_1	\ar@{-}[d]\\
\sX\sK\sL_0 \sL_1
}
\]
given by $x_{m+1} = 1$ in Figure~\ref{figure: {\kernelfunctor^i} projection}.
This should be thought of as composing $F_1(X) = (\sX\to \sK)$ with the kernel diagram
\[
	\xymatrix{
	\sX \ar[dr]
		&
		&
		\\
		&\sK \ar[r] \ar[d] \ar[rd]
		&\sL_0	^{\dd 1}\ar[d]^{\sP^{\dd 1}} \\
		&0	\ar[r]	
		&\sL_1^{\dd 1},
	}
\]
In alternative but equivalent terms, to construct the functor  $F_1$, we must functorially solve the lifting problem
\[
\xymatrix{
	\Delta^m \ar@{-->}[r]  \ar[d]_-{\sX_\sD} &	\Lag^{\dd 1}_\sK	\ar[d]	\\
 \Lag^{\dd 0}_\sD \ar[r]^-{\dd} & \Lag^{\dd 1}_\sD.
}
\]
	
\begin{construction}[of the functor $F_1$ on an $m$-simplex $\sX_\sD$]\label{construction: map to kernel}
The construction is summarized in Figure~\ref{figure: map to kernel}. We follow the figure's steps (a) through (d).

	\begin{figure}
	\includegraphics[width=5in]{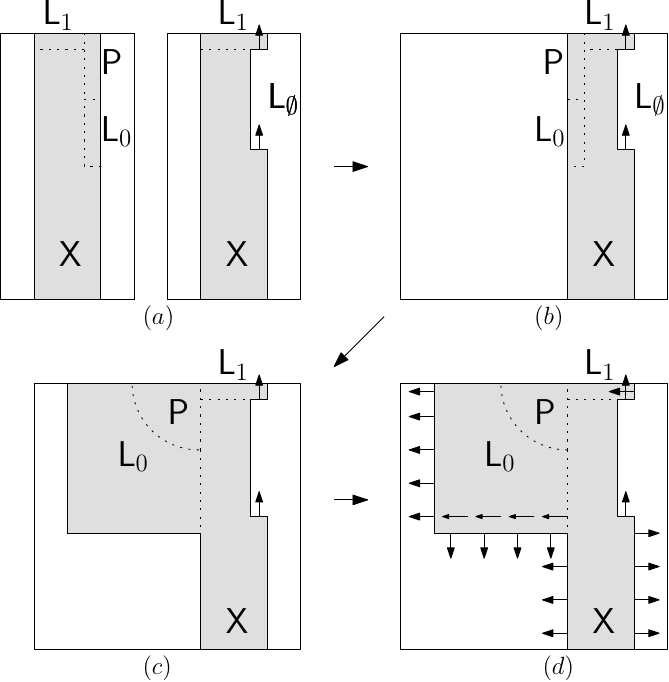}
	\begin{image}\label{figure: map to kernel}
	A depiction of the steps in the construction of $F_1(\sX_\sD)$. We begin in (a) with the two  cobordisms given by the diagram $\sX_\sD$. After gluing and rescaling (a), we obtain the cobordism in (b). We then obtain (c) by rotating the portion of (b) collared by $\sL_0$ and $\sP$. Finally, to obtain (d), we carve (c) by a one-form $d\eta$.
	\end{image}
	\end{figure}

(a) Let $\sX_\sD$ be a $m$-simplex in $\Lag^{\dd 0}_\sD$, and let the vertices of $f(\sX)$ be $\sX_0, \ldots, \sX_m \in \Lag^{\dd 0}$.  As noted in Remark~\ref{remark: Lag_D}, $\sX_\sD$ corresponds to a diagram in $\Lag^{\dd 0}$ obtained by gluing two $(m+2)$-simplices along a common face.  (One of these simplices contains the vertex $\sL_0$, while the other contains the vertex 0.)  Hence $\sX_\sD$ yields the data of two cubical Lagrangians living inside $M  \times T^*\R_t \times T^*\R^{m+1}$.  
We depict these two cubical cobordisms in Figure~\ref{figure: map to kernel}(a). 
The cubes have been projected onto the plane spanned by the coordinates $t$ and $x_{m+1}$. 

(b) The simplex of $\sX_\sD$ containing the vertex $\sL_0$ is collared along one facet by the cobordism $\sP$. The other cube has a face collared by a cobordism factoring through the empty Lagrangian $\sL_\emptyset$. The two cubes share a common face---the collaring condition says we can smoothly glue the two cubes along this common face. This yields the cube depicted in (b), which we have re-scaled to make the cube narrower. We now think of this as a Lagrangian living in 
\[
M \times T^*\R_s \times T^*\R_t \times T^*\R^m_{x},
\]
treating the $x_{m+1}$ direction as the stabilizing $s$ direction.

(c) To make the construction explicit, assume we have chosen a model so that the cobordism depicted in (b) lives above the open rectangle
\[
(0,1) \times (-4,0) \subset \R_s \times \R_t.
\]
Further assume that this model is collared by the object $\sL_0$ along the open edge
\[
\{0\} \times (-2,-1)
\]
and by the morphism $\sP$ along the open edge
\[
\{0 \} \times (-1,0).
\]
The cube is labeled accordingly in Figure~\ref{figure: map to kernel}(b).

To obtain the cube in (c), we perform the rotation construction taking the edge
\[
\{0\} \times (-2,0)
\]
to the edge
\[
(-2,0) \times \{0\}
\]
by rotating clockwise about the origin.

(d) To obtain the cube in (d), we translate by a 1-form $d\eta=d \eta_0 + d\eta_1.$ To define $\eta_0$, let $E$ be the union of three edges
\[
(\{-2\} \times (-2,\infty)) \cup ((-2,0)\times\{-2\}) \cup (\{0\} \times (-2,-\infty)) \subset \R^2 = \R_s \times \R_t.
\]
Let $\epsilon>0$ and define $N_+$ to be the closed set
\[
N_+ {\colon\thinspace}= \{ (e_s + s, e_t + t) \,\, | \,\, (e_s,e_t) \in E \text{ and } s,t \in [0,\epsilon].\}
\]
Choose a smooth function $f{\colon\thinspace} \R_s \times \R_t \to \R$ such that
\begin{enumerate}
\item
	$
	f|_E = 1, \qquad f|_{\partial N_+ \backslash E} = 0
	$, 
\item
	$f$ is locally constant on $\R^2 \backslash N_+$,
\item
	$f$ has no critical points on $N_+ \backslash \partial N_+$, and
\item
	${ \frac {\partial f} {\partial s}} < 0$ on $N_+ \backslash \partial N_+$.
\end{enumerate}
Treating $f$ as a function on $\R_s \times \R_t \times \R^{m}_x$ with dependence only on $t$ and $s,$ we define
\[
\eta_0 = -\log (1-f)
\]

Now we define $\eta_1$. Let $u(t){\colon\thinspace} \R \to \R$ be a $C^\infty$ function such that 
\[
u|_{(-\infty,1-\epsilon]} = 0, \qquad u|_{[1,\infty)}=1
\]
and
\[
\on{Crit}(u) \cap (1-\epsilon,1) = \emptyset.
\]
Then we define $\eta_1{\colon\thinspace} \R_s \times \R_t \times \R^{m}_x \to \R$ to be the function
\[
\eta_1 = \log { \frac {1 - u(s) u(t + 2-\epsilon)} {1 - u(s) u(-t-\epsilon)}}.
\]
with dependence only on $t$ and $s$. Defining $\eta {\colon\thinspace}= \eta_0 + \eta_1$, we obtain the cobordism in step (d) by translating the cobordism in (c) by $\Gamma_\eta$.

We have chosen a model in which the edge $s=1$ is collared by the empty Lagrangian $L_\emptyset$ at $t=-1$, so that this translation does not yield a singular Lagrangian. The resulting brane structure of (d) is the obvious induced brane structure. 

The careful reader will notice that for $i >0$, the same construction yields a map to $\sK(\sP^{\dd i})$, not to $\sK(\sP)^{\dd i}$. 
But the observation from Remark~\ref{remark: stabilizing kernels} tells us we can always post-compose with a rotation taking the $s_{i+1}$-axis to the $s_1$-axis, and this is how we define $\kernelfunctor^i$ for $i>0$.

This ends the construction.
\end{construction}

Finally, as we have not affected the directions $x_1,\ldots,x_m$ in the simplex $\sX$, the cobordism $\kernelfunctor(\sX_\sD^\cone)$ is collared in the $x_1,\ldots,x_m$ coordinates by the appropriate facets. So we have

\begin{lemma}
The construction above is compatible with face and degeneracy maps.
\end{lemma}

\subsection{Proof of Proposition~\ref{prop. {\kernelfunctor^i}}}\label{sect: construction of {\kernelfunctor^i}}
Let us return to the organizational framework of Figure~\ref{figure: {\kernelfunctor^i} projection}.
It depicts that we must  construct a cobordism
\[
\kernelfunctor^1(\sX_\sD^\cone) \subset M \times T^*\R_s \times T^*\R_t \times T^*\R^{m+2}_x
\]
living over the open rectangle
\[
R = (0,1) \times (0,2) \subset \R_{x_{m+1}} \times \R_{x_{m+2}}.
\]

Furthermore, we have the following prescribed boundary values.
First, along the labelled edge
\[
\xymatrix{
\sX0\sL_1	\ar@{-}[d]\\
\sX\sL_1	\ar@{-}[d]\\
\sX\sL_0 \sL_1
}
\]
given by $x_{m+1} = 0$, we have the cobordism given by the initial data $\sX_\sD$ itself.

Along the labelled edge
\[
\xymatrix{
\sX\sK 0\sL_1	\ar@{-}[d]\\
\sX\sK\sL_1	\ar@{-}[d]\\
\sX\sK\sL_0 \sL_1
}
\]
given by $x_{m+1} = 1$, we have the cobordism $F_1(\sX_\sD)$ constructed in the previous section.

Here is a rough summary of our approach to the above challenge. First, we construct a cobordism living
over the triangle 
\begin{equation}\label{triangle T}\tag{$T$}
\xymatrix{
\sX0\sL_1	\ar@{-}[d] & \\
\sX\sL_1	\ar@{-}[d] & \sX\sK\sL_1 \ar@{-}[ul]\ar@{-}[dl]\\
\sX\sL_0 \sL_1 &
}
\end{equation}
with vertices given by the coordinates $(0,0), (0,2), (1,1)$.
Then we extend it over the two missing triangles
\begin{equation}\label{other triangle 1}\tag{$T_1$}
\xymatrix{
\sX 0 \sL_1 \ar@{-}[r] &\sX\sK 0 \sL_1    \\
& \sX\sK\sL_1 \ar@{-}[u]\ar@{-}[ul]\\
}
\end{equation}
\begin{equation}\label{other triangle 2}\tag{$T_2$}
\xymatrix{
& \sX\sK\sL_1 \ar@{-}[d]\ar@{-}[dl]\\
\sX\sL_0 \sL_1 \ar@{-}[r]& \sX\sK \sL_0 \sL_1
}
\end{equation}
with vertices given by the coordinates $(0,2), (1,2), (1,1)$ and $(0,0), (1,0), (1,1)$ respectively.

Finally, it will be straightforward to see that the above constructions yield the desired boundary values and are compatible with face and degeneracy maps.

\subsubsection{Cobordism over triangle~\ref{triangle T}}

Recall the rectangle $R = (0,1) \times (0,2)$, and the triangle $T\subset R$ with 
vertices given by the coordinates $(0,0), (0,2), (1,1)$.

We will construct a cobordism 
\[
\sW \subset M \times T^*\R_s \times T^*\R_t \times T^*(\R^{m}_x \times R)
\]
living over the rectangle $R$, and then pull back $\sW$ to a cobordism 
\[
\sW' = \phi^*(\sW) \subset M \times T^*\R_s \times T^*\R_t \times T^*(\R^{m}_x \times T)
\]
living over the triangle $T$ by the diffeomorphism 
\[
\xymatrix{
\phi{\colon\thinspace} T \ar[r] & R
&
\phi(x_{m+1},x_{m+2}) = (x_{m+1},{\frac {x_{m+2}-1} {1 - x_{m+1}}} + 1).
}\]

One might be concerned about the complicated behavior of $\phi$ near the boundary, but this will be of no consequence: $\sW$ will be collared by the zero section near the boundary. 

\begin{construction}[of cobordism $\sW$]\label{construction: sW}

We will construct $\sW$ as the union of three Lagrangians
\[
\sW = \sW_{[0,{\frac {1} {3}}]} \cup \sW_{[{\frac {1} {3}},{\frac {2} {3}}]} \cup \sW_{[{\frac {2} {3}},1]}.
\]
The subscripts $[a,b]$ refer to the fact that we will construct $\sW$ over the intervals $[a,b] \subset \R_{x_{m+1}}$ piece by piece. We will glue the pieces together to obtain one cobordism $\sW$ living over $[0,1] \subset \R_{x_{m+1}}$.

We describe the construction of each of these pieces, stating lemmas along the way and leaving the proof of these lemmas til afterward.

\subsubsection{Set-up for construction of $\sW$. The cobordism $\sS(\sX_\sD)$.}

Let $\sX \to \sK$ be the map given in construction~\ref{construction: map to kernel}, and let $\sh{\colon\thinspace} \sK \to \sL_1$ be the simplex constructed in Construction~\ref{construction: h}.
Consider the composition 
\[
\xymatrix{
\sX \ar[r] & \sK \ar[r]^{\sh} & \sL_1
}
\]
as given in Definition~\ref{def. higher composition}. For example if $\sX_{\sD}$ is an $m$-simplex in $\Lag^{\dd 0}_D$ (so it yields a pair of $(m+2)$-simplices in $\Lag^{\dd 0}$), we extend $\sh$ by the zero section
\[
T^*_{\R^{m+1}_x}\R^{m+1}_x
\]
before concatenating with the cobordism $\sX \to \sK$.

For concreteness we fix a model of $\sX \to \sK \to \sL_1$ as follows. We assume that the cobordism $\sX \to \sK$ lives above the time interval $(-4,0) \subset \R_t$ as in Construction~\ref{construction: map to kernel}, and that $\sh$ lives above the interval $(0,5)$. (That is, we choose a model such that $\sh$ has been translated by two units in the time direction. For instance it allows a non-characteristic cut at $t= 2.5 - \epsilon$.) 
This composition is depicted in Figure~\ref{figure: SXD}(a). 

\begin{construction}[of $\sS(\sX_\sD)$]
Now we note that there is an isotopy of domain, depicted in Figure~\ref{figure: SXD}, which takes $(\sX \to \sK) \circ \sh$ and produces a morphism $\sX \to \sL_1$ living over a rectangular domain in $\R_s \times \R_t$. (Compare to Figure~\ref{figure: map to kernel}(b).) 
\begin{figure} 
$$
\xy
\xyimport(8,8)(0,0){\includegraphics[height=3in]{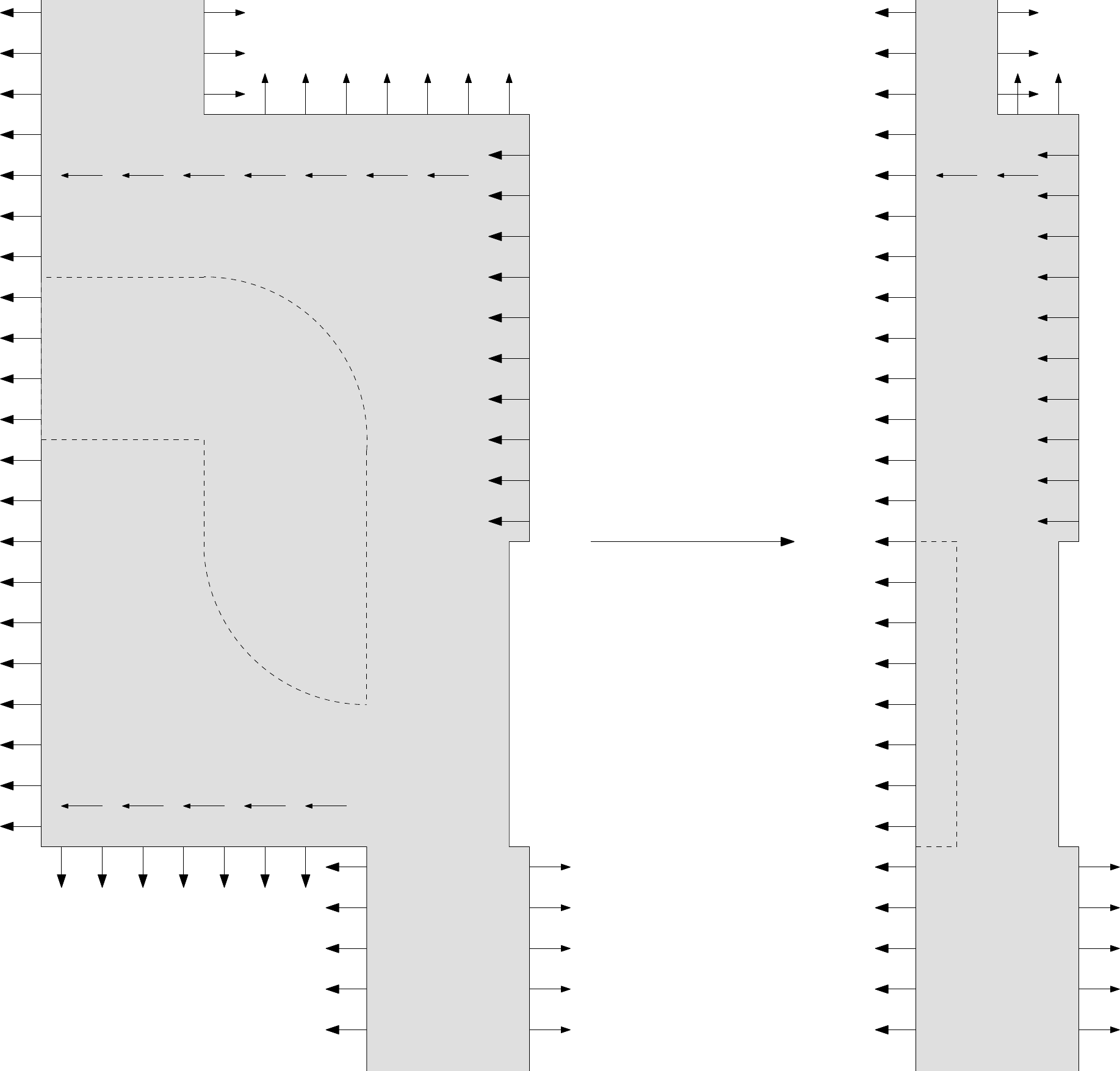}}
	,(2,0)*+!U{(a)}
	,(7,0)*+!U{(b)}
	,(2,4.5)*+{\sP}
	,(6.7,3)*+{\sP}
	,(1,7.5)*+{\sL_1}
	,(3.1,5)*+{\sL_1}
	,(1,3)*+{\sL_0}
	,(4,3)*+{\zerolag}
	,(7.8,3)*+{\zerolag}
	,(7.2,0.5)*+{\sL_0}
	,(3.2,0.5)*+{\sL_0}
	,(6.9,7.5)*+{\sL_1}
\endxy
$$
	\begin{image}\label{figure: SXD}
	An isotopy taking (a) the cobordism $(\sX \to \sK) \circ \sH$ to (b) the cobordism $\sS(\sX_\sD)$. The horizontal coordinate is the $s$ direction, and the vertical coordinate is the $t$ direction.
	\end{image}
\end{figure}
We denote the resulting morphism $\sX \to \sL_1$ by $\sS(\sX_\sD)$.
\end{construction}

\begin{remark}\label{remark: the essence of W}
Note that a copy of $\sX_{\sD}$ is contained in the Lagrangian $\sS(\sX_{\sD})$ by construction: As we travel in the $s$ direction of $\sS(\sX_{\sD})$, we witness exactly the evolution of $\sX_{\sD}$. The issue is that $s$ is a {\em stabilizing} direction; hence we want $\sW$ to be a cobordism which takes this stabilizing direction and `turns' it into a direction in which morphisms are allowed to evolve, namely the direction $x_{m+2}$. If there were no inherent difference between a stabilizing $s$ direction and a morphism $x$ direction in our set-up, the construction of $\sW$ would be trivial.
\end{remark}

\subsubsection{Construction of $\sW_{[{\frac {2} {3}},1]}$}
We parametrize the isotopy from Figure~\ref{figure: SXD} by the interval
\[
[{\frac {2} {3}},1] \subset \R_{x_{m+1}}.
\]
By Example~\ref{example: isotopy of domain}, we can realize the isotopy by a cobordism
\[
\underline{\sW}_{[{\frac {2} {3}}, 1]} \subset (M \times T^*\R_s) \times T^*\R_t \times T^*\R_{x}^{m+1}.
\]
Setting
\[
\pi_{m+2}{\colon\thinspace} \R_{s} \times \R_t \times \R_x^{m+2} \to \R_{x_{m+2}}
\]
to be projection onto the last coordinate, we define $\sW_{[{\frac {2} {3}}, 1]}$ by the pullback
\[
\sW_{[{\frac {2} {3}}, 1]} 
{\colon\thinspace}= 
\pi_{m+2}^* \underline{\sW}_{[{\frac {2} {3}},1]}
=
\underline{\sW}_{[{\frac {2} {3}},1]} \times T^*_{\R_{x_{m+2}}} \R_{x_{m+2}}
.
\]
We direct the isotopy so that $\sW_{[{\frac {2} {3}},1]}$ is collared by 
\[
\sX \sK \sL_1 \times T^*_{\R_{x_{m+2}}}\R_{x_{m+2}}
\]
along $x_{m+1} = 1$, and by
\[
\sW_{{\frac {2} {3}}} {\colon\thinspace}= \sS(\sX_{\sD}) \times T^*_{\R_{x_{m+2}}} \R_{x_{m+2}}
\]
along $x_{m+1} = {\frac {2} {3}}$.

\subsubsection{Construction of $\sW_{[{\frac {1} {3}},{\frac {2} {3}}]}$}\label{sect: sW1323}
This is done by carving $\sW_{[{\frac 2 3},1]}$. First we define a subset
\[
A \subset \R_t \times \R_s \times \R_{x_{m+1}} \times \R_{x_{m+2}}
\]
by $A = A' \coprod A''$, where
\[
A' {\colon\thinspace}= 
[-4+\epsilon, 5-\epsilon]_t 
\times
(-\infty, 1-\epsilon)_s 
\times
(-\infty, {\frac 1 3} + \epsilon]_{x_{m+1}}
\times
[2-\epsilon,\infty)_{x_{m+2}}
\]
and
\[
A''
{\colon\thinspace}=
[-4+\epsilon, 5-\epsilon]_t 
\times
[\epsilon,\infty)_s
\times
(-\infty, {\frac 1 3} + \epsilon]_{x_{m+1}}
\times
(-\infty,\epsilon]_{x_{m+2}}.
\]
We let $N_{\epsilon/2}(A)$ be its open ${\epsilon/2}$-neighborhood. 

Then we let 
\[
f{\colon\thinspace} 
\R^4 = 
\R_t \times \R_s \times \R_{x_{m+1}} \times \R_{x_{m+2}}
\to
\R
\]
be a function such that
\[
f^{-1}(1) = A,
\qquad
f^{-1}(0) = \R^4 \backslash N_{\epsilon/2}(A),
\qquad
\Crit(f) = f^{-1}(1) \cup f^{-1}(0).
\]
We then declare 
\[
\sW_{[{\frac 1 3}, {\frac 2 3}]} 
\subset
(\sW_{{\frac 2 3}} \times T^*_{\R_{x_{m+1}}} \R_{x_{m+1}})
+
\Gamma_{-\log(1-f)}
\]
to be the portion living above $[{\frac 1 3}, {\frac 2 3}] \subset \R_{x_{m+1}}$. It is easily seen that this cobordism is collared at $x_{m+1} = {\frac 1 3}$; we denote by $\sW_{{\frac 1 3}}$ the brane which collars $\sW_{[{\frac 1 3},{\frac 2 3}]}$ there.

\begin{remark}[An explanation of $A$]
The interval $[4+\epsilon, 5-\epsilon]_t$ is chosen in the definition of $A$ so that the resulting carving by $f$ is still collared by $\sL$ after time $t=5$, and by $\sX$ before time $t=-4$. 

The interval $(-\infty, {\frac 1 3} + \epsilon]_s$ is chosen so that $\sW_{[{\frac 1 3},{\frac 2 3}]}$ is indeed collared by $\sW_{{\frac 2 3}}$ for $s >> {\frac 1 3}$, and by the Lagrangian $\sW_{{\frac 1 3}}$ along $s \leq {\frac 1 3}$. 

\begin{figure}
$$
\xy
\xyimport(8,8)(0,0){\includegraphics[height=3in]{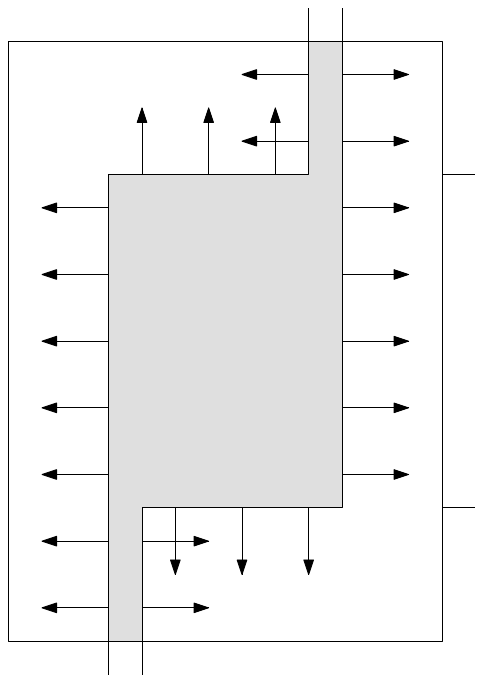}}
	 ,(5.2,8)*+!R{1-\epsilon}
	 ,(5.6,8)*+!L{1}
	 ,(1.8,0)*+!R{0}
	 ,(2.2,0)*+!L{\epsilon}
	 ,(8,2)*+!U{\epsilon}
	 ,(8.3,6)*+!U{2-\epsilon}
\endxy
$$
	\begin{image}\label{figure: sW 1/3}
	The portion of $\sW_{{\frac 1 3}}$ living over the the time interval
	$
	(-4 + \epsilon, 5-\epsilon) \subset \R_t.
	$
	Depicted is the projection onto $\R_s \times \R_{x_{m+2}}$. The horizontal direction is the $s$ direction, and the vertical direction is the $x_{m+2}$ direction.
	\end{image}
\end{figure}

Finally, depicted in Figure~\ref{figure: sW 1/3} is the portion of $\sW_{{\frac 1 3}}$ living over the time interval
\[
(-4 + \epsilon, 5-\epsilon) \subset \R_t.
\]
Note that for $x_{m+2} < \epsilon$, $\sW_{{\frac 1 3}}$ is collared by (a tilt of) the morphism $\sX \sL_0 \sL_1$, and for $x_{m+2}> 2 - \epsilon$, $\sW_{{\frac 1 3}}$ is collared by (a tilt of) $\sX \zerolag \sL_1$.
\end{remark}

\subsubsection{Construction of $\sW_{[0,{\frac {1} {3}}]}$}
Assume the following lemma:

\begin{lemma}\label{lemma: 0 to 1/3}
There is a Lagrangian cobordism from $\sW_{{\frac 1 3}}$ to (a tilt of) $\sX_{\sD}$.
\end{lemma}

We let this cobordism be $\sW_{[0,{\frac 1 3}]}$, living over the interval
\[
[0,{\frac 1 3}] \subset \R_{x_{m+1}}
\]
and directed in such a way that $\sX_{\sD}$ collars $\sW_{[0,{\frac 1 3}]}$ at $x_{m+1}=0$. We are finished by declaring
\[
\sW
{\colon\thinspace}=
\sW_{[0,{\frac {1} {3}}]} \cup \sW_{[{\frac {1} {3}},{\frac {2} {3}}]} \cup \sW_{[{\frac {2} {3}},1]}.
\]

\begin{proof}[Proof of Lemma~\ref{lemma: 0 to 1/3}]
Let $U \subset \R_{s} \times \R_{x_{m+2}}$ be the region depicted in Figure~\ref{figure: sW 1/3}. Explicitly, $U$ is the union of three rectangles:
\[
U
=
(1-\epsilon,1)_s \times [2-\epsilon,\infty)_{x_{m+2}}
\coprod
(0,1)_s \times (\epsilon,2-\epsilon)_{x_{m+2}}
\coprod
(0,\epsilon)_s \times (-\infty,\epsilon]_{x_{m+2}}.
\]
Fix an orientation-preserving diffeomorphism
\[
\psi{\colon\thinspace} (0,1)_s
\to
(0,\epsilon)_s
\]
and choose a smooth isotopy
\[
\mathscr{S}_r{\colon\thinspace}
[0,{\frac 1 3}]_r \times (0,1)_s \times \R_{x_{m+2}}
\to
(0,1)_s \times \R_{x_{m+2}}
\]
satisfying the following properties:
\begin{enumerate}
	\item
	$\mathscr{S}_0 = id_{(0,1) \times \R}$
	\item
	$\mathscr{S}_{{\frac 1 3}} ( (0,1) \times \R) = U$
	\item
	For all $s \in (0,1)$, $\mathscr{S}_{{\frac 1 3}}(s,x) ={\colon\thinspace} \gamma_s(x)$ is a smooth curve such that
	\begin{enumerate}
		\item
		$\gamma_s(x) = (\psi(s),x)$ for $x \not \in [\epsilon, 2- \epsilon]$, and
		\item
		writing the tangent vector to $\gamma_s(x)$ as 
		\[
		{\frac {\partial} {\partial x}} \gamma_s(x) = a \partial_s + b \partial_{x_{m+2}} \in T_{\gamma_s(x)}( (0,1) \times \R),
		\] 
		we require that one has $a, b \geq 0$ and $a + b > 0$ for all $x$.
	\end{enumerate}
\end{enumerate}
The point of the isotopy is that as $r$ moves from $0$ to ${\frac 1 3}$, a vertical line in the $x_{m+2}$ direction is isotoped into a curve which points primarily in the $s$ direction. This is the essence of the $\sW$ as explained in Remark~\ref{remark: the essence of W}.

One can clearly choose such an isotopy $\mathscr{S}_r$ such that the isotopy takes the tilt
\[
(\sX_{\sD} \times \Gamma_{\log b}) \cap \pi_t^{-1}(-4,5)
\]
at $r=0$ to the Lagrangian
\[
\sW_{{\frac 1 3}} \cap \pi_t^{-1}(-4,5)
\]
at $r={\frac 1 3}$. Here, $b{\colon\thinspace} \R_s \to \R$ is a tilting function, and $\pi_t$ is the projection to the zero section in the time direction:
\[
\pi_t{\colon\thinspace} M \times T^*\R_{s} \times T^*\R_t \times T^*\R^{m}_x \times T^*\R_{x_{m+2}} \to \R_t.
\]
We are also imagining that $\sX_{\sD}$, as it moves from 0 to 2 in the $x_{m+2}$ direction, is evolving from the cobordism $\sX \sL_0 \sL_1$ to the cobordism $\sX \zerolag \sL_1$.

By Example~\ref{example: isotopy of domain} one can realize the isotopy $\mathscr{S}$ as a cobordism evolving in the $x_{m+1}$ direction---we have replaced by $x_{m+1}$ the coordinate $r$. All that remains, then, is to extend the cobordism associated to $\mathscr{S}_{x_{m+1}}$, which lives above the time interval $(-4,5) \subset \R_t$, to a cobordism living over the whole of $\R_t$. This is done by first constructing a non-smooth Lagrangian $\sV$, then creating a smooth Lagrangian $\sW_{[0,{\frac 1 3}]}$ by pushing the non-smooth portions to infinity via carving. That is, we will define
\[
\sW_{[0,{\frac 1 3}]} 
{\colon\thinspace}=
\sV + \Gamma_{f_S}
\]
for a suitable function $f_S$ and some singular Lagrangian $\sV$.

The non-smooth Lagrangian $\sV$ is defined as follows. Let
\[
\sV_{(-\infty,-4]} 
{\colon\thinspace}=
 \pi_t^{-1}(-\infty,-4] \bigcap 
 	\left( 
 	\sS(\sX_{\sD}) \times T_{(0,{\frac 1 3})}^*(0,{\frac 1 3})_{x_{m+1}} \times T^*_{\R_{x_{m+2}}} \R_{x_{m+2}}
	\right)
\]
\[
\sV_{[5,\infty)} 
{\colon\thinspace}=
 \pi_t^{-1}[5,\infty) \bigcap 
 	\left( 
	\sS(\sX_{\sD}) \times T_{(0,{\frac 1 3})}^*(0,{\frac 1 3})_{x_{m+1}} \times T^*_{\R_{x_{m+2}}} \R_{x_{m+2}}
	\right)
\]
and let $\sV_{[-4,5]}$ be the cobordism associated to $\mathscr{S}_{x_{m+1}}$. Clearly
\[
\sV{\colon\thinspace}= 
\sV_{(-\infty,-4]} 
\cup
\sV_{[-4,5]} 
\cup
\sV_{[5,\infty)} 
\]
is singular above the times $t=-4$ and $t=5$. However, note that at $x_{m+1}=0$, $\sV$ is collared by the tilt
\[
\sX_{\sD} \times \Gamma_{\log b}
\]
as desired.

We fix the non-smoothness of $\sV$ via carving by a function $f_S$, constructed as follows.

Consider the set
\begin{eqnarray}
S'
&= &
\bigcup_{x_{m+1}}
\{x_{m+1}\}\times \mathscr{S}_{x_{m+1}}( (0,1) \times \R) \nonumber \\
&\subset & 
\R_{x_{m+1}} \times \R_s \times \R_{x_{m+2}}. \nonumber
\end{eqnarray}
$S'$ is the projection of the graph of an isotopy. Setting
\[
S = S' \times [-4,5]_t \subset \R^4
\]
we can choose a function
\[
f_S {\colon\thinspace} \R^4 \to [0,1]
\]
such that
\[
f_S^{-1}(1) = \overline{\R^4 \backslash S},
\qquad
f_S^{-1}(0,1) = S \cap N_\epsilon( \partial S),
\qquad
\Crit(f_S) = f_S^{-1}(1) \cup f_S^{-1}(0).
\]
Here, $\overline{\R^4 \backslash S}$ is the closure of $\R^4 \backslash S$, and $N_\epsilon(\partial S)$ denotes the $\epsilon$-neighborhood of $\partial S$. The resulting Lagrangian
\[
\sW_{[0,{\frac 1 3}]} 
{\colon\thinspace}=
\sV + \Gamma_{f_S}
\]
is the cobordism we seek.
\end{proof}

This concludes the construction of $\sW$.
\end{construction}

\begin{remark}[Description of $\sW$.]
By construction, $\phi^* \sW$ is indeed collared along $x_{m+1}=0$ by a tilt of the diagram $\sX_{\sD}$, and is collared at the vertex
\[
(0,1) \in \R_{x_{m+1}} \times \R_{x_{m+2}}
\]
by the cobordism $\sX \sK \sL_1$.
\end{remark}

\begin{remark}
Each part of the decomposition
\[
\sW = \sW_{[0,{\frac {1} {3}}]} \cup \sW_{[{\frac {1} {3}},{\frac {2} {3}}]} \cup \sW_{[{\frac {2} {3}},1]}
\]
can be described as follows. We began by isotoping $\sX \sK \sL_1$ to $\sS(\sX_{\sD})$ along the interval $[{\frac 2 3},1] \subset \R_{x_{m+1}}$. Along the interval $[{\frac 1 3},{\frac 2 3}]$ we turned $\sS(\sX_{\sD})$ into a form where we could begin to change the $s$-evolution of $\sS(\sX_{\sD})$ into an evolution along the $x_{m+2}$ direction. We then `turned' the $s$ evolution into an $x_{m+2}$ evolution via an isotopy along the interval $[0,{\frac 1 3}]$.
\end{remark}

\subsubsection{Cobordisms over triangles~\ref{other triangle 1}, \ref{other triangle 2}}
In Construction~\ref{construction: sW} we constructed a cobordism 
\[
\sW' = \phi^*(\sW) \subset M \times T^*\R_s \times T^*\R_t \times T^*(\R^{m}_x \times T)
\]
living over the triangle $T$. We will now construct cobordisms 
\[
\sW_1  \subset M \times T^*\R_s \times T^*\R_t \times T^*(\R^{m}_x \times T_1)
\]
\[
\sW_2  \subset M \times T^*\R_s \times T^*\R_t \times T^*(\R^{m}_x \times T_2)
\]
living over the triangles \ref{other triangle 1} and \ref{other triangle 2}, respectively, extending the cobordism $\sW'$.
The two are similar in nature, and we will focus on the latter and leave the former to the reader.

\begin{construction}
Let $C \subset \R_{x_{m+1}} \times \R_{x_{m+2}}$ be the interior of the triangle with vertices
\[
v_0 = (0,0),  v_1 = (1,0),  v_2 = (1,1) \in \R_{x_{m+1}} \times \R_{x_{m+2}}
\]
and let $d_i C$ be the edge in $\partial C = \overline{C} \backslash C$ opposite the vertex $v_i$. We choose an orientation-preserving diffeomorphism
\[
\phi{\colon\thinspace} C \to (0,1)_{y_1} \times (0,1)_{y_2} \subset \R_{y_1} \times \R_{y_2} ={\colon\thinspace} \R_y^2
\]
satisfying the following:
\begin{enumerate}
\item
There is some open neighborhood $U \subset \R_{x_{m+1}} \times \R_{x_{m+2}}$ of the edge $d_2 C$ such that 
\[
\phi(C \cap U) = (0,\epsilon)_{y_1} \times (0,1)_{y_2}
\]
for some $\epsilon$, and 
\[
\lim_{x \to ({\frac 1 3},{\frac 1 3})} \phi(x) = (0, {\frac {1} {2}}) \in \R_y^2.
\]
\item
There is some open neighborhood $V \subset \R_{x_{m+1}} \times \R_{x_{m+2}}$ of the union $d_0 C \cup d_1 C$ such that
\[
\phi(C \cap V) = (1- \epsilon, 1)_{y_1} \times (0,1)_{y_2}
\]
and
\[
\lim_{x \to (1,0)} \phi(x) = (1, {\frac {1} {2}}) \in \R_y^2.
\]
\end{enumerate}
We will construct a Lagrangian cobordism
\[
\sY \subset M \times T^*(\R_s \times \R_t \times \R_x^m \times \R^2_y)
\]
and define
\[
\phi^* \sY \subset M \times T^*(\R_s \times \R_t \times \R_x^{m+2})
\]
to be the portion of $F(\sX_\sD^\cone)$ living above $C$.

First, note that $\sY$ must be collared by certain cobordisms $\sY_i$ along the edges $y_1 = i$ for $i=0,1$:

Along $y_1 = 0$, $\sY_0$ begins as the cobordism $\sX \sL_0 \sL_1$, and (a) as we move from $y_2 = 0$ toward $y_2 = {\frac {1} {2}}$, it is isotoped to the portion of $\sW_{{\frac 1 3}}$ living over the vertex $({\frac 1 3},{\frac 1 3}) \in \R_{x_{m+1}} \times \R_{x_{m+2}}$. (b) At $y_2 = {\frac {1} {2}}$, $\sY$ is carved into $\sW_{{\frac 2 3}}$, and (c) is isotoped thereafter to the cobordism $\sX \sK \sL_1$, which collars $(0,1) \in \R_y^2$.

The description along $y_1=1$ is analogous: (a) At $(1,0)$, $\sY_1$ begins as $\sX \sL_0 \sL_1$, then (b) at $y_2 = {\frac {1} {2}}$, $\sY$ is carved into the Lagrangian $\sX \sK \sL_0 \sL_1 = \sP \circ \sK \circ \sX$ from Construction~\ref{construction: H} and (c) it is then isotoped until $\sY$ equals $\sX \sK \sL_1$ at $(1,1) \in \R_y^2$. 

We have described the cobordisms $\sY_i$ which collar the edges $y_1 = i$. The key observation is

\begin{lemma}\label{lemma: Y}
There exists an isotopy of domain in $\R_s \times \R_t$ taking $\sY_0$ to $\sY_1$.
\end{lemma}

We parametrize this isotopy by the interval $[0,1] \subset \R_{y+1}$, and we let $\sY$ be the Lagrangian cobordism corresponding to this isotopy. This concludes the construction of $F(\sX_\sD^\cone)$.

\begin{proof}[Proof of Lemma~\ref{lemma: Y}]
Let $A_i \subset \R_s \times \R_t$ be the domain of the Lagrangian collaring $\sY_i$ at $y_2 = {\frac {1} {2}} + \epsilon$, and let $A_i '$ be the domain of the Lagrangian collaring $Y_i$ at $y_2 = {\frac {1} {2}} - \epsilon$. Then $A_i' \backslash A_i \subset \R_s \times \R_t$ is the region carved out of $Y_i$ at $y_2 = {\frac {1} {2}}$. 

There is an obvious isotopy taking $A_0'$ to $A_1'$ which restricts to an isotopy from $A_0$ to $A_1$; this induces the isotopy of domain taking $\sY_0$ to $\sY_1$.
\end{proof}

\end{construction}

\subsubsection{Compatibility with simplicial maps}
Given an $m$-simplex $\sX_\sD{\colon\thinspace}\Delta^m\to \Lag_\sD$, we have constructed $F(\sX_\sD^\cone)$. Note that all the constructions depended only on the coordinates $s, t, x_{m+1},$ and $x_{m+2}$, and did not depend on any spatial coordinate $x_j$ for $j\leq m$. In other words, the constructions restrict to the faces of $\sX_\sD^\cone$---note that the dependence on the time coordinate is remedied by the fact that all morphisms are identified up to time reparametrization. This completes the proof of Proposition~\ref{prop. {\kernelfunctor^i}}.

\subsection{Proof of Proposition~\ref{prop. I}}\label{sect: I}
Now that we have constructed the functor $\kernelfunctor$ as described in Proposition~\ref{prop. {\kernelfunctor^i}}, 
we turn to the proof of Proposition~\ref{prop. I}.

Without loss of generality,  we may assume that $i=1$ and study the functor $F^1$. To reduce clutter, let us introduce the notation 
\[
\sX_\sD = F^1(*) \in \Lag_\sD^{\dd 2}.
\]
Then we seek to show
 the natural morphism
\[
\xymatrix{
{\kernelfunctor^1}(\sK_\sD(\sP) \to *) \in \Lag_\sD^{\dd 2}(F^i(\sK_\sD(\sP)), \sX_\sD)
}
\]
is invertible. Here $\sK_\sD(\sP) \to *$ is the unique simplex from $\sK_\sD$ to $*$ in the cone category.

Our proof once again will be constructive: we will find a cobordism 
\[
\sI \subset M \times T^*\R_s \times T^*\R_t \times T^*\R^{m+3}_x 
\]
with prescribed boundary values. To approach this, following the pattern of Figure~\ref{figure: {\kernelfunctor^i} projection},
let us project onto the final three space coordinates
\[
\xymatrix{
M \times T^*\R_s \times T^*\R_t \times T^*\R^{m+2}_x \ar[r] &  \R_{x_{m+1}} \times \R_{x_{m+2}} \times \R_{x_{m+3}}.
}
\]
Then the cobordism $\sI$ will live over the following labelled cube
\[
\xymatrix{
& \sX \sX 0 \sL_1 \ar@{-}[rr] \ar@{-}[dd] && \sX \sK \sX 0 \sL_1 \ar@{-}[dd]
\\
\sX 0 \sL_1 \ar@{-}[rr]\ar@{-}[dd] \ar@{-}[ur] && \sX \sK 0 \sL_1 \ar@{-}[dd] \ar@{-}[ur] & 
\\
& \sX \sX \sL_1 \ar@{-}[rr] \ar@{-}[dd] && \sX \sK \sX \sL_1 \ar@{-}[dd] 
\\
\sX \sL_1\ar@{-}[ur]  \ar@{-}[dd] \ar@{-}[rr] && \sX \sK \sL_1 \ar@{-}[dd]\ar@{-}[ur]  & 
\\
& \sX \sX \sL_0 \sL_1 \ar@{-}[rr] && \sX \sK \sX \sL_0 \sL_1 
\\
\sX \sL_0 \sL_1 \ar@{-}[rr] \ar@{-}[ur] && \sX \sK \sL_0 \sL_1 \ar@{-}[ur] &
}
\]

We view this cube as a parameterized version of the rectangle of Figure~\ref{figure: {\kernelfunctor^i} projection}.
Namely, the new $x_{m+3}$-direction points into the paper so that along the face $x_{m+3} = 0$ we recover
precisely Figure~\ref{figure: {\kernelfunctor^i} projection}. The face is labelled to indicate that the cobordism $\sI$ will be collared
by $\sX_\sD = F^1(*)$ along it. Along the lefthand face given by $x_{m+1}= 0$, the cobordism $\sI$ will be collared by
the degenerate cobordism obtained by setting the morphism $\sX \to \sX$ to be the identity of $\sX$.

\subsubsection{Construction of $\sI|_{x_{m+3} = 1}$}\label{construction: x_3=1}
The point is to construct a diagram
\[
\xymatrix{
\sX \ar[dr]_{id_\sX} \ar[r]^{\kernelfunctor^1_1} & \sK \ar[d]^{\on{Rot}} \\
 & \sX
}
\]
at the level of objects in $\Lag$ (not at the level of objects in the over-category $\Lag_D$.) The map $\on{Rot}$ indicates the isotopy of domain described in Remark~\ref{remark: stabilizing kernels}, rotating the $s_2$ axis into the $s_1$ axis. We follow the construction as depicted in Figure~\ref{figure: Ix=3}.

\begin{figure}
$$
\xy
\xyimport(8,8)(0,0){\includegraphics[width=3in]{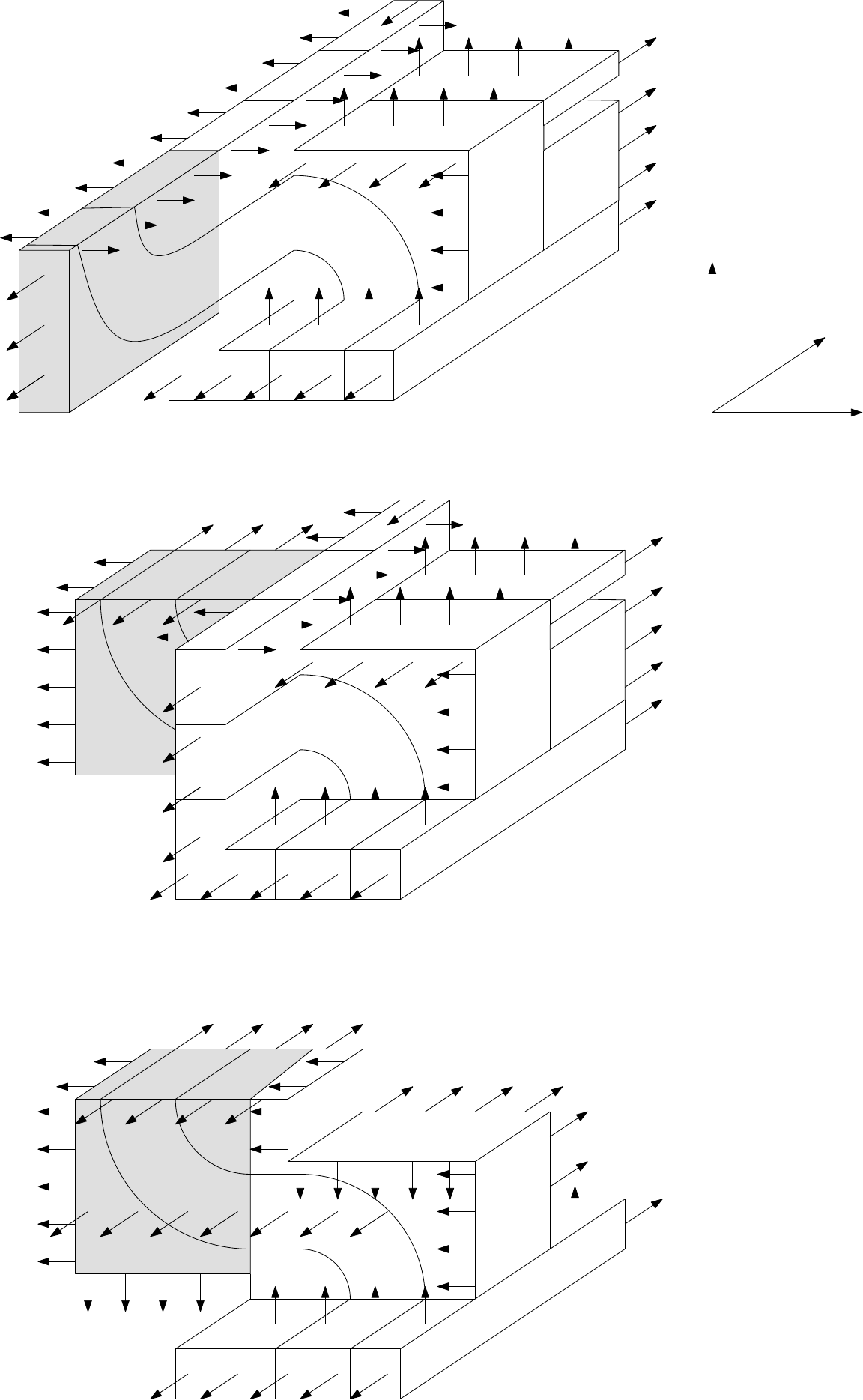}}
	 ,(-0.5,6.2)*!R+{(a)}
	 ,(-0.5,3.5)*!R+{(b)}
	 ,(-0.5,1)*!R+{(c)}
	 ,(6.8,6.5)*+{t}
	 ,(8,5.5)*+{s_1}
	 ,(7.9,6.2)*+{s_2}
\endxy
$$
	\begin{image}\label{figure: Ix=3}
	A depiction of the face $\sI|_{x_{m+3}=1}$. (a) is the morphism $\sX \to \sK$ given by $\kernelfunctor^1$. (b) is obtained by an isotopy of domain in the $s_1$-$s_2$-plane. The shaded portion indicates the part of the cobordism affected by the isotopy. Note the isotopy is possible because along $s_1=-1$ the morphism $\sX \to \sK$ is collared by $\sP$. We parametrize this isotopy by the coordinate $x_{m+2}$, though this is not indicated in the figure. We carve the cobordism from (b) to obtain the cobordism in (c). Clearly the morphism in (c) is isotopic to the identity morphism of $\sX$. Note that this figure is not to scale.
	\end{image}
\end{figure}

(a) Examine the morphism $\sX \to \sK$ given by $\kernelfunctor^1_1$; this is given by a Lagrangian cobordism living in $M \times T^*\R_t \times T^*\R^2_s$. Its projection onto $T^*_{\R^2_s}\R^2_s$ is the union of two rectangles:
\[
\pi(\kernelfunctor_1^1) = R_1 \cup R_2,
\qquad
R_1 = 
(-2,-1)_{s_1} \times (-2,0]_{s_2},
\qquad
R_2 = 
(-2,1)_{s_1} \times (0,1)_{s_2}.
\] 
The part of $\kernelfunctor_1^1$ living over $R_2$ encodes the data of the map $\sX \to \sD$, and the portion over $R_1$ is contained in Figure~\ref{figure: map to kernel}---it is the difference between part (b) and part (c) of the figure. It is also the shaded region in Figure~\ref{figure: Ix=3}(a).

(b) The main observation is that, since $\kernelfunctor^1_1$ is collared along the interval $s_1 = -2$ by the morphism $\sP$, we can perform an isotopy of domain $\phi$ moving $R_1 \cup R_2$ to a new domain
\[
R_1 ' \cup R_2
\]
where the set $R_1'$ is given by isotoping $R_1$ to the domain
\[
R_1 '
=
(-4,-2)_{s_1} \times ({\frac {1} {2}}-\epsilon,{\frac {1} {2}}+\epsilon)_{s_2}.
\]
We correspondingly induce an isotopy on the rotation morphism $\sK \to \sX$ (this is not depicted in the figure).

(c) We now carve (b) to obtain yield the cobordism described in Figure~\ref{figure: Ix=3}(c). Note that the carving yields a cobordism which, along the region of $\sh$ allowing a non-characteristic cut, approaches $-\infty$ in the $dt$ direction. Moreover, straightening the snake-like path that the morphism $\sP$ takes within the cobordism of (c), it is clear that we can isotope (c) into the identity morphism from $\sX$ to itself. This completes the construction of $\sI_{x_{m+3} = 1}$.

\subsubsection{Construction of $\sI$}
The key observation is that the cobordism from $\sX \to \sK$ to $\sD$ has two geometric copies of the diagram $\sX_\sD$ inside it; the construction of $\sI$ is simply an interpolation between the two. Specifically, there is a copy of $\sX_\sD$ in the $x_{m+2}$ direction, given by $\sD$ itself, and there is a second copy in the $s_2$ direction, as depicted in Figure~\ref{figure: map to kernel}(b). This was part of the definition of the morphism $\kernelfunctor^1_1{\colon\thinspace} \sX \to \sK$. 

The only difference between the faces we have constructed at $x_{m+3} = 0$ and $x_{m+3} = 1$ is that different copies of $\sD$ are isotoped from the $x_{m+1} = 1$ face to the $x_{m+1} = 0$ face. More explicitly, along the face $x_{m+3} = 0$, the cobordism from Construction~\ref{construction: x_3=1} isotopes the copy of $\sD$ in the $s_2$ direction to re-orient $\sD$ into the $x_{m+2}$ direction, and carves out the copy of $\sD$ originally living in the $x_{m+2}$ direction. Along the face $x_{m+3}=1$, we have done just the opposite: We carved out the copy of $\sD$ living in the $s_2$ direction, keeping only the copy living in the $x_{m+2}$ direction. 

(1) We first fill in the portion of $\sI$ along the interval $[{\frac 2 3},1]_{x_{m+2}}$. The simple observation is that $\sW_{{\frac 2 3}}$ and the cobordism from Figure~\ref{figure: Ix=3}(b) are both obtained by performing an isotopy on the basic morphism $\sX\sK \sL_1$. Further we can demand that the face $\sI_{x_{m+2} =1}$ be collared by the isotopy which isotopes $\sX$ into $\sK$ by performing a rotation on the coordinates $s_1$ and $s_2$. So we have three isotopies collaring three edges---$[{\frac 2 3},1]_{x_{m+2}} \times \{0\}_{x_{m+3}}$, $[{\frac 2 3},1]_{x_{m+2}} \times \{1\}_{x_{m+3}}$, and $\{1\}_{x_{m+2}} \times [0,1]_{x_{m+3}}$.  Clearly we can find a 2-simplex in the space of all isotopies of domain, parametrizing between these three isotopies. Realizing this 2-simplex as a cobordism, we have filled in the portion of $\sI$ living above the rectangle
\[
[{\frac 2 3},1]_{x_{m+2}} \times [0,1]_{x_{m+3}} \subset \R_{x_{m+2}} \times \R_{x_{m+3}}.
\]

(2) Next we observe that there is an isotopy $\psi$ of $\R_{s_2} \times \R_{x_{m+2}}$ which we depict in Figure~\ref{figure: I}. The salient feature is that it takes the domain of $\sW_{{\frac 1 3}}$ as described in Figure~\ref{figure: sW 1/3} and transports it to the rectangle in $\R_{s_2} \times \R_{x_{m+2}}$ given by
\[
({\frac {1} {2}} - \epsilon, {\frac {1} {2}} + \epsilon) \times \R_{x_{m+2}}.
\]
We parametrize this isotopy $\psi$ by the coordinate $x_{m+3}$. Note that the function $f$ (which we used to carve out $\sW_{{\frac 1 3}}$ in Section~\ref{sect: sW1323}) can be isotoped along $\psi$. Let us call this new function $\tilde f$ and carve the cobordism from (1) by $f$, realizing this carving as taking place along the interval $[{\frac 1 3},{\frac 2 3}]_{x_{m+2}}$. Then at $x_{m+2} = {\frac 1 3}$, we have a cobordism such that at $x_{m+3} = 1$ we have the cobordism from Figure~\ref{figure: Ix=3}(c), and at $x_{m+3}=0$ we have carved out the cobordism $\sW_{{\frac 1 3}}$ from Section~\ref{sect: sW1323}. 

\begin{figure}
$$
\xy
\xyimport(8,8)(0,0){\includegraphics[height=3in]{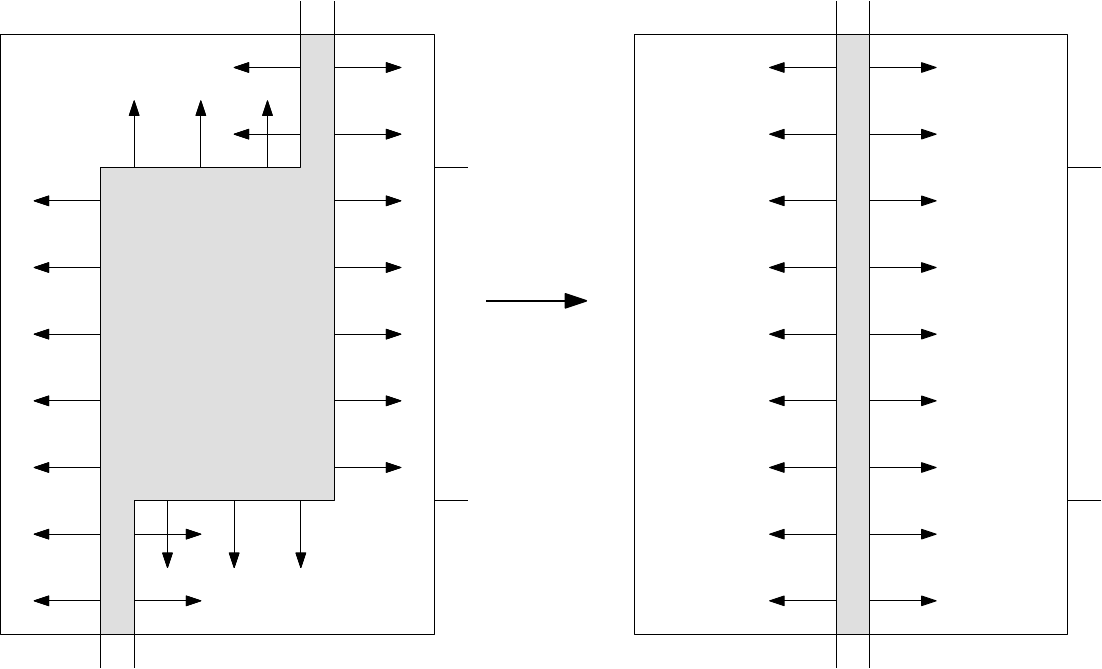}}
	 ,(2.2,8)*+!R{1-\epsilon}
	 ,(2.4,8)*+!L{1}
	 ,(0.7,0)*+!R{0}
	 ,(1,0)*+!L{\epsilon}
	 ,(8,2)*+!U{\epsilon}
	 ,(8.3,6)*+!U{2-\epsilon}
	 ,(6,8)*+!R{{\frac {1} {2}}-\epsilon}
	 ,(6.3,8)*+!L{{\frac {1} {2}}+\epsilon}
\endxy
$$
	\begin{image}\label{figure: I}
	The isotopy $\psi$ described in the construction of $\sI$. The horizontal coordinate is $s_2$ and the vertical coordinate is $x_{m+2}$. The indicated asymptotics are a result of the carving function $\tilde f$ which we induce using $\psi$.
	\end{image}
\end{figure}

(3) Now we can use the same trick we used in step (1). Note that the cobordisms given by $\sI|_{x_{m+3} =1}, \sI|_{x_{m+3}=0}$ are merely isotopies of domain along the interval $[0,{\frac 1 3}]_{x_{m+2}}$. We can again find a 2-simplex in the space of isotopies to fill in the three isotopies at play. (Namely, these are the isotopies realized along the edges $[0,{\frac 1 3}]_{x_{m+2}} \times \{1\}_{x_{m+3}}, [0,{\frac 1 3}]_{x_{m+2}} \times \{0\}_{x_{m+2}},$ and $\{{\frac 1 3}\}_{x_{m+2}} \times [0,1]_{x_{m+3}}$.) Realizing this 2-simplex as an actual cobordism living above the rectangle $[0,{\frac 1 3}]_{x_{m+2}} \times [0,1]_{x_{m+3}}$ and demanding that along $\{0\}_{x_{m+2}} \times [0,1]_{x_{m+3}}$ it be the identity, we have our desired $\sI$.

\clearpage
\section{The loop functor}\label{sect. loop}
Finally, with the preceding construction of kernels in hand, and with Theorem~\ref{stablealternative} in mind, we discuss properties of the {loop functor} of $\Lag$.

\subsection{Definition of loop functor}

Let $\sC$ be an $\oo$-category with a zero object $0\in \sC$.

Define the simplicial set $\sC^\Omega$ to consist of all maps 
$$
\xymatrix{
\Delta^1 \times \Delta^1 \ar[r] & \sC
}
$$
whose image is a limit diagram (Cartesian square) of the form 
$$\xymatrix{
	y \ar[r]\ar[d]	&	0\ar[d]	\\
	0	\ar[r]	&	x
	}$$

The obvious evaluations give a pair of maps $y, x{\colon\thinspace}\sC^\Omega \to \sC$, and if such Cartesian squares always exist, then $x{\colon\thinspace}\sC^\Omega \to \sC$ is a trivial fibration, and hence  there is a contractible space of sections
	$s{\colon\thinspace}\sC \to \sC^\Omega.$
	In this case,  the {\em loop functor} $\Omega{\colon\thinspace}\sC\to \sC$ is defined to be the composition 
	\[
	\xymatrix{
	\Omega{\colon\thinspace} \sC \ar[r]^-{s} & \sC^\Omega \ar[r]^-y & \sC.
	}
	\]
Similarly, if colimit diagrams (cocartesian squares) of the form
	$$\xymatrix{
	x	\ar[r]	\ar[d] &0\ar[d]		\\
	0	\ar[r] & y
	}$$
	always exist, then 
one can define the {\em suspension functor} $\Sigma{\colon\thinspace}\sC\to \sC$.

\subsection{Loop functor for \texorpdfstring{$\Lag$}{Lag}}
Let us begin by describing  the action of the loop functor of $\Lag$ on objects.
Then we will formally calculate its action on all simplices.

Fix an object $\sL \in \Lag$. By definition, the object $\Omega \sL \in \Lag$ is the kernel of the zero map $0{\colon\thinspace}\zerolag \to \sL$.
By construction, this is a cobordism $\Omega \sL \subset M \times T^*\R$ given by the product of $\sL$ with a graph.
The graph is Hamiltonian isotopic to the conormal to a point, and hence $\Omega \sL$ is equivalent 
 {\em up to a shift of grading} with the stabilization $\sL^{\dd 1}$.
 Deciding what the grading is comes down to comparing the gradings on the graph and the conormal to a point.
We refer to Figure~\ref{fig. gradings} to see that they will differ by one.

\begin{figure}\label{fig. gradings}
$$\xy
	\xyimport(15,23)(0,0){\includegraphics[height=3in]{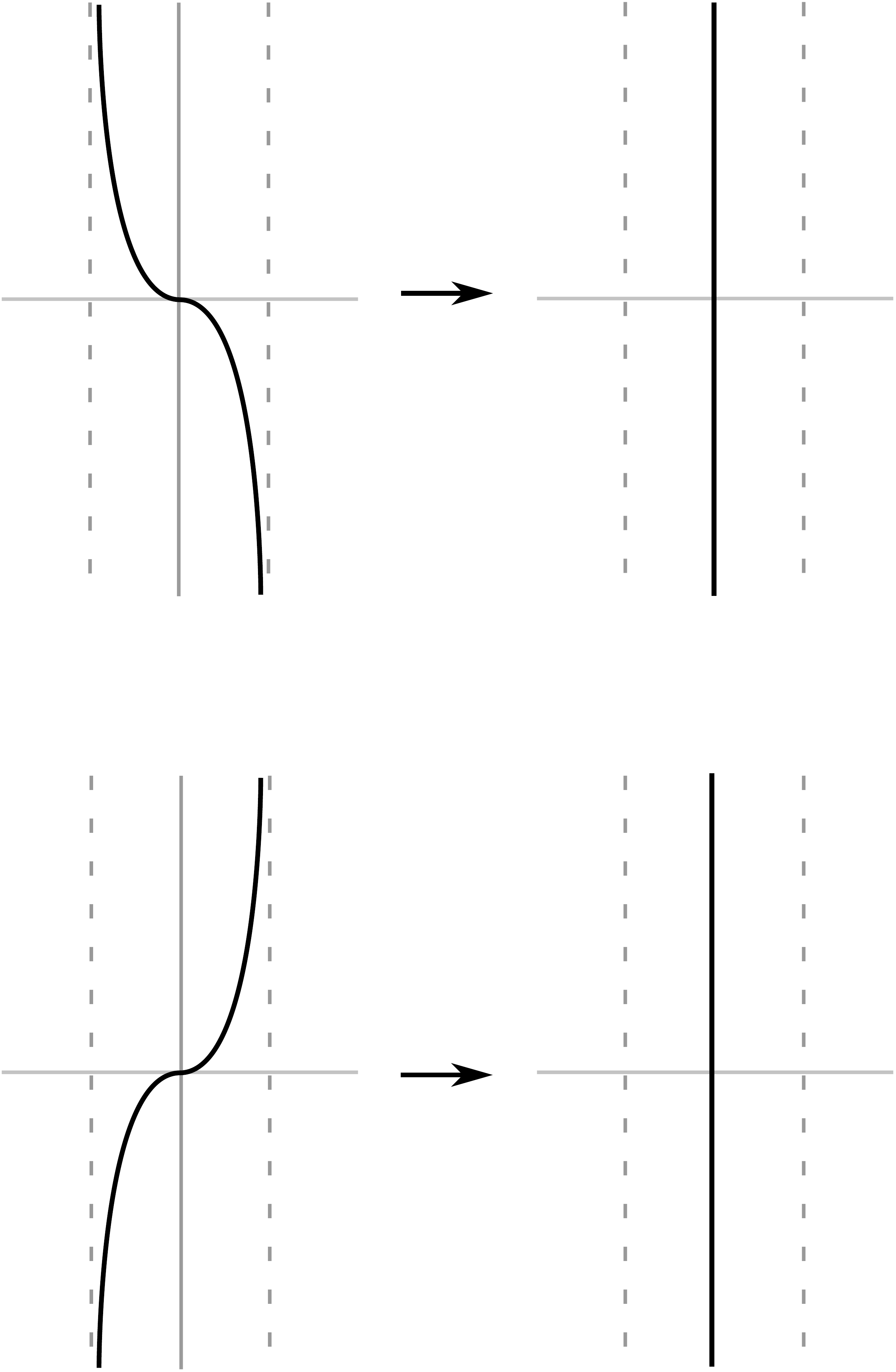}}
	,(0,5)*+!R{X}
	,(3,5)*+!L!U{\alpha}
	,(13,5)*+!L!U{\alpha+{\frac {1} {2}}}
	,(0,18)*+!R{\Omega X}
	,(3,18)*+!L!D{\alpha}
	,(13,18)*+!L!D{\alpha-{\frac {1} {2}}}
	\endxy$$
	\begin{image}\label{fig:loops and tilts} A comparison of $\Omega \sL$ and $\sL$.  At the origin, both $\Omega \sL$ and the tilt of $\sL$ have grading equal to $\alpha$.  Isotoping both $\Omega \sL$ and the tilt of $\sL$ to the product of $\sL$ with the conormal to a point, we see that the two objects are isomorphic up to a grading shift of one.\end{image}
\end{figure}

\begin{proposition}\label{prop:loop functor}
Fix an object $\sL\in \Lag$.
The kernel of the zero morphism $\zerolag \to \sL$, or in other words, the  limit to the diagram
	$$\xymatrix{
		&	\zerolag	\ar[d]	\\
	\zerolag	\ar[r]	&	\sL
	}$$
is the brane $\sL$ with grading shifted downward by one.  

More generally, given a $k$-simplex $\sP{\colon\thinspace}\Delta^k\to \Lag$ thought of as a brane, its looping $\Omega \sP{\colon\thinspace}\Delta^k \to \Lag$ is 
again the brane $\sP$  with grading shifted downward by one.  
\end{proposition}

\begin{proof}
Set $M' = \Lambda' = \sL' = pt$, and consider the zero map $\zerolag \to \sL'$. 
Recall the corresponding  cobordism $\sh' \subset T^*\R^2$ constructed as part of the kernel diagram.

Returning to the given $M$ and $\Lambda$, consider a $k$-simplex $\sP{\colon\thinspace}\Delta^k \to \Lag_\Lambda(M)$ representing
a cobordism $\sQ \subset M \times T^*\R \times T^* \R^k$.  Then the product $\sQ \times \sh' \subset M \times T^*\R \times T^* \R^k \times T^*\R^2$ is again a cobordism giving a $(k+2)$-simplex of $\Lag$.  

It is easily seen that $\sQ \times \sh'$ gives a simplicial section $\Lag\to \Lag^\Omega$.

Finally, observe that  the product with $\sh'$ gives a cobordism from the original $k$-simplex $\sP$ to a new $k$-simplex $\Omega \sP$ whose underlying brane has grading  shifted downward by one.  
\end{proof}

\begin{remark}
One can similarly prove that the suspension functor is given by a shift of grading upward by  one.
\end{remark}

\begin{remark}
The reader may wonder if the grading shift arose simply from grading conventions.  But the grading conventions are necessary---for the cobordism $\sH$ from the composition $\sP \circ \sk$ to the zero map to restrict to the appropriate gradings on its boundaries, a grading shift of one needed to appear.
\end{remark}

Note that shifting grading is clearly an autoequivalence. Hence we  can now conclude with our main result:

\begin{proof}[Proof of  Theorem~\ref{thm. stable}.]
Recall Theorem~\ref{stablealternative}). We note that its three hypotheses follow from Theorem~\ref{thm. zero} ($\Lag$ has a zero object), Theorem~\ref{thm. kernels exist} ($\Lag$ admits kernels), and Proposition~\ref{prop:loop functor} (the loop functor of $\Lag$ is invertible).
\end{proof}

\clearpage
\section{Appendix on \texorpdfstring{$\infty$}{oo}-categories}

Here we review $\infty$-categories, in particular collecting the results we need for our present work. In this appendix only, we will use the term {\em category} to mean a strict category -- that is, a category in the sense of Mac Lane~\cite{maclane} for example.  Outside the appendix, by a {\em category}, we mean an {\em $\infty$-category}.

\subsubsection{Background}
The theory of $\infty$-categories was developed by Joyal (a recent reference is \cite{Joyal}), who coined them {\em quasi-categories}. Boardman and Vogt \cite{BV73}, the first to name the idea, coined them {\em weak Kan complexes}. The notion of $\infty$-category is a direct generalization of ordinary categories, allowing spaces of morphisms between objects. Moreover, the notion of a functor between quasi-categories, in which homotopies between morphisms must be preserved, is easily defined.

There are many extant models for capturing the idea of an $(\infty,1)$-category, including complete Segal spaces and simplicial categories, all of which are equivalent. (See \cite{Bergner}.) 
We choose to use the language of quasi-categories since a formidable toolkit based on quasi-categories has been developed, as most notably evidenced in Lurie's \cite{LurieTopos} and \cite{LurieHigherAlgebra}. Moreover, since every aspect of the theory is based on the well-studied theory of simplicial sets, deep homotopical statements can be reduced to proofs requiring concrete objectives. The model is formal enough to capture the structures we take interest in, and down-to-earth enough that we can prove concrete, applicable results.

As the main references for quasi-categories are quite large, we hope that this appendix will save the reader some effort in digging through the archives. Nothing collected in the appendix is original, and we learned most of this material from~\cite{LurieTopos}. (This is also our source for the term {\em $\infty$-category}.)

\subsection{Simplicial sets, horns, quasi-categories}
Let $\Delta$ be the category whose objects are finite, ordered, non-empty sets, and whose morphisms are order-preserving maps. Every object in $\Delta$ is isomorphic to the object $[n] = \{0,1,\ldots,n\}$ with the standard ordering.
 
As usual the injective morphism $d_i{\colon\thinspace} [n] \into [n+1]$ skipping the $i$th element of $[n+1]$ is called the {\em $i$th face map}, and the surjective morphism $s_i{\colon\thinspace} [n+1] \to [n]$ whose fiber above $[i] \in [n]$ is the set $\{i,i+1\} \subset [n+1]$ is called the {\em $i$th degeneracy map}.  Every morphism in $\Delta$ is a composition of face and degeneracy maps.

\begin{definition}
 A {\em simplicial set} is a functor $X{\colon\thinspace} \Delta^{op} \to \Sets$.  A morphism between simplicial sets is a natural transformation.  The set $X([i])$ is called the {\em set of $i$-simplices of $X$}.
\end{definition}

\begin{example}
 The {\em $n$-simplex}, written $\Delta^n$, is the representable functor given by $\Delta^n([i]) = \Hom_\Delta([i],[n]).$  Given an $n$-simplex $\Delta^n \to X$, the {\em $j$th face} is the $(n-1)$-simplex opposite the $j$th vertex.
\end{example}

\begin{example}[The nerve of a category]
 Given any small category $Z$, one can define a simplicial set called the {\em nerve of $Z$}.  The 0-simplices are the objects of $Z$, the 1-simplices are the morphisms of $Z$, and the $k$-simplices are diagrams of the form
\[
 \xymatrix{
 Z_0 \ar[r]^{f_{0}} & Z_1 \ar[r] & \ldots \ar[r]^{f_{k-1}} &Z_k.
 }
\]
For $0<i<k$, the $i$th face map is given by composing $f_i$ and $f_{i-1}$.  The other face maps `forget' the $0$th or $(k-1)$st morphism.  For all $0 \leq i \leq k$, the $i$th degeneracy map is given by introducing the morphism $id_{Z_i}{\colon\thinspace} Z_i \to Z_i$.
\end{example}

\begin{example}
For $0 \leq i \leq n$ we let $\Lambda_i^n$ be the simplicial set obtained by deleting the face opposite the $i$th vertex of the $n$-simplex.  We call $\Lambda_i^n$ the {\em $i$th horn of $\Delta^n$}.  When $0<i<n$, we call $\Lambda_i^n$ an {\em inner horn}.
\end{example}

\begin{definition}
Let $\Lambda_i^n \to \sC$ be a map of simplicial sets.  If there exists a map $\Delta^n \to \sC$ making the following diagram commute
	$$\xymatrix{
	\Lambda_i^n	\ar[r]		\ar[d]	&\sC		\\
	\Delta^n	\ar[ur]	&	
	}$$
we say that the horn $\Lambda_i^n \to \sC$ can be {\em filled}.
\end{definition}

\begin{definition}\label{def. weak Kan}
If for all $n\geq0$, any horn $\Lambda_i^n \to \sC$ can be filled, we say that $\sC$ is a {\em Kan complex}.
If any horn can be filled for $0<i<n$, we say that $\sC$ is a {\em weak Kan complex}, an {\em $\infty$-category}, and a {\em quasi-category}. (The three terms are synonymous, and all in common usage.)
\end{definition}

\begin{example}
If $\sC$ is the nerve of a small category, we can fill $\Lambda_1^2$ because we can compose any two functions.  To be able to fill $\Lambda_0^2$ or $\Lambda_2^2$ means we can find inverses to any function.  Intuitively, the difference between a Kan complex and a weak Kan complex is that the former is a groupoid, while the latter may have non-invertible morphisms.
\end{example}

\begin{example}
Given any topological space $X$, one can define the simplicial set of {\em singular chains on $X$}, $\Sing(X)_\bullet$, whose $n$-simplices are continuous maps of the standard $n$-simplex into $X$. This is a Kan complex.
\end{example}

\begin{example}
 One can show that, if the horn-fillings in Definition~\ref{def. weak Kan} are unique for every inner horn, then $\sC$ is the nerve of a category.  Hence one already sees that a weak Kan complex (where horn-fillings are not necessarily unique) describes a flimsier, more general notion of ``category," in which compositions are not unique.
\end{example}

\begin{definition}
 Let $\sC$ be an $\infty$-category.  Then we call the 0-simplices of $\sC$ the {\em vertices}, or {\em objects}, of $\sC$.  We also call the 1-simplices the {\em edges}, or the {\em morphisms}, of $\sC$.  We call $k$-simplices {\em $k$-morphisms.}
 
A {\em functor} between two quasi-categories is simply a map of simplicial sets.
\end{definition}

\begin{example}
Let $\sC$ be the nerve of an ordinary category, and let $\mathsf{Spaces}$ be the $\infty$-category of topological spaces. Then a functor from $\sC$ to $\mathsf{Spaces}$ takes commutative diagrams in $\sC$ to {\em homotopy commutative} diagrams in $\mathsf{Spaces}$.
\end{example}

\subsection{Mapping spaces} For a quasi-category $\sC$, there are several equivalent models for the space $\Hom_\sC(X,Y)$, and we use the following.

\begin{definition}
 Given two objects $X,Y$ in $\sC$, let $\Hom_\sC(X,Y)$ be the simplicial set whose $k$-simplices are maps $\Delta^{k+1} \to \sC$ such that $(k+1)$st face of $\Delta^{k+1}$ is sent to the object $X$, and such that the $(k+1)$st vertex is sent to the object $Y$.  (The face and degeneracy maps are defined in the natural way.)  We call this the {\em mapping space from $X$ to $Y$.}
\end{definition}

It is not hard to justify the terminology mapping {\em space}:

\begin{proposition}
 If $\sC$ is a quasi-category, then for all objects $X, Y \in \sC$, $\Hom_\sC(X,Y)$ is a Kan complex.
\end{proposition}

\begin{proof}
See Proposition 1.2.2.3 of \cite{LurieTopos}.
\end{proof}

There is another model for the mapping space, given by labeling vertices dually. That is, we can define a Kan complex $\Hom^L_\sC(X,Y)$ to be simplicial set whose $k$-simplices are maps $\Delta^{k+1} \to \sC$ where the 0th vertex is $X$, but whose 0th face is the degenerate $n$-simplex at $Y$. The choice of models is inconsequential:

\begin{proposition}
$\Hom_{\sC}(X,Y)$ is homotopy equivalent to $\Hom^L_\sC(X,Y)$.
\end{proposition}

\begin{proof}
See Remark 1.2.2.5 and Corollary 4.2.1.8 in \cite{LurieTopos}.
\end{proof}

\subsection{Limits and terminal objects}
Now we discuss the notion of (homotopy) limits and colimits in the language of $\infty$-categories. The homotopical flexibility is built into the $\infty$-category framework.

\begin{definition}\label{def. join}
Let $A$ and $B$ be two simplicial sets.  The {\em join from $A$ to $B$}, denoted $A \star B$, is the simplicial set defined as follows.  The set of $n$-simplices is given by 
\[
 (A \star B)[n] = A[n] \bigcup B[n] \bigcup_{i+j = n-1} A[i]\times B[j].
\]
The face and degeneracy maps for an $n$-simplex $(a,b) \in A[i] \times B[j]$ are given by
	\[
		d_k (a,b) = \begin{cases}
		(d_k a, b) & k \leq i \\
		(a, d_{k-i-1}b) & k > i.
		\end{cases}
		\qquad
		s_k (a,b) = \begin{cases}
		(s_k a, b) & k \leq i \\
		(a, s_{k-i-1}b) & k > i.
		\end{cases}
	\]
When $a$ (or $b$) is a zero-simplex, the corresponding boundary map merely forgets $a$ (or $b$). 
\end{definition}

Note that if $A$ is a point, then $A \star B$ is the cone over $B$.  So if $D{\colon\thinspace} B \to \sC$ is a diagram in $\sC$, a map $A \star B \to \sC$ restricting to $D$ along $B$ yields an object living over the diagram $D$.  Similarly, a map $\Delta^1 \star B \to \sC$ restricting to $D$ over $B$ would be a map between two objects living over $D$.  This prompts us to define

\begin{definition}
Let $D{\colon\thinspace} B \to \sC$ be a diagram in $\sC$, and let $\sSet_D( A \star B , \sC)$ be the simplicial set of all maps $\phi{\colon\thinspace} A \star B \to \sC$ such that $\phi|_B = D$.  We define a simplicial set $\sC_D$ by the identity
\[
 \sSet(\Delta^n, \sC_D) = \sSet_D(\Delta^n \star B, \sC)
\]
with the obvious face and degeneracy maps.  We call $\sC_D$ the {\em $\infty$-category of objects living over $D$}, or the {\em over-category} when the context is clear.
\end{definition}

\begin{definition}
A {\em terminal object $X$} in a quasi-category  $\sC$ is a vertex $X \in \sC$ such that the forgetful map $\sC_{X} \to \sC$ from the overcategory to $\sC$ is a trivial Kan fibration.  That is, if the right lifting problem
	$$\xymatrix{
		\partial \Delta^k	\ar[r]	\ar[d]	&	\sC_{X}	\ar[d]	\\
		\Delta^k		\ar[r]	\ar@{-->}[ur]&	\sC
	}$$
can always be solved.
\end{definition}

\begin{remark}
Let us restrict our attention to maps $\partial \Delta^k \to \sC_X$ such that the composition $\partial \Delta^k \to \sC_X \to \sC$ is constant, with image $X'$.  Then the condition of being a trivial Kan fibration guarantees that the mapping space $\Hom_\sC(X',X)$ is contractible. The following assures the converse:
\end{remark}

\begin{proposition}\label{prop. terminal objects}
If $\sC$ is a quasi-category, $X \in \sC$ is a terminal object if and only if $\Hom_\sC(X',X)$ is contractible for all objects $X' \in \sC$.
\end{proposition}
\begin{proof}
See Proposition 1.2.12.4 of \cite{LurieTopos}.
\end{proof}

\begin{definition}
Fix a diagram $D{\colon\thinspace} B \to \sC$.  We say that $X \in \Lag_D$ is a {\em limit} for $D$ if it is a terminal object in $\Lag_D$.
\end{definition}

\begin{remark}
In ordinary category theory, terminal objects are unique up to unique isomorphism.  In a quasi-category, any two terminal objects are unique up to a contractible choice of isomorphisms.
\end{remark}

\begin{example}[Kernels]
We define a {\em zero object} of $\sC$ is an object 0 which is both initial and terminal.  Then a {\em kernel} to a map $P{\colon\thinspace} L_0 \to L_1$ is a limit to the diagram $D$, which we depict as
\[
	\xymatrix{ 	&	L_0	\ar[d]^P		\\
				0	\ar[r]	&	L_1.
	}
\]
\end{example}

\subsection{Dual notions}
To speak of initial objects, and colimits, all we need is to define the opposite of a quasi-category.  In general, we can define the opposite of a simplicial set.

\begin{definition}
Let $\sC$ be a simplicial set.  Then the {\em opposite} of $\sC$, written $\sC^{op}$, is given on objects by $\sC^{op}([n]) = \sC([n])$.  On morphisms,
\[
 d_i{\colon\thinspace} \sC^{op}([n]) \to \sC^{op}([n-1])
\]
is given by
\[
 d_{n-i} {\colon\thinspace} \sC([n]) \to \sC([n-1])
\]
and likewise for the degeneracy maps $s_i$.
\end{definition}

\begin{example}
When $\sC = \Delta^1$, with non degenerate edge $X$ and vertices $d_0(X) = X_0$, $d_1(X) = X_1$, $\sC^{op}$ has the same simplices, but $d_0(X) = X_1$, and $d_1(X) = X_0$.  This shows that, in general (as one would expect) the directions of the morphisms are reversed.
\end{example}

\begin{example}
Given a 2-simplex
\[
\xymatrix{
 B \ar[r] & C \\
 A \ar[u] \ar[ur] &
}
\]
in $\sC$, the corresponding 2-simplex in $\sC^{op}$ can be drawn as
\[
\xymatrix{
 B \ar@{<-}[r] & C \\
 A \ar@{<-}[u] \ar@{<-}[ur] &
}
\]
which simply asserts the existence of a homotopy from the composite $C \to B \to A$ to $C \to A$.
\end{example}

\begin{definition}
Given an $\infty$-category $\sC$, an {\em initial object} $X$ is a terminal object of $\sC^{op}$.  Similarly, given a diagram $D$, a {\em colimit} to $D$ is a limit to $D^{op}$ in $\sC^{op}$.   
\end{definition}

\subsection{The simplicial nerve}\label{appendix. nerve}
This section is not necessary for reading the present work. However, the idea of the simplicial nerve inspired the definition of our quasi-category, so we include it here.

Given any fibrant simplicial category $\cC$ (a category enriched over Kan complexes) one can form a quasi-category $\sC$.  This is done by taking the {\em simplicial nerve}. This in fact gives an equivalence between the model category of quasi-categories and simplicial categories. (See \cite{Bergner}.)  One should consider the simplicial nerve and the original simplicial category equivalent.  For instance, one can prove:

\begin{proposition}
Let $\sC$ be the simplicial nerve of $\cC$. Then for all objects $X,Y$, the space $\Hom_{\sC}(X,Y)$ is homotopy equivalent to $\cC(X,Y)$.
\end{proposition}

In this section we discuss the simplicial nerve construction in some detail. We hope this may motivate the factorization conditions placed on Lagrangian cobordisms, which were inspired by this construction.

\begin{remark}
This section closely follows Lurie ~\cite{LurieTopos}.  In other literature, what we call the simplicial nerve is also called the {\em homotopy coherent nerve}. It is also a special case of Boardman and Vogt's W construction for simplicial operads.
\end{remark}

Let $\sSet$ denote the category of simplicial sets and $\sCat$ that of simplicial categories
(an object of $\sCat$ is a category enriched over simplicial sets).
There is an adjunction
$$
\xymatrix{
\fC{\colon\thinspace}\sSet \ar@<0.25em>[r] & \ar@<0.25em>[l] \sCat{\colon\thinspace}N
}
$$
where $\fC$ is the rigidification functor and $N$ is the simplicial nerve or coherent nerve (see also \cite{DuggerSpivak}). 

If $\cC$ is a fibrant simplicial category (a simplicial category whose morphism simplicial sets are Kan complexes), then the simplicial 
nerve $N(\cC)$ is an $\oo$-category. 

We will not describe the rigidification $\fC$ in general, but just enough to recall the construction of the simplicial nerve $N$.

\begin{definition}
Given integers $0\leq i \leq j$, let $P_{ij}$ be the partially ordered set whose elements are subsets of $[i,j] \subset \mathbb{N}$ containing $i$ and $j$ and whose ordering is given by inclusion.  Viewing this partially ordered set as a category, we let $N(P_{ij})$ denote the nerve of $P_{ij}$.  
\end{definition}

\begin{example} Here are some small examples of $P_{ij}$.

\begin{enumerate}
\item If $i=j$ or $i+1 = j$, then $P_{ij}$ is a category with the single object $\{i\}$ or $\{i, i+1\}$, and no non-identity morphisms.  

\item  If $i+2 = j$, then $P_{ij}$ is a category with two objects $\{i, i+2\}$, $\{i, i+1, i+2\}$ and a single non-identity morphism
$$
\xymatrix{
\{i, i+2\} \ar[r] & \{i, i+1, i+2\}.
}
$$

\item
If $i+3=j$, then $P_{ij}$ can be depicted as two glued triangles
\[
 \xymatrix{
 \{i,i+3\} \ar[r] \ar[d] \ar[rd]
 & \{i,i+1,i+3\} \ar[d]
 \\ \{i,i+2,i+3\} \ar[r] 
 & \{i,i+1,i+2,i+3\}. 
 }
\]
\end{enumerate}
\end{example}

\begin{remark} 
In general for $i < j$, the category $P_{ij}$ can be depicted as a simplicial decomposition of a $k$-cube where $k = i-j- 1$.  The $2^k$ vertices of the cube are in bijection with the ways to choose  a subset of $[i,j] \subset \Z$ containing $i$ and $j$.
\end{remark}

Let $\Delta^{n}$ denote the $n$-simplex as a representable simplicial set.
One defines the simplicial category $\fC[\Delta^n]$ to consist of $n+1$ objects labeled $0,1,\ldots,n$, with morphisms
$$
\fC[\Delta^n] (i,j) = P_{ij}, 
\quad\mbox{ for $i\leq j$},
$$ 
and no morphisms from $i$ to $j$, for $i>j$.
Composition is induced by unions of sets where we regard an element of $P_{ij}$ as a subset of $\Z$.  

\begin{remark}
Here we are thinking of the category $P_{ij}$ as a simplicial set by taking its nerve.
\end{remark}

\begin{definition}[Simplicial Nerve]
Given a simplicial category $\cC$, its {\em simplicial nerve $\sC = N(\cC)$} is the simplicial set defined by the property
\[
 \Hom_{\sSet}(\Delta^n,\sC) = \Hom_{\sCat}(\mathfrak{C}[\Delta^n],\cC).
\]

In other words, an $n$-simplex in $\sC$  is a functor of simplicial categories $\mathfrak{C}[\Delta^n] \to \cC$.
\end{definition}

For any arbitrary simplicial category this gives an arbitrary simplicial set. It is not hard to verify the following. (See also Proposition~1.1.5.10 of~\cite{LurieTopos}.)

\begin{proposition}
When $\cC$ is a fibrant simplicial category, $\sC$ is a quasi-category.
\end{proposition}

 Here are some examples of low-dimensional simplices in the simplicial nerve $\sC= N(\cC)$.

\begin{example} [$0$-simplices and $1$-simplices]
The $0$-simplices (or vertices) of $\sC$ are the objects of $\cC$.  The $1$-simplices in $\sC$ are the $1$-morphisms in $\cC$.  
\end{example}

\begin{example}[$2$-simplices]
A $2$-simplex in $\sC$ is the following data: three objects $x_0, x_1, x_2$, three $0$-morphisms $f_{01}{\colon\thinspace}x_0 \to x_1$, $f_{12}{\colon\thinspace}x_1 \to x_2$, and $f_{02}{\colon\thinspace} x_0 \to x_2$, and a $1$-morphism $f_{02}\implies f_{12} \circ f_{01}$.  
We will alternatively depict this either as the segment
$$
\xymatrix{
\{0, 2\} \ar[r] & \{0, 1, 2\}
}
$$
inside of $\cC(x_0, x_2)$ given by the above data,
or as the solid rectangle of compositions and homotopies the segment parameterizes:
\[
 \xymatrix{
 x_2 \ar@{-}[r]
 & x_2
 \\ 
 &x_1 \ar[u]_{f_{12}}
 \\ x_0  \ar@{-}[r] \ar[uu]^{f_{02}}
 &x_0 \ar[u]_{f_{01}}
 }
\]
\end{example}

\begin{example}[$3$-simplices]
 Similarly, a $3$-simplex is a choice of objects $x_i$, for $0 \leq i\leq 3$, along with morphisms $f_{ij}$,  for $0 \leq i< j \leq 3$,
 and coherent homotopies among their compositions. 
 Again, we will  depict this in two alternative fashions. First, we can draw the triangulated square 
 inside of $\cC(x_0, x_3)$ given by the above data

\[
 \xymatrix{
\{0, 3\} \ar[r] \ar[d] \ar[dr]
& \{0, 2, 3\} \ar[d] 
\\ \{0, 1, 3 \} \ar[r]
& \{ 0, 1, 2, 3 \}
}
\]
Here each vertex represents a composition of morphisms (for instance, the vertex $\{0, 2, 3\}$ represents the composition $f_{23} \circ f_{02}$), and the higher simplices represent homotopies between the compositions. Alternatively,
we can extend this diagram vertically and draw the compositions it parameterizes
\[\xymatrix{
		\\x_3 	\ar@{-}[rdd]
				\ar@{-}[rrrrdd] 
				&&&&
		\\&&&&
		\\&x_3 \ar@{-}[rrr] &&&x_3
		\\&&&&
		\\&&&&x_2 \ar[uu]_{f_{23}}
		\\&&&&
		\\x_0 \ar@{-}[rdd] \ar@{-}[rrrrdd] \ar[uuuuuu]^{f_{03}} &x_1 \ar[uuuu]^{f_{13}} \ar@{-}[rrr] &&&x_1 \ar[uu]_{f_{12}}
		\\&&&&
		\\&x_0 \ar@{-}[rrr] \ar[uu]^{f_{01}} &&&x_0 \ar[uu]_{f_{01}}
	}
	\xymatrix{
			&x_3 \ar[rrrddd] \ar@{<-}[dd]^{f_{23}}&&&
		\\x_3 	\ar@{-}[ru] 
				\ar@{-}[rrrrdd] 
				\ar@{<-}[dddddd]^{f_{03}} 
				&&&&
		\\&x_2 \ar@{<-}[dddd]^{f_{02}} \ar@{-}[rrrddd] &&&
		\\&&&&x_3 \ar@{<-}[dd]^{f_{23}}
		\\&&&&
		\\&&&&x_2 \ar@{<-}[dd]^{f_{12}}
		\\&x_0 \ar@{-}[rrrddd]&&&
		\\x_0 \ar@{-}[ru] \ar@{-}[rrrrdd]&&&&x_1 \ar@{<-}[dd]^{f_{01}}
		\\&&&&
		\\&&&&x_0
		\\&&&&
	}
\]
For clarity, we have drawn the above picture as two disjoint prisms, but one should glue them together above the common rectangle living over the edge $\{0, 3\} \to \{0,1,2,3\}$.

One should  note that there are some rectangles in the picture whose horizontal edges are drawn without arrowheads; these correspond to the trivial homotopy between the morphisms along their vertical edges.
\end{example}

\subsection{Definitions of Stability}
In this paper we have cited two definitions of a stable $\infty$-category. Here we briefly recall why the two are equivalent. As there are two possible sources to cite (Lurie's \cite{LurieStab} and Chapter One of his \cite{LurieHigherAlgebra}), we will cite where the necessary statements are contained in both works. (The latter is a significant revision of the former.) Note also that Lurie writes (co)kernel in his earlier work, but in \cite{LurieHigherAlgebra} he writes (co)fiber for the same notions.

\begin{definition}[Definition 2.9 \cite{LurieStab}, Definition 1.1.1.9 \cite{LurieHigherAlgebra}]
 An $\oo$-category $\cC$ is said to be {\em stable} if it satisfies the following:
\begin{enumerate}
\item There exists a zero object $0\in \cC$.
\item Every morphism in $\cC$ admits a kernel and a cokernel.
\item Any commutative diagram in $\cC$ of the form
$$
\xymatrix{
a\ar[d] \ar[r] & b \ar[d]\\
0 \ar[r] & c
}
$$
is Cartesian if and only if it is cocartesian.
\end{enumerate}
\end{definition}
The following yields a second definition of a stable $\infty$-category, and we reproduce it here for convenience.
\begin{theorem}[Corollary 8.28, \cite{LurieStab}. Corollary 1.4.2.20, \cite{LurieHigherAlgebra}]
Let $\cC$ be an $\infty$-category with a zero object. The following are equivalent:
\begin{enumerate}
\item $\cC$ is stable.
\item $\cC$ admits finite colimits and the suspension functor $\Sigma{\colon\thinspace} \cC \to \cC$ is an equivalence.
\item $\cC$ admits finite limits and the loop functor $\Omega{\colon\thinspace} \cC \to \cC$ is an equivalence. 
\end{enumerate}
\end{theorem}
As Lurie points out in the remark following Theorem 3.11 of \cite{LurieStab} (which we cite as Theorem~\ref{stablealternative} in the present work) the existence of a cokernel and the shift functor gives enough `additive' structure to our category to guarantee all finite colimits. 

For instance, to show $\cC$ admits finite colimits, it suffices to show that $\cC$ admits coequalizers and finite coproducts. A coequalizer for a pair of maps $f, g {\colon\thinspace} X \to Y$ is simply a cokernel to the map $f - g$, and a coproduct of $X$ and $Y$ can be written as the cokernel of the zero map $\Sigma^{-1}X \to Y$. (The map $f-g$ can be defined so long as cokernels exist, and the suspension functor is an equivalence.)

\end{document}